\documentclass[12pt]{article}

\usepackage[margin=1.20in]{geometry}

\usepackage{amssymb,amsmath}
\usepackage{amsthm,enumitem}
\usepackage{hyperref}
\usepackage{mathrsfs}
\usepackage{textcomp}

\usepackage[abbrev,nobysame]{amsrefs}
\hypersetup{
    colorlinks,%
    citecolor=black,%
    filecolor=black,%
    linkcolor=black,%
    urlcolor=black
}

\usepackage{verbatim}

\usepackage{sectsty}
\usepackage[nocompress]{cite}

\makeatletter

\newdimen\bibspace
\makeatother

\makeatletter

\newtheorem{Theorem}{Theorem}[section]
\newtheorem{Lemma}[Theorem]{Lemma}
\newtheorem{Proposition}[Theorem]{Proposition}
\newtheorem*{Assumption*}{Assumption (H)}
\newtheorem{Definition}[Theorem]{Definition}

\newtheorem{Corollary}[Theorem]{Corollary}
\newtheorem{Remark}[Theorem]{Remark}

\newtheorem{question}[Theorem]{Open question}

\def\XXint#1#2#3{{\setbox0=\hbox{$#1{#2#3}{\int}$}
  \vcenter{\hbox{$#2#3$}}\kern-.5\wd0}}

                \newcommand{\lda}{\lambda}
                
           \newcommand{\ud}{\mathrm{d}}
\newcommand{\be}{\begin{equation}}      \newcommand{\ee}{\end{equation}}
                 
\newcommand{\Lda}{\Lambda}              
\newcommand{\T}{\mathcal{T}}
\newcommand{\R}{\mathbb{R}}              
\newcommand{\D}{\mathcal{D}}
\newcommand{\I}{\mathcal{I}}
\newcommand{\E}{\mathcal{E}}
\newcommand{\J}{\mathcal{J}}
\newcommand{\G}{\mathcal{G}}
\newcommand{\M}{\mathscr{M}}

\begin{document}

\title{\textbf{Regularity and classification of the free boundary for a Monge-Amp\`ere obstacle problem}\bigskip}

\author{\medskip  Tianling Jin,\footnote{T. Jin was partially supported by NSFC grant 12122120, and Hong Kong RGC grants GRF 16303822 and GRF 16303624.}\quad Xushan Tu, \quad
Jingang Xiong\footnote{J. Xiong was partially supported by NSFC grants 12325104 and 12271028.}}

\date{\today}

\maketitle

\begin{abstract} 
We study convex solutions to the Monge-Amp\`ere obstacle problem
\[
\operatorname{det} D^2 v=g v^q\chi_{\{v>0\}},  \quad v \geq 0, 
\]
where $q \in [0,n)$ is a constant and $g$ is a bounded positive function. This problem emerges from the $L_p$ Minkowski problem. We establish $C^{1, \alpha}$ regularity for the strictly convex part of the free boundary $\partial\{v=0\}$. Furthermore, when $g \in C^{\alpha}$, we prove a Schauder-type estimate. As a consequence, when $g\equiv 1$, we obtain a Liouville theorem for entire solutions with unbounded coincidence sets $\{v=0\}$. Combined with existing results, this provides a complete classification of entire solutions for the case $q=0$.

\medskip


\medskip


\end{abstract}

\section{Introduction}

Let $n\ge 2$ and $\Omega \subset \mathbb{R}^n$ be a bounded convex open set. In \cite{savin2005obstacle}, Savin studied convex solutions to
\begin{equation}\label{eq:obs problem savin}
\det  D^2 v = g \chi_{\{v > 0\}} \ \text{ in } \Omega, \quad  v> 0\ \text{ on } \partial\Omega,
\end{equation}
where $g$ is a positive bounded function on $\Omega$, and $\chi_E$ denotes the characteristic function of a set $E$. Let $K=\{x\in\Omega: v(x)=0\}$ be the coincidence set. When $|K|=0$, the problem \eqref{eq:obs problem savin} reduces to the Monge-Amp\`ere equation $\det  D^2 v = g$, where the regularity of solutions was establish by Caffarelli \cite{caffarelli1990ilocalization,caffarelli1990interiorw2p,caffarelli1991regularity}.  Where  $|K|>0$, this is a Monge-Amp\`ere  obstacle problem, which can be interpreted as a model in optimal transportation with a Dirac measure and also as the dual problem for the Monge-Amp\`ere equation with an isolated singularity. Savin \cite{savin2005obstacle} proved that the free boundary $\partial K$ is $C^{1,\alpha}$, and furthermore, when $g\equiv 1$, the free boundary is $C^{1,1}$ and  uniformly  convex. In two dimensions, G\'alvez-Jim\'enez-Mira \cite{galvez2015classification} proved that the free boundary $\partial K$ is $C^{\infty}$  and analytic, provided that $g$ is $C^{\infty}$ and analytic, respectively. 
This regularity in higher dimensions has recently been completely  solved by Huang-Tang-Wang \cite{huang2024regularity}. The regularity of free boundary problems associated with the Monge-Amp\`ere operator has also been studied in other literatures, see, e.g.,  \cite{chou1993obstacle, lee2001obstacle, daskalopoulos2012fullydegenerate, caffarelli2010freeboundaries, chen2024c2aregularity}.

In this paper, we  study  convex solutions to the obstacle problem
\begin{equation}\label{eq:obs equation}
\operatorname{det} D^2 v=g v^q\chi_{\{v>0\}} ,  \quad v\ge 0 \quad \mbox{in } \Omega,
\end{equation}
where $q \in [0,n)$ is constant, $\lambda \leq g  \leq   \Lambda$ in $\Omega$, and $\lda$ and  $\Lda$ are two positive constants. This obstacle problem emerges from the $L_p$ Minkowski problem \cite{lutwak1993brunnminkowskifirey, lutwak2004lpminkowski, chou2006lpminkowski} when $0$ is on the boundary of the convex body, which will be explained in Section \ref{sec:Minkowski}.  If $q\ge n$, then it was shown in \cite{diaz2015freeboundary} that either $v\equiv 0$ or $v>0$, that is, there is not free boundary.

Our first result is a classification of the global convex solutions to \eqref{eq:obs equation} with $ g\equiv 1$. 

\begin{Theorem}\label{thm:classify}
Let $n\ge 2$ and $0\le q<n$. Suppose $v\not\equiv 0$ is a convex solution of \begin{equation}\label{eq:global obs}
	\det D^2 v = v^q\chi_{\{v>0\}},\quad v\ge 0, \quad \text{in } \mathbb{R}^n.
\end{equation}
If the coincidence set $K:=\{v=0\}$ is unbounded, then it must  be a paraboloid, and $v$ is affine equivalent\footnote{Two functions \( u_1, u_2: \mathbb{R}^n \to \mathbb{R} \) are called \emph{affine equivalent} if there exists an \( n \times n \) matrix \( A \) with \(\det A \neq 0\) and a vector \( b = (b_1, \ldots, b_n)^T \) such that $u_1(x) = (\det A)^{-2/n} u_2(Ax + b)$ almost everywhere in \(\mathbb{R}^n\). 
} to
\begin{equation}\label{eq:glo model 2}
	\varpi(x',x_n) = c_{n,q}\max\left\{\frac{1}{2} |x'|^2 -x_n,0\right\} ^{\frac{n+1}{n-q}} \text{ for } x=(x',x_n) \in \R^n,
\end{equation}
where $c_{n,q}=\frac{(n-q)^{\frac{n+1}{n-q}}}{(n+1)^{\frac{n}{n-q}}(q+1)^{\frac{1}{n-q}}}$.
\end{Theorem}

If $q=0$, then we have the following complete classification. 

\begin{Corollary}\label{thm:classify0}
Let $n\ge 2$ and $v\not\equiv 0$ be a convex solution of
\[
\det D^2v  = \chi_{\{ v > 0\}},\quad v\ge 0,\quad  \text{in }\R^n.
\]
Let $K:=\{v=0\}$ denote the coincidence set.  Then one of the following holds:
\begin{enumerate}
\item[(i).] $K= \emptyset$ or is a single point set, and $v$ is a quadratic function;
\item[(ii).] $K$ is an ellipsoid and $v$ is  affine equivalent to $\displaystyle\int_{0}^{|x|} \max\left\{r^n-1,0\right\}^{\frac{1}{n}}dr$;
\item[(iii).] $K$ is a paraboloid and $v$ is affine equivalent to 
$
\displaystyle\frac{n^{\frac{n+1}{n}} }{(n+1) } \max\left\{\frac{1}{2} |x'|^2 -x_n,0\right\} ^{\frac{n+1}{n}}.
$
\end{enumerate} 
\end{Corollary}

Case (i) corresponds to the case $|K|=0$. The conclusion follows from the theorem of J\"orgens \cite{jorgens1954losungen}, Calabi \cite{calabi1958improper}, and Pogorelov \cite{pogorelov1978minkowski} on the classification of the global solutions to $\det D^2 v = 1$ in $\mathbb{R}^n$.

Case (ii) corresponds to the case $0<|K|<\infty$. The conclusion follows from the classification of the global solutions to $\det D^2 u = 1$ in $\mathbb{R}^n\setminus\{0\}$, which the Legendre transform of $v$ satisfies. This classification result was established by J\"orgens \cite{jorgens1955harmonische} for $n = 2$ and Jin-Xiong \cite{jin2016solutions} for $n \geq 3$.

Case (iii) corresponds to the case $|K|=\infty$. The conclusion follows from Theorem \ref{thm:classify} with $q=0$.

The classification of solutions and their coincidence sets of the global obstacle problem
\[
\Delta v =\chi_{\{ v > 0\}},\quad v\ge 0, \quad \text{in }\R^n 
\] 
has been established through the works of Dive \cite{dive1931attraction}, Lewy \cite{lewy1979inversion}, Sakai \cite{sakai1981null}, DiBenedetto-Friedman \cite{dibenedetto1986bubble} and Friedman-Sakai \cite{friedman1986characterization} if the coincidence sets are assumed to be bounded, and has recently been fully resolved by Eberle-Shahgholian-Weiss \cite{eberle2023global} for $n\geq 6$ and  Eberle-Figalli-Weiss \cite{eberle2022complete} for all $n$. Nonetheless, the classification of solutions to the global Alt-Phillips \cite{alt1986free} equation
\[
\Delta v = v^{\frac{q}{n}} \chi_{\{v > 0\}},\quad v\ge 0, \quad \text{in } \R^n
\]
remains largely unresolved; see \cite{bonorino2001regularity,wu2022fully,dipierro2023classification,savin2025stable,savin2025concentration} for some results on the classification of homogeneous solutions.  We refer to \cite{de2021certain, wu2022fully} for recent advancements and generalizations of the Alt-Phillips equation to fully nonlinear uniformly elliptic operators.

\begin{question}\label{question1}
Let $n\ge 2$. When $0<q<n$, the classification of solutions to \eqref{eq:global obs} is unknown if the coincidence sets are assumed to be bounded.
\end{question}

Our proof of Theorem \ref{thm:classify} makes use of the regularity of local solutions to the obstacle problem \eqref{eq:obs equation}, which will be discussed below.

The existence of solutions to \eqref{eq:obs equation} will be presented in Proposition \ref{prop:existence obs} using Perron's method. Let the convex set 
\[
K:= \left\{ x\in\Omega:\; v(x)=0\right\} 
\]
be the coincidence set, and $\Gamma= \Omega \cap \partial K$ be the free boundary in $\Omega$. In the scenario where $K =\emptyset$,  the regularity of strictly convex solutions follows from Caffarelli \cite{caffarelli1990ilocalization,caffarelli1990interiorw2p,caffarelli1991regularity}. So we will consider the case where $K\neq \emptyset$.  Our main focus will be on the regularity of the solution $v$, and in particular, the regularity of the free boundary $\Gamma$.

Strict convexity plays a fundamental role in the regularity theory for Monge-Amp\`ere equations. Inspired by Caffarelli \cite{caffarelli1990ilocalization,caffarelli1990interiorw2p,caffarelli1991regularity} and Savin \cite{savin2005obstacle}, we define  the strictly convex part of $\Gamma$ as
\[
\Gamma_{sc}:= \left\{x \in \Omega:\; x \text{ is an exposed point of } K  \right\}\subset \Gamma.
\]
This is a relatively open subset of $\Gamma$; see Lemma \ref{lem:exposed set}. 
We will show the $C^{1,\alpha}$ regularity of both $v$ and $\partial K$ around $\Gamma_{sc}$ in Theorems \ref{thm:c1a vK K>0} and \ref{thm:c1a v K=0} below. However, since we allow $\varphi$ to vanish somewhere on $\partial\Omega$, it is possible that $\Gamma_{sc}\neq \Gamma$ and the solution $v$ can be merely Lipschitz at $\Gamma_{nsc}:=\Gamma\setminus\Gamma_{sc}$; see the examples in Remark \ref{rem:nonstrictconvex}.  
On the other hand, there are sufficient conditions under which $\Gamma_{sc}= \Gamma$; see Proposition \ref{prop:infinite ray}.


We split the study of the regularity of $v$ and $\partial K$ around $\Gamma_{sc}$ into two cases: 

\begin{itemize}
\item[(i).] $|K|>0$, and 
\item[(ii).] $|K|=0$. If $\Gamma_{sc}\neq\emptyset$, then Lemma \ref{lem:ext are exp} will show that $K$ consists of a single point. Note that this case does not require special treatment for \eqref{eq:obs problem savin}, since when $|K|=0$, equation \eqref{eq:obs problem savin} reduces to $\det D^2 v = g$, where the regularity of its solutions follows from Caffarelli's results \cite{caffarelli1990ilocalization,caffarelli1990interiorw2p,caffarelli1991regularity}. However, this simplification does not apply to equation \eqref{eq:obs equation} when $q>0$.
\end{itemize}

We first consider the case that $|K|>0$. When $q = 0$ and $v>0$ on $\partial\Omega$, it is \eqref{eq:obs problem savin}, and the regularity theory  has been established in \cite{savin2005obstacle, galvez2015classification, huang2024regularity}. For the case of $n=2$, $0<q<2$ and $g\equiv 1$, the optimal regularity of $v$ and $\Gamma_{sc} $ was obtained in Daskalopoulos-Lee \cite{daskalopoulos2012fullydegenerate}. Here, we will establish the $C^{1,\alpha}$ regularity of $\Gamma_{sc}$, as well as the $C^{1,\alpha}$ regularity of $v$ around $\Gamma_{sc}$ for \eqref{eq:obs equation}, with $\alpha \in (0,1)$ depending only on $n$, $q$,  $\lambda$ and $\Lambda$. After an suitable affine transformation\footnote{We say a map $T:\R^n\to\R^n$ is an affine transformation if $Tx=Ax+b$ for some $n\times n$ matrix $A$ and a vector $b\in\R^n$. We say $T$ is a unimodular affine transformation if in addition $|\det A|=1$.}, we can presume, for simplicity, that the following assumption holds:

\begin{Assumption*}
Suppose that $0 \in \Gamma_{sc}$\,, $K\subset \{x=(x',x_n): x_n \geq 0\}$, $K \cap \{0 \leq x_n\leq 1\} =  \{x:\psi(x') \leq x_n \leq 1\}$ for some convex function $\psi$, and $K_1':= \left\{ x':\; (x',1)\in K\right\}$ is normalized such that
\[
B_{c(n)}'(0) \subset K_1' -x_0'\subset B_{C(n)}'(0),
\]
where $x_0'\in \R^{n-1}$ denote the mass center of $K_1'$, and $B_{r}'(0)$ denotes the ball in $\R^{n-1}$ of radius $r$ with center $0$.
\end{Assumption*}
Then we have the following regularity for the obstacle problem \eqref{eq:obs equation}.
\begin{Theorem}\label{thm:c1a vK K>0}
Let $n\ge 2$ and $0\le q<n$. Suppose $v$ is a solution to \eqref{eq:obs equation}, and Assumption (H) holds. Then, we have  $B_c'(0) \subset K_1'\subset B_C'(0)$, $B_c(0) \subset \left\{ v\leq 1\right\} \cap \left\{x_n \leq \frac{3}{4}\right\} \subset B_C(0)$, $ \left\|v\right\|_{C^{1,\alpha}(B_{c}(0))} \leq C$, and
\[
c|x'|^{\frac{1+\alpha}{\alpha}}\leq \psi\left( x' \right) \leq C|x'|^{1+\alpha},
\]
where $\alpha \in (0,1)$, $c$ and $C$ are positive constants, all of which depend only on $n$, $q$, $\lambda$ and $\Lambda$. Moreover,  $v$ is strictly convex in $B_c(0)\setminus K$ and satisfies
\begin{equation}\label{eq:v distance alpha} 
c\operatorname{dist}(x,K)^{\frac{1+\alpha}{\alpha}} \leq v \leq C\operatorname{dist}(x,K)^{1+\alpha} \quad \text{in }B_{c}(0).
\end{equation}  
\end{Theorem} 

When $g \in C^{\alpha}$, we will further show the $C^{2,\alpha}$ regularity of the free boundary and establish the corresponding behavior of the solution $v$ at the free boundary. More specifically, we have
\begin{Theorem}\label{thm:c2a origorous}
Let $n\ge 2$ and $0\le q<n$. Suppose   $v$ is a solution to
\begin{equation}\label{eq:obstacle problem}
	\det D^2 v = g(x,v,Dv) v^q\chi_{\{v>0\}},
\end{equation}
and Assumption (H) holds. Assume that $ 0<\lambda \leq g \leq \Lambda <\infty$, and $\left\|g\right\|_{C^{\alpha}} \leq \Lambda$ for $\alpha \in (0,1)$. Then, we have
\[
\left\|\psi\right\|_{C^{2,\beta} (B_c'(0))} \leq C  ,\quad 0 < c\I_{n-1}\leq D_{x'}^2 \psi \leq  C \I_{n-1}\quad  \text{in }B_c'(0),
\]
and
\begin{equation}\label{eq:v distance optimal}
c\operatorname{dist}(x,K)^{\frac{n+1}{n-q}} \leq v(x) \leq C\operatorname{dist}(x,K)^{\frac{n+1}{n-q}}   \quad \text{in }B_{c}(0) ,
\end{equation}
where $\beta=\min\left\{\frac{2(q+1)}{n-q},1\right\}\alpha$,  $c$ and $C$ are positive constants depending only on $n$, $q$, $\alpha$, $\lambda$ and $\Lambda$.
Moreover, if $g(x,v,Dv)=g(x,v)$ is independent of $Dv$, then we can take $\beta=\alpha$.   
\end{Theorem}

\begin{Remark}
In our proof of the $C^{1,1}$ regularity and the uniform convexity of the free boundary for the case $g\equiv 1$, the Pogorelov type estimate in Theorem \ref{lem:Pogorelov p>0} for the equation
\[
\operatorname{det} D^2 w  = \left(x\cdot Dw-w\right)^{-q} \quad\mbox{in }\left\{x\cdot Dw-w>0\right\}
\]
plays an essential role.
\end{Remark}

Now, let us consider the second case that $|K|=0$. As mentioned earlier, when $\Gamma_{sc}\neq\emptyset$, then $K$ is a single point set (Lemma \ref{lem:ext are exp}). 

\begin{Theorem}\label{thm:c1a v K=0}
Let $n\ge 2$ and $0< q<n$. Let $v$ be a solution to \eqref{eq:obs equation}. Suppose $K=\left\{0\right\}$, and $S_1:=\left\{ v \leq 1\right\}\subset\subset\Omega$. Then, after applying a unimodular affine transformation, $S_1$ can be normalized so that $B_{c}(0) \subset  S_1\subset B_{C}(0)$. Furthermore, after this normalization, we have  $ \left\|v\right\|_{C^{1,\alpha}\left(\left\{ v \leq 1/2\right\}\right)} \leq C$,  and
\begin{equation}\label{eq:v distance alpha 0} 
c|x|^{\frac{\alpha}{1+\alpha}} \leq v(x) \leq C|x|^{1+\alpha}  \quad \text{in }B_{c}(0),
\end{equation}   
where $\alpha \in (0,1)$, $c$ and $C$ are positive constants, all of which depend only on $n$, $q$, $\lambda$ and $\Lambda$. 
\end{Theorem}
Note that the equation \eqref{eq:obs equation} is degenerate at the single point set $K$. The optimal regularity of $v$ in the neighborhood of this point remains an open question, even in the  case where $n=2$ and $g \equiv 1$. Regularity of solutions to the two-dimensional degenerate Monge-Amp\`ere equation $\det D^2u=|x|^{\alpha}$ with $\alpha>-2$ has been studied in Daskalopoulos-Savin \cite{daskalopoulos2009monge}.

This paper is organized as follows. In Section \ref{chp:MA obs}, we analyze the structure of the free boundary, and show its $C^{1,\alpha}$ regularity by a compactness argument.  In Section \ref{sec:solutionregularity1}, we establish the $C^{1,\alpha}$ regularity of the solution, and provide the proofs of Theorems  \ref{thm:c1a vK K>0} and \ref{thm:c1a v K=0}.  In Section \ref{sec:c11}, we show the $C^{1,1}$ regularity and the uniform convexity of the free boundary for the case $g\equiv 1$. In Section \ref{sec:c2a regularity}, we investigate the $C^{2,\alpha}$ regularity using a perturbative argument,  prove Theorem \ref{thm:c2a origorous}, and establish the classifications in Theorem \ref{thm:classify} and Corollary \ref{thm:classify0}. Finally, in Section \ref{sec:Minkowski}, we discuss the applications of these results to the $L_p$ Minkowski problem.

\section{The structure and $C^{1,\alpha}$ regularity of the free boundary}\label{chp:MA obs}

For a convex function $w$ defined on an open set $U$, the subdifferential of $w$ at a point $x_0 \in U$ is defined as the convex set
\[
\partial w(x_0) = \left\{ p \in \mathbb{R}^{n} :\; w(x) \geq w(x_0) + p \cdot (x - x_0),\ \forall\  x \in U \right\},
\]
where the elements of this set are called the subgradients of $w$ at $x_0$. For any measurable set $E \subset \subset U$ (i.e., $E$ is compactly contained in $U$), we denote the union of subdifferentials at points in $E$ by
\[
\partial w(E) := \bigcup_{x \in E} \partial w(x).
\]
The Monge-Amp\`ere measure $\M w$ is then defined as
\[
\M w(E) = |\partial w(E)|\quad  \text{for each Borel set } E \subset \subset U,
\]
where $|\partial w(E)|$ represents the measure of the set $\partial w(E)$.  Given a positive Borel measure $\mu$ in $\Omega$, we say that $w$ solves the Monge-Amp\`ere equation (in the Aleksandrov sense)
\[
\det D^2w =\mu  \quad  \text{ if } \M w = \mu. 
\]
 
The following lemma is well known for convex functions and is a corollary of the Aleksandrov-Bakelman-Pucci maximum principle (see Guti\'errez \cite{gutierrez2016monge} or Figalli \cite{figalli2017monge}).
\begin{Lemma}[Aleksandrov’s Maximum Principle]\label{lem:measure constrain 2}
Let $\Omega\subset \R^n$ be a convex bounded open set. Suppose $w \in C(\overline{\Omega})$ is convex, $w=0$ on $\partial \Omega$, then
\[
w(x) \geq-C(n)\left([\operatorname{diam}(\Omega)]^{n-1}\operatorname{dist}\left(x, \partial \Omega\right) \M u(\Omega)\right)^{\frac{1}{n}},
\]
where $C(n)$ is a positive constant depending only on $n$.
\end{Lemma}	

The existence of solutions to \eqref{eq:obs equation} can be derived by considering it as an obstacle problem as follows.  Let $\mu_0$ be a finite non-negative Borel measure on $\Omega$, and $\varphi \in C(\overline{\Omega})$ be a convex function \footnote{If $\Omega$ is strictly convex, then any continuous function on $\partial \Omega$ can be extended to a convex function on $\overline\Omega$. In this case, we can omit the assumption that $\varphi $ is convex.} with $\varphi\ge 0$ on $\partial\Omega$. Let $q\in[0,n)$ be a constant. We define
\[
\D_{\mu_0, \varphi} = \left\{ w \in C(\overline\Omega):\;    w  \geq 0\ \text{is convex},\   w=\varphi \ \text{on }\partial \Omega, \ \mbox{and } \M w\leq  w^q\mu_0 \right\},
\] 
where $w^q\mu_0$ is the Borel measure defined as
\[
(w^q\mu_0)(E):=\int_{E} w^q(x)\ud\mu_0
\]
for every Borel set $E\subset\Omega$. Consider the minimization problem 
\begin{equation}\label{eq:obs problem n}
v = \inf_{w \in  \D_{\mu_0, \varphi} } w.
\end{equation}

\begin{Proposition}\label{prop:existence obs}
The minimizer $v$ of \eqref{eq:obs problem n} exists and belongs to  $\D_{\mu_0, \varphi}$, and it satisfies 
\begin{equation}\label{eq:obs equation00}
\M v = v^q \mu_0 \quad \text{on } \left\{ v >0\right\} \cap \Omega.
\end{equation}
If $\ud\mu_0=g(x)\ud x$ for some $g \in L^{\infty}(\Omega)$, then it satisfies that 
\begin{equation}\label{eq:obs eq 0}
\operatorname{det} D^2 v=g v^q\chi_{\{v>0\}}\quad \text{in } \Omega.
\end{equation}
\end{Proposition}
\begin{proof}
By the solvability of the Dirichlet problem for Monge-Amp\`ere equations, see \cite{gutierrez2016monge,hartenstine2006dirichlet},  there exist $w_0 ,\varphi_0 \in C(\overline{\Omega})$ such that
\[
\M w_0 = \left\|\varphi\right\|_{L^{\infty}}^q \cdot \mu_0  \quad \text{on } \Omega , \quad w_0=\varphi  \quad \text{on }\partial \Omega,
\]
 and
 \[
 \M \varphi_0 = 0  \quad \text{on } \Omega , \quad \varphi_0=\varphi  \quad \text{on }\partial \Omega.
 \]
 The comparison principle implies  $w_0 \leq w \leq \varphi_0$ for all $w \in \D_{\mu_0, \varphi}$. Hence, the convex functions in $\D_{\mu_0, \varphi}$ are uniformly continuous on $\partial \Omega$, and also locally uniformly Lipschitz in $\Omega$. Then all the functions in $\D_{\mu_0, \varphi}$ has a uniform modulus of continuity. As in the proof of \cite[Lemma 2.4]{savin2005obstacle}, for any $w_1, w_2 \in \D_{\mu_0, \varphi}$, the convex envelope of $\min\left\{w_1,w_2\right\}$, denoted as $w_3 \geq 0$, satisfies $|\M w_3 (E)| =0$ for $E=\Omega \setminus \left(\left\{w_3=w_1 \right\} \cup \left\{w_3=w_2 \right\}\right)$, implying that 
 \[
 \M w_3 \leq w_3^q \mu_0 \quad 
\text{in } \Omega
 \]
 in the Aleksandrov sense. Then, by the Arzel\`a–Ascoli theorem and using the diagonal subsequence technique, we find a decreasing minimizing sequence $\{w_k\}$ which converges uniformly to $v \in C(\overline{\Omega})$. The weak convergence of measures under uniformly convergence \cite[Theorem 2.1]{savin2005obstacle} implies for any open set $O \subset \subset \Omega$ that 
\[
\M v (O)\leq \liminf_{k\to\infty} \M w_k (O) \leq \lim_{k\to\infty}   (w_k^q \mu_0)(O) = (v^q\mu_0) (O).
\]
Hence, the minimizer $v$ exists and belongs to $\D_{\mu_0, \varphi}$.

The Perron's argument as in the proof of  \cite[Proposition 1.1]{savin2005obstacle} implies that $\M v = v^q \mu_0$ on $\left\{ v >0\right\} \cap \Omega$. Indeed,
since $\mu_0$ is a finite Borel measure, it suffices to prove that for $x \in\{v>0\}\cap \Omega$, there exists $\varepsilon_x>0$ such that
\begin{equation}\label{eq:perron method}
\M v\left(B_{\varepsilon}(x)\right)=\int_{B_{\varepsilon}(x)} v^q \ud \mu_0, \quad \forall\ \varepsilon<\varepsilon_x .
\end{equation}
Let $\omega \leq v$ denote the solution of the Dirichlet problem $\M \omega= v^q\mu_0$ in $B_{\varepsilon}(x)$, $\omega=v$ on $\partial B_{\varepsilon}(x)$.  By taking $\varepsilon_x$ sufficiently small,  the Aleksandrov maximum principle implies $\omega>0$. If $ \omega<v$ at some point inside $B_{\varepsilon}(x)$, by extending $w$ to be infinite outside $B_{\varepsilon}$, the convex envelope of $\min \{\omega, v\}$ belongs to $\mathcal{D}_{\mu_0,\varphi}$, contradicting the minimality of $v$.  Therefore, we have  $\omega =v$, which gives \eqref{eq:perron method}.

Lastly, if $\ud\mu_0=g\ud x$ with $g \in L^{\infty}$, one easily see that $\M v \leq v^q\mu_0 =v^q g\ud x =0 $ on $\partial\left\{v = 0\right\}$, and thus, $\M v =0 $ on $\left\{v = 0\right\}$. This concludes the proof of the proposition.
 \end{proof}

 We have the following comparison principle for the obstacle problem
\begin{Lemma}[Comparison Principle]\label{lem:comparison principle obs} 
Suppose the non-negative convex functions $\underline w \in C(\overline{ \Omega})$ and $\overline{w} \in C(\overline{ \Omega})$ satisfy 
\[
\M \underline{w} \geq  \underline w^q\chi_{\{\underline w>0\}}\mu_0\quad  \text{in } \Omega ,
\quad 
\M \overline w\le  \overline w ^q  \mu_0 \quad \text{in } \Omega,
\]
and $\overline w\geq \underline w$ on $\partial \Omega$, then  $\overline w \geq \underline w$ in $ \Omega$.
\end{Lemma}
\begin{proof}
Let $O \subset \Omega $ denote the open set of points $x$ for which $\underline{w}(x) > \overline{w}(x)$. Then $O \subset \{ \underline w>0\} $. Hence,  
\[
\M \underline{w} \geq    \underline w^q\chi_{\{\underline w>0\}} \mu_0=  \underline w^q \mu_0 \geq   \overline w ^q \mu_0  \geq \M\overline w   \quad \text{in }O, \quad \underline w  =\overline w \quad  \text{on }\partial O. 
\]   
By the comparison principle \cite[Theorem 1.4.6]{gutierrez2016monge}, we have $\underline w  \leq \overline w$ on $O$, and hence $O= \emptyset$. 
\end{proof}
\begin{Corollary}
The minimizer $v$ is the unique solution to \eqref{eq:obs equation00} within the class $\D_{\mu_0, \varphi}$.
\end{Corollary}

\begin{Lemma}\label{lem:volume concidence}
Suppose $\lda\le g(x)\le\Lda$.  Let $h>0$ be a constant. Suppose $ w>0$ satisfies $\operatorname{det} D^2 w\ge g w^q \chi_{\{w>0\}}$ on a convex open set $O$ with $w \leq h $ on $\partial O$. Then
\[
|O| \leq C(n,q,\lambda)h^{\frac{n-q}{2}} .
\]
\end{Lemma}
\begin{proof}
We can assume $O$ is bounded, since otherwise we just apply the following argument to every bounded convex subset of $O$. 

By some unimodular affine transformation, we may assume that the ellipsoid of minimum volume containing $O$ is $Mh^{\frac{n-q}{2n}}B_{1}(0)$, where $M>0$ (may depend on $h$). By John's lemma, $n^{-\frac 32} Mh^{\frac{n-q}{2n}}B_{1}(0) \subset O$. The function $\Phi(x)= c_1|x|^{\frac{2n}{n-q}}$ satisfies 
\[
\det D^2 \Phi=c_1^{n} c(n,q)|x|^{\frac{2nq}{n-q}}=c_1^{n-q} c(n,q) \Phi^q \le \lda \Phi^q \le g \Phi^q
\] 
if $c_1(n,q,\lambda)$ is small enough. Furthermore,  there exists $C_0$ depending only on $c_1$ and $n$ such that if $M\ge C_0$, then $\Phi \geq h$ on $\partial O$. Consequently,  the comparison principle in Lemma \ref{lem:comparison principle obs}  implies $w(0) \leq \Phi(0)=0$, which is impossible. Hence, $M\le C_0$, and thus, $|O| \leq C(n,q,\lambda)h^{\frac{n-q}{2}}$.
\end{proof}

 \begin{Corollary}\label{coro:volume concidence}
Suppose $\lda\le g(x)\le\Lda$, and  $ w\geq 0$ satisfies $\operatorname{det} D^2 w\ge g w^q \chi_{\{w>0\}}$ in $\Omega$. Let $K:=\left\{x\in\Omega: w(x)=0\right\}$. If $|K|=0$, then $ \dim K \leq \frac{n+q}{2}$.
\end{Corollary}
\begin{proof}
	After an affine transformation, let us assume for simplicity that the mass center of $K$ is $0$, and $K \subset \left\{ x_1 = 0\right\}$ (since $|K|=0$). Let $O_h:=\left\{ w < h\right\} \cap \left\{ x_1 > 0\right\} \subset \left\{ 0<w<h\right\}$. Since $w$ is convex and locally Lipschitz, we can find some $c_w>0$ small such that $(\{x_1>0\}\cap \operatorname{conv} \left\{ c_{w}B_{h}(0)  , K\right\}) \subset O_h$,  where $\operatorname{conv} E$ denotes the convex hull of the set $E$.  Applying Lemma \ref{lem:volume concidence}, we obtain
\[
c_{w,K} h^{n- \dim K}  \leq  |O_h| \leq   C(n,q,\lambda) h^{\frac{n-q}{2}}.
\]
By sending $h \to 0$, we conclude that $ \dim K \leq  \frac{n+q}{2}$.
\end{proof}

Let $v$ be a solution of \eqref{eq:obs equation}. Similar to \cite[Lemma 3.2]{savin2005obstacle}, a lower bound for the expansion of $v$ away from the free boundary will be established below. This ensures the stability of the coincidence sets under the uniform convergence of solutions in the case where $|K|>0$.
\begin{Lemma}\label{lem:bound from below}
Suppose $0 \in \partial K$ and $K \subset\left\{x_n\ge 0\right\}$. Then, if $y=(y',y_n)\in \Omega$, $y_n \leq 0$ we have
\[
v(y) \geq c(n,q,\lambda)\left|K \cap\left\{x_n \le \left|y_n\right|\right\}\right|^{\frac{2}{n-q}}.
\] 
\end{Lemma}
\begin{proof}
If $z \in K \cap\left\{x_n <\left|y_n\right|\right\}$, then $v\left(\frac{1}{2} y+\frac{1}{2}z\right) \leq  v(y)=: h$, hence $v<h$ on the convex set
\[
\Omega_y:=\frac{1}{2} y+\frac{1}{2}\left(K \cap\left\{x_n \leq\left|y_n\right|\right\}\right) .
\]
As $\Omega_y \subset \left\{x_n<0\right\} $, it follows from the assumption that $v>0$ on $\Omega_y$. Thus, the conclusion can be derived from Lemma \ref{lem:volume concidence} by choosing $O=\Omega_y  $.
\end{proof}

\begin{Remark}
If $B_{c_1}(x_0) \subset (K \cap  B_{C_1}(0))$, then the function $v$ satisfies
\begin{equation}\label{eq:bound from below}
v(x) \geq c(n, q, \lambda, c_1, C_1)[\operatorname{dist}(x, \partial K)]^{\frac{2n}{n-q}}  \quad   \text{in  }  B_{C_1}(0)\setminus K.
\end{equation} 
Indeed, let $y\in\partial K$ be the point where the distance from $x$ to $\partial K$ is realized.  Since $K$ contains the convex set generated by $y$ and $B_{c_1}(x_0)$, we find
\[
\left|K \cap\left\{z:\;\left(z-y\right) \cdot \frac{\left(y-x\right)}{\left|y-x\right|} \leq  \operatorname{dist}(x, \partial K) \right\}\right| \geq c\left|x-y\right|^n,
\]
which then, by Lemma \ref{lem:bound from below}, implies
\[
v(x) \geq c|x-y|^{\frac{2n}{n-q}}.
\]
This proves \eqref{eq:bound from below}.
\end{Remark}

Suppose  $Q$ is a closed convex set (a relative closed subset of $\Omega$). A nonempty convex subset $E$ of $Q$  is called a face of $Q$ if $\alpha x+(1-\alpha) y \in E$ with some  $x$, $y \in Q$ and some $\alpha\in(0,1)$, then $x, y \in E$. If $E$ contains an interior point, then $E=Q$.  A face that can be represented as the intersection of the convex set with its supporting hyperplane  is called an exposed face. As a convention here, we alway call $Q$ itself an exposed face of $Q$ as well. Each face is contained within an exposed face. A zero-dimensional face is an extreme point, and a zero-dimensional exposed face is an exposed point. 
We denote by $Q^{ext}$ and $Q^{exp}$ the sets of extreme points and exposed points of $Q$, respectively. By definition,  we have
\[
Q^{exp} \subset Q^{ext} \subset \operatorname{rel bd}(Q) \subset \partial Q,
\]
 where $\operatorname{rel bd} (Q)$ is the boundary of $Q$ relative to its affine hull.
The converse $Q^{exp} \supset Q^{ext}$ may not necessarily hold  by considering
\[
Q = \left\{  (x_1,x_2) \in \R^2:\; x_2 \geq  \max\{x_1,0\}^2\right\},
\]
where $Q^{ext} =\partial Q \cap \left\{ x_1 \geq 0\right\}$ and $Q^{exp}=\partial Q \cap \left\{ x_1 > 0\right\}$. We shall call the set of exposed points the strictly convex part of $\partial Q$. Given an extreme point, if there is no line segment on $\partial Q$ with this point as one of the endpoints, then it is also an exposed point. Furthermore, Straszewicz's theorem states that the set of exposed points is dense within the set of extreme points.

For any two closed convex sets $E \subset Q$, we have 
\[
(Q^{exp} \cap E) \subset  E^{exp} ,\quad   (Q^{ext} \cap E)\subset  E^{ext}.
\] 
In addition, if $E $ is a face of $Q$, then $E^{ext}=(Q^{ext} \cap E)$.

A convex function $w  $ on $\Omega$ is strictly convex at $x_0 \in \Omega$ if $x_0$ serves as an exposed point of the epigraph of $w$.   This condition is equivalent to the existence of a $p \in \partial w(x_0)$ such that
\[
w(x)>w(x_0)+p \cdot(x-x_0),   \quad  \forall\ x \neq x_0.
\]
The section of $v$ at $x_0$ with height $t$  for subgradient $p \in \partial v(x_0)$ is denoted as
\[
S_{t}(x_0)= S_{t}^v(x_0,p) = \left\{ x   :\;  v(x) < v(x_0)+p\cdot (x-x_0)+t  \right\}.
\]
The strict convexity at $x_0$ is also equivalent to say that $S_{t}^v(x_0,p) \subset \subset\Omega$ for some $t>0$ small. If $w$ is strictly convex in $\Omega$, i.e., $w$ is strictly convex at every point in $\Omega$, then any straight line connecting any two points on the graph of $w$ lies strictly above the graph, except  at the endpoints.

We begin by revisiting the notion of centered sections initially introduced by Caffarelli in \cite[Lemma 1]{caffarelli1996boundary}. 
\begin{Lemma}[Centered Section {\cite[Lemmas 2.6]{savin2005obstacle}}]\label{lem:savin lem2.6} 
Let $w: \mathbb{R}^n \rightarrow \mathbb{R} \cup\{\infty\}$ be a globally defined convex function. Also, assume $w$ is bounded in a neighborhood of $x_0$ and the graph of $w$ does not contain an entire line crossing $x_0$.
	
Then, for each $h>0$, there exists a ``centered section" $\widetilde{S}_h^w(x_0)$ at $x_0$, that is, there exists $p_h\in\R^n$ such that the convex set
\begin{equation}\label{eq:centered section}
\widetilde{S}_h:=\widetilde{S}_h^w(x_0)=\left\{x: \; w(x)< w(x_0)+p_h\cdot (x-x_0)+h \right\}
\end{equation}
is bounded and has $x_0$ as the center of mass.
\end{Lemma}

A bounded set $E\subset \R^{n}$ is said to be $\kappa$-balanced around $x_0 \in \R^n$ for $\kappa>0$ if
\begin{equation}\label{eq:balanced def}
t(x_0-E)\subset E-x_0,  \quad \forall\ t \in [0,  \kappa]. 
\end{equation}
If $\kappa $ is a universal constant, meaning a constant depending only on $n,q,\lambda$ and $\Lambda$, we just say that $E$ is balanced around $x_0$ for simplicity. By John's lemma, a bounded convex set (and hence, every centered section in Lemma \ref{lem:savin lem2.6}) is always $c(n)$-balanced around its mass center for some constant $c(n)$.

Let $v$ be a convex solution of \eqref{eq:obs equation}. We extend $v$ to be infinite outside $\Omega$. Suppose $0\in\Omega$. We consider the centered section
\[
\widetilde{S}_h=\widetilde{S}_h^v(0),
\]
and express $\widetilde{S}_h=\left\{x: \; v(x)<\ell_h(x)\right\}$ for some $\ell_h(x)=v(0)+p_h\cdot x+h$. Since $v$ is non-negative,
$\ell_h \geq v \geq  0$ in $\widetilde{S}_h$.  Suppose  $x\in \widetilde{S}_h$.  By the balanced property of $\widetilde{S}_h$, we have for $c=c(n)$ that $-cx\in \widetilde{S}_h$, and thus,
\[
\frac{c}{1+c}v(x) \leq \frac{c}{1+c}\ell_h(x)=  \ell_h(0)-\frac{1}{1+c}\ell_h(-cx) \leq \ell_h(0) = v(0)+h,
\] 
which implies 
\begin{equation}\label{eq:centerS right}
\widetilde{S}_{h} \subset S_{C(n)(v(0)+h)}, 
\end{equation} 
where $C(n)=\frac{1+c(n)}{c(n)}$, and
\[
S_{h} :=  S_{h}^v(0,0) =\left\{ x:\; v(x) <  h\right\} \quad \text{ for } h>0.
\]
Since $v(0)<\ell_h(0)$, then $\widetilde{S}_{h} $ is a section of $v$ at certain point with height at least $h$.  Also, due to the convexity of $v-\ell_h$, the balanced property implies 
\begin{equation}\label{eq:centerS nondege}
v(x)- \ell_{h}(x) \geq \frac{1+c}{c}(v(0)-\ell_h(0)) -\frac{1}{c}(v(-cx)- \ell_{h}(-cx)) \geq  -C(n) h  \quad  \text{in }\widetilde{S}_h.
\end{equation}
Furthermore, if $0$ is an extreme point of $K$, then due to the balanced property of $\widetilde{S}_{h}$ around $0$, we have
\[
\varlimsup_{h \to 0} \widetilde{S}_{h} \subset \left(\left(\varlimsup_{h \to 0} \widetilde{S}_{h}\right) \cap \left(-\varlimsup_{h \to 0} C\widetilde{S}_{h}\right) \right) \subset  \left\{0\right\},
\]
which implies 
\begin{equation}\label{eq:center sec shrink}
\lim_{h \to 0} \widetilde{S}_{h}= \left\{0\right\}.
\end{equation} 
In general, we are only concerned with the centered sections that satisfy  $\widetilde{S}_{h} \subset\subset \Omega$.

\begin{Lemma}\label{lem:volume Shc up}
Suppose $0\in\partial K$,  then
\[
|\widetilde{S}_h| \leq C(n,q,\lambda)h^{\frac{n-q}{2}}, \quad |\M v(\widetilde{S}_h)| \leq C(n,q,\lambda,\Lambda)h^{\frac{n+q}{2}}.
\]
\end{Lemma}
\begin{proof}
After a rotation, we may assume that  $K \subset \left\{ x_n \geq 0\right\}$. Then, we can apply Lemma \ref{lem:volume concidence} to the convex set $S_{C(n)h} \cap  \left\{ x_n \leq 0\right\}$  to conclude that $|S_{C(n)h}| \leq C(n,q,\lambda)h^{\frac{n-q}{2}}$. Thus, using \eqref{eq:centerS right} and the balanced property of $\tilde{S}_h$,  
\[
|\widetilde{S}_h| \leq C(n) |\widetilde{S}_h\cap  \left\{ x_n \leq 0\right\} | \leq  C(n) |S_{C(n)h}\cap  \left\{ x_n \leq 0\right\} | \leq  C(n,q,\lambda)h^{\frac{n-q}{2}}.
\]
Consequently, we have $|\M v (\widetilde{S}_h)|\leq \Lambda \int_{\widetilde{S}_h}v^q \ud x \leq C(n,q,\lambda,\Lambda)h^{\frac{n+q}{2}}$.
\end{proof}

The following lemma originates from \cite{caffarelli1997properties}, and the proof can be found in \cite[Lemma 2.5]{mooney2015partial}.  
\begin{Lemma}\label{lem:measure constrain} 
For any section $S_{h}^v(x_0,p) \subset \subset \Omega$, we have 
\[
| \M v (S_{h}^v(x_0,p))| \cdot |S_{h}^v(x_0,p)| \geq c(n) h^{n}. 
\] 
\end{Lemma}

Consequently, 

\begin{Lemma}\label{lem:volume of Shc}
Suppose $0\in\partial K$, and $\widetilde{S}_h\subset \subset \Omega$, then we have
\[
c(n,q,\lambda,\Lambda)h^{\frac{n-q}{2}} \leq |\widetilde{S}_h| \leq C(n,q,\lambda)h^{\frac{n-q}{2}},
\]
thus
\[  c(n,q,\lambda,\Lambda)h^{\frac{n+q}{2}} \leq |\M v(\widetilde{S}_h)| \leq C(n,q,\lambda,\Lambda)h^{\frac{n+q}{2}}
\]
\end{Lemma}
\begin{proof}
Note that $\widetilde{S}_{h} $ is a section of $v$ at certain point with height at least $h$. Then the conclusions follow from Lemma \ref{lem:volume Shc up}, and Lemma \ref{lem:measure constrain} applied to $\widetilde{S}_h$.
\end{proof}

\begin{Lemma}\label{lem:volume of Sh K=0}
Suppose $K=\left\{0\right\}$, and $S_{h} \subset \subset \Omega$, then we have
\[
\widetilde{S}_{c(n)h}  \subset  S_{h} \subset C(n,q,\lambda) \widetilde{S}_{c(n)h} ,\quad C(n,q,\lambda,\Lambda) h^{\frac{n-q}{2}} \leq |{S}_h| \leq C(n,q,\lambda) h^{\frac{n-q}{2}}.  
\]
\end{Lemma}
\begin{proof}
From \eqref{eq:centerS right}, we know that $\widetilde{S}_{c(n)h}  \subset  S_{h}$ since $v(0)=0$.

Suppose, for simplicity, there exists a point $x_0= ae_{n} \in  S_{h} \setminus M \widetilde{S}_{c(n)h}  $ with $M$ being large.  
	Then, by applying Lemma \ref{lem:volume concidence} for the convex set $O := \operatorname{conv}\left\{ x_0, \widetilde{S}_{c(n)h}  \cap \left\{ x_n >0 \right\} \right\}$, we obtain that $M \leq C(n,q,\lambda) $. This implies that $S_{h} \subset C \widetilde{S}_{c(n)h}$.   Then the estimates on $|S_h|$ follow from Lemma \ref{lem:volume of Shc}.
\end{proof}

\begin{Lemma}\label{lem:balanced other}
Suppose $0\in\partial K$, then for any centered section $\widetilde{S}_h \subset \subset \Omega$,  and every $x_0$ satisfying  $v(x_0) -\ell_h(x_0) \leq -\ell_h(0)$, $\widetilde{S}_h$ is $ c(n,q,\lambda,\Lambda)$-balanced around $x_0$.
\end{Lemma}
\begin{proof}
Applying Lemma \ref{lem:volume of Shc}, we can use John's lemma to find a diagonal transformation $\T_h$ with $c  h^{\frac{n-q}{2}} \leq \det \T_h \leq Ch^{\frac{n-q}{2}} $ that maps the unit ball to the minimum volume ellipsoid of $\widetilde{S}_h $, such that $ \T_h B_c(0)  \subset   \widetilde{S}_h  \subset  \T_h  B_C(0)$. Then, we introduce the following normalization  of $v$ at $0$ for $\widetilde{S}_h $,
\begin{equation*}
v_{h} (x) = \frac{ v\left( \T_hx\right)}{h} ,\quad \tilde{\ell}_{h} (x)=\frac{\ell_{h}\left( \T_hx\right)}{h} , \quad \tilde{g}_{h} (x)=(\det \T_h)^2h^{q-n}g\left( \T_hx\right), \quad x \in \T_h^{-1}\tilde{S}_h.   
\end{equation*}
Then, $0\leq v_{h} \leq C$ in $\T_h^{-1}\tilde{S}_h$ (by \eqref{eq:centerS right}) and  is a solution of 
\[
\det D^2 v_{h} =g_h  v_{h}^q \chi_{ \left\{v_{h}>0  \right\}}  \quad \text{in } \T_h^{-1} \tilde{S}_h, \quad v_{h} =\tilde{\ell}_{h}   \quad  \text{on }\partial \T_h^{-1} \tilde{S}_h, \quad   0 < c \leq g_h  \leq   C. 
\]
In particular, $\det D^2 v_h \leq C$ in $\T_h^{-1} \tilde{S}_h$. Applying Lemma \ref{lem:measure constrain 2} to the normalization $v_{h}- \tilde{\ell}_{h} $, we obtain the uniform continuity modulo of $v_{h}- \tilde{\ell}_{h}$  up to the boundary of $\T_h^{-1}\widetilde{S}_h $. Hence, the linear function $\tilde \ell_{h}$ is bounded on $\T_h^{-1} \widetilde{S}_h$, which provides a uniform bound for $|\nabla \ell_{h}|$ in $\T_h^{-1} \tilde{S}_h$. Thus, we obtain a uniform  modulo of continuity of $v_{h}$  up to the boundary of $\T_h^{-1}\widetilde{S}_h $.

Furthermore, observing that 
\[
v_{h}(\T_h^{-1}  x_0) -\tilde{\ell}_{h} (\T_h^{-1}  x_0) =  \frac{v(x_0) -\ell_h(x_0)}{h}\le \frac{-\ell_h(0)}{h}=-1 ,
\]
we can infer from Lemma \ref{lem:measure constrain 2}  that $\operatorname{dist}(\T_h^{-1}x_0, \partial \T_h^{-1}\widetilde{S}_h) \geq c$. Consequently, $\T_h^{-1} \widetilde{S}_h $ is balanced around $\T_h^{-1}  x_0$, and thus, $\widetilde{S}_h $ is balanced around $x_0$, since the balanced property is affine invariant.
\end{proof}

\begin{Lemma}\label{lem:ext are exp}
Suppose $0\in \Omega$ is an extreme point of $K$, then it cannot serve as an endpoint of a line segment on $\partial K$, and thus, it is an exposed point. In addition, if $|K|=0$, then $K =\left\{0\right\}$.
\end{Lemma}
\begin{proof}
Suppose the contrary that $0$ is an endpoint of a line segment on $\partial K$. Denote this line segment as $[0,e]$.  By \eqref{eq:centerS right}, we have for $0<a<1$ that
\[
\varlimsup_{h \to 0}  \widetilde{S}_h(ae) \subset K= \lim_{h \to 0}  S_{C(n)h}.
\]
By noting that $\widetilde{S}_h(ae)$ is $c(n)$-balanced around its mass center $ae$, $\varlimsup_{h \to 0}  \widetilde{S}_h(ae) $ is $c(n)$-balanced around $ae$, and we have
\[
\varlimsup_{h \to 0}  \widetilde{S}_h(ae) \subset K \cap \left(ae-C(n)K\right).
\]	 
Since $0$ is an extreme point, we have
\[
\varlimsup_{a \to 0} K \cap \left(ae-C(n)K\right) =\{0\}. 
\]
Therefore, we have for  $a>0$ small that $\varlimsup_{h \to 0}  \widetilde{S}_h (ae) \subset \subset \Omega \setminus \left\{ e\right\}$, and consequently, for $h>0$ and $a>0$ small that 
\[
\widetilde{S}_h(ae) \subset \subset \Omega  \setminus \left\{ e\right\}.
\]
For these small $h$, we denote 
\[
\widetilde{S}_h(ae) =\left\{x: \; v(x)<\ell_h(x):=v(ae)+p\cdot (x-ae)+h\right\}.
\] 
Since $e \notin \widetilde{S}_h(ae)$, we have $\ell_{h}(e) \leq v(e) =0$. Given that $\ell_{h}(ae)=h$,  by the linearity of $\ell_h$, we find that  $\ell_{h}(0) > h$. Consequently, $0 \in 	\widetilde{S}_h(ae)$ and  $v(0)-\ell_h(0) \leq -\ell_h(ae)$. Thus, Lemma \ref{lem:balanced other} implies that $	\widetilde{S}_h(ae)$ is balanced around $0$.  Since $ae \in \widetilde{S}_h(ae)$, then $-cae$, and consequently the line segment $[-cae,ae]$, are in $\widetilde{S}_h(ae)$ for all $h$, where the constant $c=c(n,q,\lambda,\Lambda)>0$. Since $\varlimsup_{h \to 0}  \widetilde{S}_h(ae) \subset  K $, we obtain $[-cae,ae]\subset K$, which is a contradiction to the assumption that $0$ is an extreme point.

Suppose now that $|K|=0$. Then we have $K =\partial K$. If there exists another point $x_0 \in K \setminus \left\{0\right\}$, then $0$ is an endpoint of the segment on $\partial K$ that connects $0$ and $x_0$, which is impossible.
\end{proof}

Next, we provide a characterization for the set $\Gamma_{nsc}:=\Gamma \setminus \Gamma_{sc}=(\Omega \cap \partial K) \setminus K^{exp}$. Since we are interested in the interior regularity of the solution $v$ and its free boundary $\Gamma$, one notes from the definition of $K$ that $K$ is a subset of $\Omega$, not touching the boundary $\partial\Omega$. 
We say a convex set is non-trivial if it contains more than one elements.

\begin{Lemma}\label{lem:nsc property}
$\Gamma_{nsc}=\bigcup\{E \subset \partial K:  E \mbox{ is a non-trivial exposed face of }K\}$. Furthermore, for every non-trivial exposed face $E \subset \partial K$ of $K$, we have $E^{ext} =\emptyset$. 
\end{Lemma}
\begin{proof}
For every point $x_0\in \Gamma_{nsc}$. Let $H$ be a supporting hyperplane of $K$ at $x_0$. Then $E:= (H\cap K)\subset \partial K$  is a non-trivial exposed face containing $x_0$. 

Conversely, let $E\subset \partial K$ be a non-trivial exposed face of $K$. For any $x_0\in E$, there always exists a line segment $L\subset E\subset\partial K$ which contains $x_0$. Lemma \ref{lem:ext are exp} implies  $x_0\in \Gamma_{nsc}$. Therefore, $E\subset \Gamma_{nsc}$. Furthermore, we have  
\[
E^{ext}=(K^{ext} \cap E)=(K^{exp} \cap E) \subset  (K^{exp} \cap \Gamma_{nsc}) =\emptyset.
\]
\end{proof}

Consequently, we can show that $\Gamma_{sc}$ is relative open as follows. 

\begin{Lemma}\label{lem:exposed set}
If there exists  $\delta >0$ such that
\[
K \subset \left\{ x_n \geq 0\right\}, \quad K \cap \left\{ x_n \leq \delta \right\} \subset \subset \Omega,
\]	
then every point  on $\partial K \cap \left\{ x_n \leq \delta \right\}$ is an exposed point of  $K$. Consequently, the set $\Gamma_{sc}=K^{exp}\cap\Omega$ is an open subset of $\partial K$.
\end{Lemma}
\begin{proof}
If  $x \in  \partial K \cap \left\{ x_n \leq \delta \right\} $ is not an exposed point of $K$, then by Lemma \ref{lem:nsc property}, there exists a non-trivial convex set $E \subset \partial K$ such that the extreme points of $E$ lie on $\partial\Omega$ and $E$ passes through $x$. Since $K \cap \left\{ x_n \leq \delta \right\}$ does not intersect $\partial\Omega$, we have that $E^{ext}\subset\{x_n>\delta\}$. Hence, $x\in E\subset\{x_n>\delta\}$, which is a contradiction.

Suppose $\Gamma_{sc}$ is not empty. Let $x_0\in \Gamma_{sc}$. After a translation and rotation, we assume $x_0=0$, $K \subset \left\{ x_n \geq 0\right\}$, and $K \cap \left\{ x_n = 0\right\}=\{0\}$. Then there exists $\delta>0$ such that  $K \cap \left\{ x_n \leq \delta \right\} \subset \subset \Omega$. Hence, every point  on $\partial K \cap \left\{ x_n \leq \delta \right\}$ is an exposed point of  $K$. This implies that $\Gamma_{sc}$ is an open subset of $\partial K$. 
\end{proof}

Here, we present some results regarding the non-existence of $\Gamma_{nsc}$.

\begin{Proposition}\label{prop:infinite ray}
Suppose $v$ is a solution to \eqref{eq:obs equation}. Then in any of the following three scenarios:
\begin{itemize}
\item[(1).] $\Omega=\R^n$;
\item[(2).] $v > 0$ on $\partial \Omega$;
\item[(3).] $n=2$ and $q=0$;
\end{itemize}
we have that $ \Gamma \setminus \Gamma_{sc}=\emptyset $.
\end{Proposition}

\begin{proof} 
Suppose $x \in \Gamma \setminus \Gamma_{sc}$. By Lemma \ref{lem:nsc property},  there exists a line segment $L$, say $L \subset \left\{x'=0\right\}$ on $\Omega \cap \partial K $, extending to the boundary $\partial\Omega$ or infinity.
If $L$ does not intersect $\partial \Omega$, then $L$ is a line. Then for any $(x',x_n) \in \Omega$ and $t \in \R$, by using the convexity of $\Omega$ and $v$,  we have $(x',x_n+t ) \in \Omega$ and $v(x',x_n+t)=v(x',x_n)$. Thus,   $\partial v (\Omega) \subset \left\{ x_n =0\right\}$. Consequently, $\det D^2 v \equiv 0$. By the equation  \eqref{eq:obs equation}, we have $v \equiv 0 $,  and $\Gamma=\emptyset$. This  in particular implies that $ \Gamma \setminus \Gamma_{sc}=\emptyset $ if $\Omega=\R^n$.

Therefore, it remains to consider the case that $L$ intersects with $\partial \Omega$.

In case (2), since $v>0$ on $\partial\Omega$, then $L$ cannot intersect with $\partial \Omega$, which is a contradiction.

In case (3),  for simplicity, let us assume that  $(\Omega \cap \left\{x_1 = 0\right\})\subset \Gamma_{nsc}$, and $K \subset \left\{x_1 \leq 0\right\}$. Then, we have  $\lambda \leq \det D^2 v \leq \Lambda$ on $\Omega\cap \left\{x_1 >0\right\}$ with $v=0$ on $  \left\{x_1 = 0\right\}$, this is impossible, by \cite[Lemma 6.1]{mooney2024sobolev}  or \cite[Lemma 3.5]{jian2021boundary}. 
\end{proof}

\begin{Remark}\label{rem:nonstrictconvex}
It is not always true that  $ \Gamma\setminus \Gamma_{sc} \neq \emptyset $. For the case $q = 0$, there are lots of examples, such as Pogorelov's example or the singular solutions constructed in \cite{caffarelli1993note} and \cite{caffarelli2022singular}. 	 

Below, we introduce a two-dimensional example for the case $q \in (0,2)$. Specifically, let $\alpha =\frac{q}{2} \in (0,1)$ and let $\zeta(t), t\in \R,$ solve
\[
\alpha (1+\alpha)\zeta \zeta''-(1+\alpha)^2\zeta'^2=(1+|t|^{\alpha}\zeta)^{2\alpha},\quad \zeta(0)=1,\quad \zeta'(0)=0,
\]
which always admits a positive convex solution for small $|t|<c(\alpha)$. Then the function
\[
w(x)=\max\left\{x_{2},0\right\}+ \max\left\{x_{2},0\right\}^{1+\alpha}\zeta(x_{1}),\quad x=(x_1,x_2) \in \R^2 
\] 
is a solution of \eqref{eq:obs eq 0}  in $B_{c(\alpha)}(0)$ with $g\equiv 1$. In this case, $\Gamma_{sc}=\emptyset$, $\Gamma_{nsc}$ is a line segment, and the solution is merely Lipschitz continuous.
\end{Remark}

Combining Lemma \ref{lem:volume concidence} and Lemma \ref{lem:ext are exp}, we have the following stability of the coincidence sets under uniform convergence.
 \begin{Lemma}\label{lem:stability of K}
 Let $v_i \geq 0$, $i=1,2,\cdots,$ be convex solutions of $	\det D^2 v_i =g_i  v_i^q \chi_{\{v_i>0\}}$ in $\Omega$ with $0 <\lambda \leq g_i  \leq   \Lambda$. Suppose that $\left\{v_i\right\}$ locally uniformly converges to some $v_{\infty}$. Suppose there exists a nonempty,  convex, relative closed subset $K_\infty$ of $\Omega$ such that $K_{i}=\left\{v_i=0\right\}$ converges to $K_{\infty}$ under the Hausdorff metric. If $\left\{v_{\infty}=0 \right\}$ contains an extreme point in $\Omega$, then $\left\{v_{\infty}=0 \right\}=K_{\infty}$.
 \end{Lemma}
 \begin{proof}
By uniform convergence, we have $K_{\infty} \subset \left\{v_{\infty}=0 \right\}$.  Note that $v_{\infty}$ is still a solution of the obstacle problem, and $\left\{v_{\infty}=0 \right\}$ is a relative closed subset of $\Omega$. 	If $ K_{\infty} \subsetneq \left\{ v_{\infty}=0 \right\}$, then $\left\{v_{\infty}=0 \right\}$ contains at least two points. By Lemma \ref{lem:ext are exp}, we have $|\left\{v_{\infty}=0 \right\}| >0$. Therefore,  we can find an open ball $B_{r_0}(x_0) \subset \left\{v_{\infty}=0 \right\} \setminus K_{\infty}$. Moreover, we may assume that  $x_0 \notin K_i$ for all $i$ large. Then, due to convexity, there exist half spaces $H_i^-$ containing  $x_0$ such that $H_i^- \cap K_{i} =\emptyset$. Thus, $v_{i} > 0 $ on $ B_{r_0}(x_0)  \cap H_i^-$. Applying Lemma \ref{lem:volume concidence}, we obtain that $\sup_{x \in B_{r_0}(x_0)  \cap H_i^-  } v_{i}\geq c(n,q,\lambda,r_0) >0$. This contradicts the fact that $v_{i}$ uniformly converges to $v_{\infty} =0$ on $B_{r_0}(x_0) \subset \left\{ v_{\infty} =0 \right\}$.
 \end{proof}

Suppose now $|K|>0$.   After a rotation, we can always assume that  $K \subset \left\{ x_n \geq 0\right\}$ and that $\partial K$ is locally an upper graph in the $e_n$-direction. Then, we consider the following normalization of $v$  in the spirit of  \cite{savin2005obstacle}.  

\begin{Definition}\label{def:normalized p obp}
For ${\kappa}>0$, let $\E_{\kappa}$ denote the family of convex functions $v$  satisfying:
\begin{itemize}
\item[i)] $v$ is a non-negative, convex and continuous function defined in $\overline{V_v} \cap \{x_n \leq 1 \}$ for some convex bounded open set $V_v$;
\item[ii)] $ v = \kappa >0$ on $(\partial V_v)\cap  \{x_n < 1 \}$;
\item[iii)] $\M v=0$ on $K= \{v=0 \}$, and $\lambda v^q \ud x \leq \M v \leq \Lambda v^q \ud x$ in the set $\{v>0\}$;
\item[iv)] $K \subset \{x_n \geq 0 \}$ and $0\in K \cap \{x_n=0 \} $;
\item[v)]  $K \cap \{x_n=1 \}$ is normalized in $\mathbb{R}^{n-1}$, i.e.,  $\mathcal{H}^{n-1}(K\cap\{x_n=1\})=1$, 
\[
\big(\{x_n=1 \} \cap  \{ |x^{\prime}-y' | \leq c(n) \} \big) \subset \big(K \cap \{x_n=1 \} \big) \subset \{ |x^{\prime} | \leq C(n) \},
\]
where $y'$ denote the mass center of $K \cap \{x_n=1 \}$;
\item[vi)] $K\cap \{x_n \leq 1 \}$ is an upper graph, i.e., if $x \in K$ then $x+t e_n \in K$ for all $t \in (0,1-x_n)$.
\end{itemize}
\end{Definition}

\begin{Remark}
Note that we will use only finitely many iterations of normalizations, so a change in the universal constants does not affect our proof below.
\end{Remark}

The emergence of the model in Definition \ref{def:normalized p obp} arises from rescaling solutions on the free boundary as follows:
\begin{Definition}\label{def:normal type1}
Let $v$ be a convex solution of \eqref{eq:obs equation}. Suppose $ 0 \in K \subset \left\{ x_n \geq 0\right\}$, and for some $t_0>0$, $(K \cap \left\{x_n \leq t_0\right\}) \subset\subset \Omega $ is locally an upper graph of $\psi$ in the $e_n$-direction. For $t\in(0,t_0)$, we can denote 
\[
K_t'=S_{t}^{\psi}(0,0) =\left\{ x'\in \R^{n-1}:\; \psi (x') \leq t\right\}= \{x'\in \R^{n-1}:(x',t)\in K\}.
\]
Applying John's lemma, for each fixed $t\in (0,t_0)$, there exists an affine transformation $\D_{t}'$ in $\mathbb{R}^{n-1}$ such that  
\[
\D_{t}' B_{c(n)}'(x_t') \subset K_t' \subset \D_{t}' B_{C(n)}'(x_t'), \quad \det D_{t}'=\mathcal{H}^{n-1}(K_t')
\]
hold for some $x_t'$.  Let $h:=h(t)=(t \operatorname{det} \D_t')^{\frac{2}{n-q}}$.  Clearly, $h$ is continuous and increasing on $t$, and as $t \to 0$, we have $h \to 0$ and $\kappa/h \to \infty$. Moreover, $c_Kt^{\frac{2n}{n-q}}\le h(t)\le C_K t^{\frac{2}{n-q}}$. Then, the following normalization $v_h$ of $v$ at $0 \in K \cap \left\{ x_n \geq 0\right\}$  is in $\E_{\kappa/h}$ for all $0<\kappa \leq  \inf_{\partial \Omega \cap \left\{x_n \leq t\right\}} v $, where
\begin{equation}\label{eq:normalized sol}
v_{h}\left(x', x_n\right):=h^{-1} v\left(\D_t' x', t x_n\right).
\end{equation}
Furthermore, $v\in \E_{\kappa}$ implies $v_{h} \in \E_{\kappa/h}$, which further implies that $v_h\in \E_1$ for all $h\le \kappa$.
\end{Definition}

Similar to \cite[Lemma 3.4]{savin2005obstacle}, we shall establish the compactness within the space $\E_{1}$. 
\begin{Lemma}[Convergence of ``graph solutions"]\label{lem:compactness of E1}
Suppose $v_{i} \in \E_1$,  $i=1,2,\cdots$, then there exists $\delta=\delta(n, q,\lambda,\Lambda)>0$, a function $v_{\infty}$ which satisfies iii) in Definition \ref{def:normalized p obp}, and a subsequence of $\{v_i\}$ that uniformly converges in $\left\{x_n \leq 3/4\right\}\cap\{v_\infty\le\delta\}$ to $v_\infty$.

Moreover, $\{v_i=0\}\cap \{x_n\le 3/4\}$ converges uniformly to $\{v_\infty=0\}\cap \{x_n\le 3/4\}$ in the Hausdorff distance topology.
\end{Lemma}
\begin{proof} 
We claim that there exists a small $\delta=\delta(n,\lambda,\Lambda)\in(0,1)$ such that 
\begin{equation}\label{eq:E1 normal}
B_c(0) \subset  (\left\{v \leq \delta\right\}  \cap\left\{x_n \leq 3 / 4\right\}) \subset  (\left\{v \leq 1\right\}  \cap\left\{x_n \leq 1\right\}) \subset B_C(0), \quad \forall\ v \in \E_1.
\end{equation} 

We first demonstrate the right-hand side of  \eqref{eq:E1 normal}. By assumptions iv), v), vi) in Definition \ref{def:normalized p obp}, we have $B_{c/2}(\tilde{x}) \subset (K\cap \left\{ x_n \leq 1\right\}) \subset B_{2C}(0)$, where $\tilde{x} = \left( \frac{y'}{2}, \frac{1}{2} \right) $. By \eqref{eq:bound from below}, we obtain the right-hand side of  \eqref{eq:E1 normal}. 

Next, we  establish the left-hand side of \eqref{eq:E1 normal}.  Let $E =\left\{ v+ 16\delta(x_n -7/8) \leq 0 \right\}$ and $\tilde{v}= \inf\left\{ v+16\delta(x_n -7/8),0\right\}$. It is clear that 
\[
(\left\{v \leq \delta\right\}  \cap\left\{x_n \leq 3 / 4\right\}) \subset\subset E \subset  \left\{x_n \le 7 / 8\right\}.
\]
Since $\left\{ v <1\right\} \subset B_C(0)$, then for sufficiently small  $\delta(n,q,\lambda,\Lambda)$, we have that 
\[
v+ 16\delta(x_n -7/8)>0 \quad\text{on }\{v=1\}\cap \{x_n\le 7/8\}.
\]
By using the convexity, we obtain $E\subset \{v<1\}$. Hence,
\[
(\left\{v \leq \delta\right\}  \cap\left\{x_n \leq 3 / 4\right\}) \subset\subset E \subset (\left\{ v <1\right\}\cap \left\{x_n \le 7 / 8\right\}).
\]
Since $\det D^2 \tilde{v} \leq \Lambda$ in $E$,  the Aleksandrov maximum principle (Lemma \ref{lem:measure constrain 2}) ensures the uniform continuity of $\tilde{v}$ up to the boundary $\partial E$. Noting that $\tilde{v}(0) = -14 \delta$, we conclude that $B_{c}(0) \subset E$.  By convexity, both $\tilde{v}$ and $v$ are uniformly continuous and locally Lipschitz in $E$. By noting $v(0)=0$, we obtain that	$B_c(0) \subset  (\left\{v \leq \delta\right\}  \cap\left\{x_n \leq 3 / 4\right\}) $ for some small $c>0$. This concludes the proof of \eqref{eq:E1 normal}.
	
%
By the Blaschke selection theorem, $\{v_k \leq \delta\} \cap \{x_n \leq 3/4\}$ converge in Hausdorff distance to a closed convex set  $Z_\infty$. The uniform  modulus of continuity of $v_k + 16\delta(x_n-7/8)$ yields a subsequence $v_k \to v_\infty$ uniformly on $Z_\infty$. Then, $Z_{\infty}=\left\{x_n \leq 3/4\right\}\cap\{v_\infty\le\delta\}$, and  $v_\infty$  satisfies iii) in Definition \ref{def:normalized p obp}.

Furthermore, let $K_{i}=(\left\{ v_i=0\right\} \cap \left\{x_n \leq 3/4\right\})$. Then subject to a subsequence, $K_{i}$ converges to a nonempty closed set $K_{\infty}$ in the  Hausdorff distance topology. Noting that $0 \in K_{\infty} \subset \left\{ v_{\infty} =0\right\} $, the boundary value of $v_\infty$ implies the set $\left\{ v_{\infty} =0\right\}$ must possess an extreme point in $\{x_n<3/4\}$. By applying Lemma \ref{lem:stability of K} to $\Omega=(\{v_\infty<\delta\}\cap \{x_n<3/4\})$, we obtain that $\left\{ v_{\infty} =0\right\}=K_{\infty}$. 
\end{proof}

By utilizing Lemma \ref{lem:exposed set} and employing the normalization arguments in Definition \ref{def:normal type1}, we have
\begin{Proposition}\label{prop:Savin-1}
Suppose $v \in \E_1$. Then $B_c'(0) \subset K_1'\subset B_C'(0)$, and we have
	\[ c|x'|^{\frac{1+\alpha}{\alpha}}\leq \psi\left( x' \right) \leq C|x'|^{1+\alpha} \quad  \text{in }B_c'(0),\]
	where $\alpha \in (0,1)$, and $c$ and $C$ are positive constants depending only on $n$, $q$, $\lambda$ and $\Lambda$.  Consequently,
$\partial K \cap B_c(0)$ is $C^{1,\alpha}$ and strictly convex. 
\end{Proposition}
\begin{proof}
We  claim that there exists a constant $\epsilon(n,q,\lambda,\Lambda)>0$ such that
\begin{equation}\label{eq:universalconvex}
\frac{1+\epsilon}{2}  K_t'\subset   K_{\frac{1}{2}t}' \subset (1-\epsilon)  K_t'
\end{equation} 
for all $t\in(0,1)$, where $	K_t'=S_{t}^{\psi}(0,0)$ is defined in Definition \ref{def:normal type1}.
Notice that \eqref{eq:universalconvex} remains invariant under
the normalization \eqref{eq:normalized sol}. So it suffices to prove it for the case of $t=1$.  

We only prove the left-hand side of \eqref{eq:universalconvex}, since the right-hand side of \eqref{eq:universalconvex} can be proved similarly. Assume by contradiction the left-hand side is false. Thus there exists a sequence $v_k$ for which the distance  between $ \frac{1}{2}  \partial K_{1}'$ and $ \partial K_{\frac{1}{2}}'$ converges to $0$. Then by convexity, the limiting function $v_{\infty}$ contains a line segment in $\partial \left\{ v_{\infty}=0 \right\}$ with one endpoint being $0$. Hence, $0$ is not an exposed point of $\{v_{\infty}=0\}$, which contradicts with Lemma \ref{lem:exposed set} applied to $v_\infty$.  

This compactness argument actually can also show that $B_c'(0) \subset K_1'$. Since $v\in\E_1$, by definition, we have $K_1'\subset B_C'(0) $. From \eqref{eq:universalconvex}, a standard iteration process leads to that for all $m =1,2,\cdots$,
\[ 2^{m \left(\log_2 \left( 1+\epsilon\right)-1\right)} K_t'\subset   K_{2^{-m} t}'  \subset 2^{m\log_2 (1-\epsilon)}  K_t' .\] 
This implies for $\alpha= \inf\left\{\frac{\log_2 \left( 1+\epsilon\right)}{1-\log_2 \left( 1+\epsilon\right)}, \frac{-\log_2 (1-\epsilon)}{1+\log_2 (1-\epsilon)} \right\} \in (0,1)$ that
\[ 
c\left|x'\right|^{\frac{1+\alpha}{\alpha}}\leq \psi\left( x' \right) \leq C\left|x'\right|^{1+\alpha} \quad \text{in }  B_c'(0).
\]
Finally, by examining the normalization of $v$ at other points on $\partial K \cap B_c(0)$, we deduce that $\partial K \cap B_c(0)$ is $C^{1,\alpha}$ and strictly convex.
\end{proof}

Below we introduce two facts that will be used later.
\begin{Lemma}\label{lem:E1 normal}
Suppose $v \in \E_1$.  Then for any $t \leq 1$, and $h:=h(t)=(t |K_t'|)^{\frac{2}{n-q}}$, the convex set $\hat{S}_h:=S_h \cap \left\{x_n \leq t\right\}$ is balanced around $0$, and is comparable to $K_t'\times (-t, t)$ in the sense that
\begin{equation}\label{eq:coincidence set volume 1}
c\left(K_t' \times (-t, t)\right) \subset   \hat{S}_h \subset  C\left(K_t' \times (-t, t)\right).
\end{equation}
Consequently, $h\le Ct$.
\end{Lemma}
\begin{proof}
Due to the invariance of the balanced property and the estimate \eqref{eq:coincidence set volume 1} under our normalization, it suffices to consider the case where $t=h=1$. Then the conclusions follow directly from the estimate  \eqref{eq:E1 normal} and the inclusions $B_c'(0) \subset K_1'\subset B_C'(0)$, and that $h\le v(-Cte_n)\le Ct\|v\|_{Lip}$.
\end{proof}

\begin{Lemma}\label{lem:poly dependent}
We say two positive real-valued variables $a$ and $b$ are polynomially dependent if there exist positive constants $\gamma$, $c$ and $C$ such that $ ca^{\frac{1}{\gamma}}<b< Ca^{\gamma} $. Suppose $v \in \E_1$. Then the quantities $t\in (0,1)$, $h(t)$, $\| \D_t' \| $, and $\| \D_t'^{-1} \|^{-1}$, which are defined in Definition \ref{def:normal type1}, are polynomially dependent on each other, with $\gamma$, $c$ and $C$  depending only on $n$, $q$, $\lambda$ and $\Lambda$.
\end{Lemma}
\begin{proof}
By Proposition   \ref{prop:Savin-1}, we have $ ct^{\frac{1}{1+\alpha}}B_1'(0)\subset K_t' \subset  Ct^{\frac{\alpha}{1+\alpha}}B_{C}'(0)$, which implies that  $  ct^{\frac{1}{1+\alpha}}\leq \| \D_t' \| \leq Ct^{\frac{\alpha}{1+\alpha}}$, $ ct^{\frac{1}{1+\alpha}} \leq \| \D_t'^{-1} \|^{-1} \leq Ct^{\frac{\alpha}{1+\alpha}}$. Then  $ ct^{\frac{2(n+\alpha)}{(n-q)(1+\alpha)}}\le h(t)\le C t^{\frac{2(1+n\alpha)}{n-q}}  $. Hence, the conclusion follows.
\end{proof}

\section{$C^{1,\alpha}$ regularity of the solution}\label{sec:solutionregularity1}

Assuming  $0<\lambda \leq \det D^2 w \leq \Lambda$, Caffarelli \cite{caffarelli1990ilocalization,caffarelli1993note} showed that any convex function $w$ is strictly convex when the dimension $n=2$; and when $n \geq 3$,  the set of non-strictly convex points of $w$ is a closed set, consists of convex subsets $E $ satisfying $w$ is linear on $E$, all the extreme points of $E$ lie on the boundary, and $1\leq \dim E < \frac{n}{2}$.  The  interior $C^{1,\gamma}$ and Schauder  regularity theory for strictly convex solutions to Monge-Amp\`ere equations are then established in \cite{caffarelli1990ilocalization,caffarelli1990interiorw2p,caffarelli1991regularity}.

Similarly, given a convex function $w$ on the open convex set $\Omega$, we define $\Sigma_w$ as the union of all convex sets $E\subset\Omega$ such that
\begin{itemize}
\item $w$ is linear on $E$;
\item $E^{ext} =\emptyset$ (so all extreme points of $\overline E$ lie on $\partial \Omega$).
\end{itemize}

We first study the $C^1$ regularity of the solution $v$ to \eqref{eq:obs equation} away from $\Sigma_{v}$.

\begin{Proposition}\label{thm:nsc set}	
Suppose $v$ is a solution to \eqref{eq:obs equation}. Then, the set of non-strictly convex points of $v$ in $\Omega\setminus K$ is the set $\Sigma_{v}\setminus K$, and
\[
\Gamma_{nsc}= \Sigma_{v} \cap \Gamma.
\] 
Furthermore,  $\Sigma_v$ is closed (relative to $\Omega$), and $v\in C^1(\Omega \setminus \Sigma_{v})$.
\end{Proposition}

The proof of Proposition \ref{thm:nsc set}	 will be presented after the proof of Lemma \ref{lem:c1 around ek}. 
 
\begin{Lemma}\label{lem:volume Shc h>0}
Suppose $v(x_0) >0$. Then for any centered section $\widetilde{S}_h:= \widetilde{S}_h(x_0)\subset \subset \Omega$ with $0<h<v(x_0)$, we have
\[
c(n,\Lambda)v(x_0)^{-\frac{q}{2}}h^{\frac{n}{2}} \leq |\widetilde{S}_h| \leq C(n,q,\lambda)v(x_0)^{-\frac{q}{2}}h^{\frac{n}{2}}
\]
and
\[
c(n,q,\lambda)v(x_0)^{\frac{q}{2}}h^{\frac{n}{2}} \leq |\M v(\widetilde{S}_h)| \leq C(n,q,\lambda,\Lambda)v(x_0)^{\frac{q}{2}}h^{\frac{n}{2}}.
\]
\end{Lemma}
\begin{proof} 
From \eqref{eq:centerS right}, we have $| \M v (\widetilde{S}_h)|\leq \Lambda \int_{\widetilde{S}_h}v^q \ud x \leq C(n,\Lambda) v(x_0)^q | \widetilde{S}_h|$, which, combined with Lemma \ref{lem:measure constrain}, leads to
\[
| \widetilde{S}_h| \geq c(n,\Lambda)  \sqrt{v(x_0)^{-q}| \M v (\widetilde{S}_h)| \cdot |\widetilde{S}_h| }\geq c(n,\Lambda)v(x_0)^{-\frac{q}{2}}h^{\frac{n}{2}}.
\]
	
After a rotation, we may assume that $v(x) \geq v(x_0)$ on $\left\{ x_n \geq (x_0)_n\right\}\cap\Omega$.  Let $O:= \widetilde{S}_h \cap \left\{ x_n \geq (x_0)_n\right\}$.  By the balanced property of $\widetilde{S}_h$, we have $|O| \geq C(n)|\widetilde{S}_h|$. Suppose $\widetilde{S}_h=\left\{x: \; v(x)<\ell_h(x)\right\}$, where $\ell_h(x)=v(x_0)+p_h\cdot (x-x_0)+h$. By \eqref{eq:centerS nondege}, the convex function $\tilde{v} =v(x_0)^{-\frac{q}{n}} \left(v-\ell_{h}+2C(n)h\right)$ satisfies 
\[
0 <  \tilde{v} \leq 2C(n)v(x_0)^{-\frac{q}{n}} h   \quad \text{and } \quad  \det D^2  \tilde{v} \geq \lambda \quad   \text{on }O.
\]
Applying Lemma  \ref{lem:volume concidence} to the function $ \tilde{v} $ for the case $q=0$, we obtain that 
\[
|\widetilde{S}_h| \leq c(n)|O|   \leq   C(n,\lambda)\left\| \tilde{v}\right\|_{L^{\infty}}^{\frac{n}{2}} \leq C(n,q,\lambda)v(x_0)^{-\frac{q}{2}}h^{\frac{n}{2}}.
\]
Then the conclusions follow from  Lemma \ref{lem:measure constrain}.	
\end{proof}

\begin{Lemma}\label{lem:balanced h>0}
Assume $x_0 \notin K$, $0<h<v(x_0)$, and the centered section 
\[
\widetilde{S}_h(x_0)=\left\{x: \; v(x)<\ell_h(x)\right\} \subset \subset \Omega.
\]
 Then for any $y$ satisfying  $v(y) -\ell_h(y) \leq -h$, $\widetilde{S}_h(x_0)$ is $ c(n,q,\lambda,\Lambda)$-balanced around $y$.
\end{Lemma}
\begin{proof}
After a translation, we assume $x_0=0$. Using Lemma \ref{lem:volume Shc h>0}, we can apply John's lemma to find a diagonal transformation $\D_h$ with $ v^{\frac{q}{2}}(0)\det \D_h=h^{\frac{n}{2}}$
such that $\D_h B_c(0)  \subset    \widetilde{S}_h \subset   \D_hB_C(0)$.
Then, we introduce the following normalization  of $v$ at $0$ for $\widetilde{S}_h $,
\[
v_{h} (x) = \frac{ v\left( \D_h x\right)- \ell_h(\D_hx)}{h} ,\quad x \in \D_h^{-1} \widetilde{S}_h.
\]
Using \eqref{eq:centerS right} and \eqref{eq:centerS nondege}, we obtain that 
\[
\det D^2 v_{h}  \leq  \Lambda\frac{v^q( \D_h x)}{v^q(0)} \leq \Lambda \frac{\left(2C(n)v(0)\right)^q}{v^q(0)} \leq C  \quad  \text{in }\D_h^{-1} \widetilde{S}_h.
\]
Note that $v_{h}=0$ on $\partial \D_h^{-1} \widetilde{S}_h$. Thus, for any $y$ satisfying $v_{h}(\D_h^{-1} y) =\frac{v(y) -\ell_h(y)}{h} \leq  -1$, we can apply Lemma  \ref{lem:measure constrain 2} to obtain that  $\operatorname{dist}(\D_h^{-1}y, \partial \D_h^{-1}\widetilde{S}_h) \geq c$. Consequently, $\D_h^{-1} \widetilde{S}_h $ is balanced around $\D_h^{-1}  y$, and thus, $\widetilde{S}_h $ is balanced around $ y$.
\end{proof}

\begin{Lemma}\label{lem:c1 at exposed}
$v$ is differentiable at every point on $\Gamma_{sc}$. 
\end{Lemma}
\begin{proof}
Assume the lemma does not hold. Without loss of generality, we assume $0\in\Gamma_{sc}$ and $v$ is not differentiable at $0$. Then $a =\sup\left\{|p|:\; p\in \partial v(0) \right\}>0$. After a rotation and an affine transformation, we further assume that $ae_n \in \partial v(0)$, so that $K \subset \left\{ x_n \leq 0\right\}$. By using the convexity of $v$ and the maximality of $a$, we have
\[
v(0)=0, \quad  v \geq a \max\left\{x_n,0\right\} ,\quad  v\left(t e_n\right)=a t+o(t) \text{ for } t\geq 0.
\]
For $b \in (0,a)$ close to $a$, and $s>0$ small, there holds
\[
E_s:=S_{s}^v(0,b) =\left\{ x :\; v(x)<s+bx_n\right\} \subset \left\{-sb^{-1} \leq  x_n \leq s(a-b)^{-1} \right\}.
\]
Since $ 0\in \Gamma_{sc}=K^{exp}$, by taking $s$ sufficiently small, we have
\[
E_s \subset  \left(\left\{-sb^{-1} \leq  x_n \leq s(a-b)^{-1} \right\}  \cap \left\{ v(x) \leq 2sa(a-b)^{-1} \right\}\right) \subset\subset \Omega. 
\]
Thus, we can write $E_s= \widetilde{S}_h(x_h)=\left\{ v(x) < \ell_{h}(x)\right\}$, where $x_h$ denotes the mass center of $E_s$, $\ell_{h}=s+bx_n$, and  $h=\ell_h(x_h)-v(x_h)$. Notably, the function $v(x)-\ell_h(x)$  achieves its minimum value $-s$ at $0$, and we have $h \leq s$. 
	
By observing that $v\left(t e_n\right)=a t+o(t)$, we find for $s$ small that  $\frac{1}{2}s(a-b)^{-1} e_n \subset \widetilde{S}_h(x_h) $.  Since $b$ is close to $a $, the balanced property of  $\widetilde{S}_h(x_h)$ around $x_h$ yields that $x_h \cdot e_n \geq c(n) s(a-b)^{-1}$, thus $v(x_h) \geq  ax_h \cdot e_n \geq s \geq h$. Then, we can apply Lemma \ref{lem:balanced h>0} to conclude that $\widetilde{S}_h(x_h)$ is balanced around $0$, which contradicts with the fact $ \frac{1}{2}s(a-b)^{-1} e_n \in  \widetilde{S}_h(x_h) \subset \left\{ x_n  \geq -sb^{-1} \right\}$ for $b$ close to $a$.
\end{proof}

\begin{Lemma}\label{lem:Eext=emptyset}
Suppose $\ell \leq v$ is a linear function with $\ell \not\equiv0$. Then for every non-trivial convex set $E \subset\left\{ v=\ell \right\}$, we have $E^{ext}\cap \Gamma_{sc}=\emptyset$, and $E \subset \Sigma_{v}$. Consequently, it follows that $\Gamma_{sc} \cap \Sigma_v =\emptyset$.  
\end{Lemma}
\begin{proof}
Suppose there exists $x_0\in E^{ext}\cap \Gamma_{sc}$. Then, by Lemma \ref{lem:c1 at exposed}, we have $\nabla \ell=\nabla v(x_0)=0$. Since $\ell(x_0)=v(x_0)=0$,  it follows that $\ell\equiv 0$, which is a contradiction. Therefore, $E^{ext}\cap \Gamma_{sc}=\emptyset$.
	
Let $F:=\left\{ v=\ell \right\}\cap\Omega$. We want to show $F^{ext} =\emptyset$, from which it follows that $E\subset F\subset \Sigma_v$. First, the classical regularity theory in \cite{caffarelli1990ilocalization} implies that $F^{ext} \subset K $. Second, if $F\cap K\neq\emptyset$, then since $\ell \not\equiv0$, we know $F\cap K$ is an exposed face of both $F$ and $K$. Hence, 
\[
(F^{ext}\cap K)=(F\cap K)^{ext} =(F\cap  K^{ext}) \subset\Gamma_{sc}.
\]
Therefore, $F^{ext} \subset (F^{ext} \cap \Gamma_{sc}) =\emptyset$. 
	
Suppose $\Gamma_{sc} \cap \Sigma_v \neq \emptyset$. Then, there exists $x_0\in \Gamma_{sc} \cap \Sigma_v$. By the definition of $\Sigma_v$, there exists a linear function $\ell_0$ and a convex set $E_0 \subset \left\{ v=\ell_0 \right\}$ such that $x_0\in E_0$ and $(E_0)^{ext} =\emptyset$. By considering the supporting hyperplane of $v$ at the mass center of $E_0$, we may further assume that this hyperplane is given by the linear function $\ell_0$, so that $\ell_0 \leq v$.  If $\ell_0 \equiv0$, then $E_0\subset K$, and  $x_0 \in (\Gamma_{sc} \cap E_0) = (K^{ext} \cap E_0) \subset (E_0)^{ext}=\emptyset$, this is impossible. Hence, $\ell_0\not\equiv0 $. Let $F_0=\left\{ v=\ell_0 \right\}$, then we have $x_0 \in (F_0\cap K^{ext}) =(F_0^{ext}\cap K) =\emptyset$, which is also impossible.
\end{proof}

\begin{Lemma}\label{lem:c1 around ek}
For every $x\in \Gamma_{sc}$, there exists an open neighborhood $O$ of $x$ such that $v$ is strictly convex in $O \setminus K$, and $C^1$  on $O$.  More specifically, 
\begin{itemize}
\item when $K=\left\{0\right\}$ and $S_h:=\left\{ v(x) < h\right\} \subset \subset \Omega$,  we can choose $O=S_h$;
\item when $x=0$ and $v\in \E_1$, we can choose $O=\left\{v(x)< \delta\right\} \cap \left\{ x_n < \delta\right\}$ for $\delta(n,q,\lambda,\Lambda)>0$ small.
\end{itemize}
\end{Lemma}
\begin{proof}
Suppose $0\in \Gamma_{sc}$. We first show that there exists an open neighborhood $O$ of $0$ such that $v$ is strictly convex in $O \setminus K$. We split into the following two cases. 
	 
In the first case, we consider that $|K|=0$. Then by Lemma \ref{lem:ext are exp},  $K=\left\{0\right\}$. We will show that $v$ is strictly convex in $S_h$ as long as $S_h:=\left\{ v(x) < h\right\} \subset \subset \Omega$.  Indeed, suppose $v$ is not strictly convex at $x_0 \in S_h$. Then, there exists a supporting function $\ell$ of $v$ at $x_0$ such that  $\left\{ x_0\right\} \subsetneq \overline{\left\{ v=\ell \right\}\cap S_h}:=E$. By  \cite{caffarelli1990ilocalization}, we have $E^{ext} \subset (\{0\} \cup \partial S_h)$. Since, by Lemma \ref{lem:Eext=emptyset}, $ 0 \notin E^{ext}$, and thus, $E^{ext} \subset \partial S_h$. Hence, $v=\ell\ge h$ on $E$. This is a contradiction with $x_0\in E$.

In the second case, we assume that   $|K|>0$. After an appropriate affine transformation, and by considering the normalization in \eqref{eq:normalized sol}, we may assume for simplicity that  $v\in \E_1$. We will show for $\delta(n,q,\lambda,\Lambda)>0$ small, $v$ is strictly convex in $F:=\left\{0 <v(x)< \delta\right\} \cap \left\{ x_n \leq  \delta\right\}$. Indeed, if $v$ is not strictly convex at $x_0 \in F$, then there exists a supporting function $\ell$ of $v$ at $x_0$ such that  $\left\{ x_0\right\} \subsetneq \overline{\left\{ v=\ell \right\}\cap \{0<v<\delta\}\cap\{x_n\le 1\}}:=E$. Since $\ell\not\equiv 0$, then $E$ is convex.  
In $ \left\{x_n \leq 1 \right\}$, by  \cite{caffarelli1990ilocalization}, we have $E^{ext} \subset \partial \left(\left\{0 <v(x)< \delta\right\} \cap \left\{x_n \leq 1 \right\}\right)$. By the argument in the first case, we know $E^{ext}\not\subset \{v=\delta\}$. Then there must be a point $ \tilde{x} \in  E^{ext} $ such that $v(\tilde x)=\inf_{E}v<\delta$.  Applying Lemma \ref{lem:exposed set} and Lemma \ref{lem:c1 at exposed}, we find that $E^{ext} \cap K \cap \left\{ x_n < 1\right\} =\emptyset$. Therefore, we have $\tilde{x} \in \left\{ x_n =1\right\}$, and the line segment $L$ connecting $x_0$ and $\tilde{x}$ is in $F$. By Proposition \ref{prop:Savin-1}, we know that
\[
E_1:=\left\{ x_n \geq C|x'|^{1+\alpha} \right\} \subset K  \subset  \left\{ x_n \geq c|x'|^{\frac{1+\alpha}{\alpha}} \right\}  \quad  \text{in }B_c (0).
\]
Thus, for small $\epsilon > 0$, we have that
\begin{align*}
\left\{x: \operatorname{dist}(x,K) \leq  \epsilon, x_n \leq \frac{1}{2}\right\}  & \subset \left\{ x: c \max\left\{|x'|-C(n)\epsilon,0 \right\}^{\frac{1+\alpha}{\alpha}} \leq  x_n +C(n)\epsilon\right\} \\
&=: E_{2,\epsilon}.
\end{align*}
From \eqref{eq:bound from below}, we notice that $\left(\left\{ v <\delta \right\}  \cap B_c(0)\right) \subset \left\{x:\; \operatorname{dist}(x,K) \leq  C\delta^{\frac{n-q}{2n}}\right\} $. Therefore,
\[
L \subset \left\{ 0< v< \delta\right\} \subset  E_{2,C\delta^{\frac{n-q}{2n}}} \setminus E_1.
\] 
However, if $\delta(n,q,\lambda,\Lambda)$ is sufficiently small,  $L$ must intersect $E_1$, which is a contradiction. This proves that there exists an open neighborhood $O$ of $0$ such that $v$ is strictly convex in $O \setminus K$. 

Notice that $\lambda \leq \det D^2  \leq \Lambda $ outside $K$. By  \cite{caffarelli1990ilocalization}, we then conclude that $v\in C^1(O\setminus K)$.  We know that $v$ is $C^1$ in the interior of $K$.  Finally, combining Lemma \ref{lem:exposed set} and Lemma \ref{lem:c1 at exposed}, we can conclude that $v\in C^1$ near $0$.
\end{proof}

\begin{proof}[Proof of Proposition \ref{thm:nsc set}]
By definition, every point in $\Sigma_{v}$ is a non-strictly convex point of $v$. Let $x_0 \in \Omega \setminus K$ be a non-strictly convex point. Then, there exists $p\in \partial v(x_0)$ such that $\{x_0\} \subsetneq E:=\left\{ v(x)=v(x_0)+p\cdot x \right\}$. Clearly, $v(x_0)+p\cdot x \not \equiv 0$. Then, by Lemma \ref{lem:Eext=emptyset}, we have $x_0\in E \subset \Sigma_{v}$. In conclusion, the set of non-strictly convex points of $v$ in $\Omega\setminus K$ is the set $\Sigma_{v} \setminus K$.

It follows from Lemma \ref{lem:nsc property} that $\Gamma_{nsc} \subset (\Sigma_{v} \cap \Gamma)$. To prove the other direction, we let $x_0\in \Sigma_{v} \cap \Gamma$. Let $\ell$ be a linear function so that $\{x_0\}\subsetneq E:=\{v=\ell\}\cap\Omega$ and $\ell(x_0)=v(x_0)=0$. If $\ell\not\equiv 0$, then it follows from Lemma \ref{lem:Eext=emptyset} that $x_0\in\Gamma_{nsc}$. If $\ell\equiv 0$, then $K=E\subset \Sigma_v$ and $E^{ext}=\emptyset$. Hence, $x_0\in\Gamma_{nsc}$. 
	
Next, we show that $\Sigma_v$ is closed (relative to $\Omega$). Indeed, if $\Sigma_v$ is not closed, then there exists $x_0 \notin \Sigma_v$ and a sequence of points $x_k\in \Sigma_v$ that converges to $x_0$.  
By \cite{caffarelli1990ilocalization}, $\Sigma_v$ is closed in $\Omega\setminus K$, so $x_0 \in K$.  If $x_0 \in \mathring{K}$, we may assume that $x_k \in \Sigma_{v} \cap  \mathring{K}   $. Hence, we can select $n+1$ ray segments $L_{k,1}, \cdots, L_{k,n+1} \subset K$   starting from $x_k$ such that $E_k:=\operatorname{conv}\left\{ L_{k,1} \cap \Omega, \cdots,  L_{k, n+1}\cap \Omega  \right\}$ satisfies $(E_k)^{ext} =\emptyset$ (note that some ray segments could be repeated).  Subject to a subsequence,  $\left\{ L_{k,1}, \cdots,  L_{k,n+1}\right\}$ will converge to some $\left\{ L_{0,1} , \cdots,  L_{0,n+1}  \right\}$. Let $E_0=\operatorname{conv}\left\{ L_{0,1} \cap \Omega, \cdots,  L_{0,n+1}\cap \Omega  \right\}$. Then, $(E_0)^{ext}  =\emptyset$. Since $v$ is continuous, then $v\equiv 0$ on $E_0$.  Therefore, $x_0 \in E_0 \subset \Sigma_{v} $. This contradicts with the fact that $x_0\not\in \Sigma_{v}$.  Thus, $x_0 \in( \partial K\setminus \Sigma_{v}) =\Gamma_{sc}$. Then, we can apply Lemma \ref{lem:exposed set} and  Lemma \ref{lem:c1 around ek} to conclude that $x_k \in \mathring{K}$ for all $k$ large. Then the same argument above still implies that $x_0 \in \Sigma_v$, which is impossible. This concludes that $\Sigma_v$ is closed (relative to $\Omega$).

The classical regularity theory in \cite{caffarelli1990ilocalization} indicates $v \in C^1\left(\Omega \setminus \left( \Sigma_{v} \cup K\right) \right)$. Clearly, $v$ is $C^1$ on $\mathring{K}$. By Lemma \ref{lem:c1 at exposed}, we also have $v$ is $C^1$ on $\Gamma_{sc}$. In conclusion, $v \in C^1(\Omega \setminus \Sigma_{v})$.
\end{proof}

Subsequently, we shall proceed to establish the $C^{1,\alpha}$ regularity of $v$ around $\Gamma_{sc}$ when $|K|>0$.  The proof follows from the standard normalization methods and the $C^1$ regularity of $v$, in a similar way of proving  \eqref{eq:universalconvex} in  Proposition \ref{prop:Savin-1}. 
\begin{Lemma}\label{lem:c1a v pointwise}
Suppose $v\in \E_1$. Then 
there exists $\alpha(n,q,\lambda,\Lambda) \in (0,1)$  such that \eqref{eq:v distance alpha} holds and
\begin{equation}\label{eq:v distance alpha 1}
c\operatorname{dist}(x,K)^{\frac{1}{\alpha}} \leq |\nabla v(x)|\leq C\operatorname{dist}(x,K)^{\alpha}  \quad \text{in }B_{c}(0).
\end{equation}  
\end{Lemma}
\begin{proof} 
Let $\R^-:=\left\{ te_n:\; t\leq 0\right\}$. We claim that there exist constants $\epsilon(n,q,\lambda,\Lambda)>0$ and  $\delta=\delta(n,q,\lambda,\Lambda)>0$ such that
\begin{equation}\label{eq:universalconvex 1}
\frac{1+\epsilon}{2} \left(S_h\cap \R^-\right)
\subset 
\left(  S_{\frac{1}{2}h}\cap \R^-\right)\subset (1-\epsilon) \left(S_h\cap \R^-\right)
\end{equation} 
for all $h\in(0,\delta)$ and $v \in \E_1$. Notice that \eqref{eq:universalconvex 1} remains invariant under
the normalization \eqref{eq:normalized sol}. So it suffices to prove it for the case of $h=\delta$. 
	
We first prove the left-hand side of \eqref{eq:universalconvex 1}. Assume by contradiction that there exists a sequence $\epsilon_k\to 0 $ and $v_k$ for which the left-hand side is false for $\epsilon_k$.  By Lemma \ref{lem:compactness of E1} and  convexity, the limiting function $v_{\infty}$ is linear on $S_\delta^{v_\infty}\cap\R^-$, and $0$ is still an exposed point of $\{v_{\infty}=0\}$.  Let $\ell$ be the support function of $v_{\infty}$ at $\left\{ v_\infty =\frac{\delta}{2}\right\} \cap \R^-$, then we can apply Lemma \ref{lem:Eext=emptyset} to conclude that $\ell \equiv 0$, which is impossible.
	
The right-hand side of \eqref{eq:universalconvex 1} can be proved similarly. Assume by contradiction there exists a sequence $\epsilon_k\to 0 $ and $v_k$ for which the right-hand side is false for $\epsilon_k$ at $h =\delta$.  By Lemma \ref{lem:compactness of E1} and convexity, the limiting function $v_{\infty}$ is discontinuous on $\left(S_{3\delta/4}\cap \R^-\right) $. But by Lemma \ref{lem:compactness of E1},   $v_{\infty}$ is always uniformly continuous in $\left\{v_\infty \leq \delta\right\}  \cap\left\{x_n \leq 3 / 4\right\}$, which is impossible.

From \eqref{eq:universalconvex 1}, a standard iteration process leads to that for all $m =1,2,\cdots$,
\[
2^{m \left(\log_2 \left( 1+\epsilon\right)-1\right)} \left({S}_{\delta} \cap \R^-\right) \subset   \left({S}_{2^{-m}\delta}\cap \R^-\right)\subset 2^{m\log_2 (1-\epsilon)}  \left(S_{\delta}\cap \R^-\right).
\]
This implies 
\[
c\operatorname{dist}(x,K)^{\frac{1+\alpha}{\alpha}} \leq v \leq C\operatorname{dist}(x,K)^{1+\alpha}  \quad \text{in } B_c(0) \cap \R^-
\]
for $\alpha= \inf\left\{\frac{\log_2 \left( 1+\epsilon\right)}{1-\log_2 \left( 1+\epsilon\right)}, \frac{-\log_2 (1-\epsilon)}{1+\log_2 (1-\epsilon)} \right\} \in (0,1)$. 
Then, by invoking Proposition  \ref{prop:Savin-1} and performing an appropriate rotation, we can consider the normalization of $v$ at other points on $\partial K \cap B_c(0)$ and deduce \eqref{eq:v distance alpha}.
By convexity, this estimate implies  $c|x|^{\frac{1}{\alpha}} \leq |\nabla v|\leq C|x|^{\alpha}$ in $B_c(0)  \cap \R^-$.  By considering the normalization of $v$ at other points on $\partial K \cap B_c(0) $ again, we obtain \eqref{eq:v distance alpha 1}.
\end{proof}

\begin{proof}[Proof of Theorem  \ref{thm:c1a vK K>0}]
From  Assumption (H) and Lemma \ref{def:normal type1}, we conclude that $v$ belongs to the class $\E_1$ after appropriately modifying the universal constants $c$ and $C$. 
In Proposition \ref{prop:Savin-1}, we establish the regularity of the free boundary.
In Lemma \ref{lem:c1a v pointwise}, we obtain the estimate \eqref{eq:v distance alpha}. 
In Lemma \ref{lem:c1 around ek}, we show that $v$ is strictly convex on $(S_{\delta}\cap \left\{x_n \leq \delta\right\}) \setminus K$. This, together with the estimate  \eqref{eq:v distance alpha},  implies that $v$ is strictly convex in $B_c(0)\setminus K$ for some small $c>0$.

Therefore, it remains to show that there exist positive constants $\alpha\in(0,1)$,   $c$ and $C$, all of which depend only on $n$, $q$, $\lambda$ and $\Lambda$, such that 
\begin{equation}\label{eq:c1a vK K>0}
|\nabla v(x) -\nabla v(y)|\leq C| x-y|^{\alpha},\quad \forall\ x,y \in B_c(0).
\end{equation}
Without loss of generality, we can assume for simplicity $x=se_n$ for some $s >0$ and $ v(y) \leq v(se_n) \le \delta$ for $\delta=\delta (n,q,\lambda,\Lambda)>0$. Let $a=v(se_n)/\delta\in [0,1]$. 

First, if $a=0$, then $x=se_n\in K$ and $y\in K$. Since $v$ is $C^1$ around $ K \cap  \left\{ x_n <1\right\} $, $\nabla v(x)=\nabla(y)=0$, and thus, \eqref{eq:c1a vK K>0} holds.

Second, let us assume $a=1$. We will show \eqref{eq:c1a vK K>0}  holds for some $\alpha=\gamma_0(n,q,\lambda,\Lambda)>0$.
For any fixed  $v \in \E_1$, by recalling the strict convexity of $v$ in $(S_{\delta} \setminus K)  \cap \left\{x_n \leq 0\right\} $ as indicated in  Lemma \ref{lem:c1 around ek}, we can define 
$$h_0:=\sup\{h>0: S_h(x) \subset \subset \left(S_1 \setminus S_{\delta/2 }\right)\}>0.$$
Noting that $h_0$ is lower semicontinuous with respect to uniform convergence of functions. Then  we can use Lemma \ref{lem:compactness of E1} and employ compactness arguments to demonstrate that $h_{0}>c_0(n,q,\lambda,\Lambda,\delta)>0$. By virtue of the Lipschitz regularity of  $v$, we have $B_{ch_0}(x) \subset S_{h_0/2}(x)$.   Noting that  $\Lambda> \det D^2 v \geq \lambda v^q> (\delta/2)^q>0$ on $ S_1 \setminus S_{\delta/2}$, the classical regularity theory in \cite{caffarelli1990ilocalization} then yields the $C^{1,\alpha}$ norm of $v$ in $S_{h_0/2}(x)$, thereby concluding the proof.

Finally, we are going to prove \eqref{eq:c1a vK K>0}. Recalling \eqref{eq:v distance alpha} and \eqref{eq:v distance alpha 1}, we always have $|\nabla v| \leq C |v|^{\gamma_0}$ holds for some $\gamma_0(n,q,\lambda,\Lambda)>0$.
If $|x-y| >  a^2$, then we have 
\[
|\nabla v(x) -\nabla v(y)| \leq 2Ca^{\gamma_0} \leq 2C|x-y|^{\gamma_0/2}.
\]
If $|x-y| <a^2$, then we consider the normalization $v_a$ of $v$ as defined in Definition  \ref{def:normal type1}.  Denote $\T_a^{-1} x= \left(\D_t'^{-1}x', t^{-1}x_n\right)$ and $\T_a^{-1} y= \left(\D_t'^{-1} y', t^{-1} y_n\right)$, then we have
\[|\nabla v_a(\T_a^{-1}x) -\nabla v_a(\T_a^{-1} y)|  \leq  C|\T_a^{-1}x -\T_a^{-1} y|^{\gamma_0}.\]
By Lemma \ref{lem:E1 normal}, $\|\T_a^{-1}\|\leq ct^{-1}\le ca^{-1}$, and thus 
\[
|\nabla v(x) -\nabla v(y)|   \leq  C|\nabla v_a(\T_a^{-1}x) -\nabla v_a(\T_a^{-1} y)|.
\]
Then we obtain
\[\begin{split}
|\nabla v(x) -\nabla v(y)|  \leq C |\T_a^{-1}x -\T_a^{-1} y|^{\gamma_0} \leq C a^{-\gamma_0}|x-y|^{\gamma_0 } \leq  C |x-y|^{\frac{\gamma_0}{2}}.
\end{split}
\] 
\end{proof}

\begin{proof} [Proof of Theorem \ref{thm:c1a v K=0}]
We  only provide a sketched proof, as it can be derived using the same methodology as that in the case of $|K| > 0$.

First, we introduce the following normalized model for the case of  $K=\left\{0\right\}$,  consisting of non-negative convex functions $w$ that satisfy
\[
\lambda w^q \chi_{\{ w > 0\}} \ud x \leq \M w\leq \Lambda w^q \chi_{\{ w > 0\}} \ud x ,  \quad \left\{ w=0\right\}=0 , 
\] 
and
\[ 
w=1  \quad \text{on }\partial \Omega_w, \quad B_c(0) \subset \Omega_w\subset B_C(0) .  
\]	

Then, we discuss the corresponding compactness for these normalized solutions. Suppose $w_{i} \geq 0$, $i=1,2,\cdots$,  are normalized solutions for the case of  $K=\left\{0\right\}$.	Applying Lemma \ref{lem:measure constrain 2},  we obtain the uniform continuity modulo of $w_{i}$  up to the boundary of $ \Omega_i =\left\{w_i\leq 1\right\}$, and a subsequence of $w_{i}$ uniformly converging to some $w_\infty$. To demonstrate that $w_{\infty} $ is also a normalized solution, it remains to verify that  $\left\{w_{\infty} = 0\right\}=\left\{0\right\}$.  By observing that the set  $\left\{w_{\infty} = 0\right\}$ is compact, it always admits extreme points. Thus, Lemma \ref{lem:stability of K} implies $\left\{w_{\infty} = 0\right\}=\left\{0\right\}$.

Let $v$ be a convex solution of \eqref{eq:obs equation}. Suppose $K=\left\{0\right\}$, and $S_h \subset \subset \Omega$. Now, we can use Lemma \ref{lem:volume of Sh K=0} and John's lemma to find a diagonal transformation $\D_h$ such that $\D_h B_c(0) \subset   S_h \subset \D_h B_C(0)$ with $\det \D_h=h^{\frac{n-q}{2}} $.   Then, the following normalization $v_h$ of $v$ at $0 $ for $h\in(0,1)$,
\[
v_{h} (x) = \frac{ v\left( \D_hx\right)}{h}
\]
is a normalized solution for the case of  $K=\left\{0\right\}$.
Therefore, for the case of $K=\left\{0\right\}$, we can  always presume, for simplicity, that $h=1$ and $S_1\subset\subset \Omega$ is normalized such that $B_c(0) \subset S_1 \subset B_C(0)$. 

Hence, we can using the compactness arguments again, together with the $C^1$ regularity (Lemma \ref{lem:c1 around ek}), to show that 
\[
\frac{1+\epsilon}{2} S_h
\subset S_{\frac{1}{2}h}\subset (1-\epsilon) S_h ,\quad \forall\ h \in(0,1)
\]
holds for some $\epsilon(n,q,\lambda,\Lambda)>0$. By iterations, this implies \eqref{eq:v distance alpha 0}  and \eqref{eq:v distance alpha 1}. It also indicates a polynomial dependence among $h$, $\| \D_h \| $, and $\| \D_h^{-1} \|^{-1}$. Therefore, by employing the same argument as in the proof of Theorem \ref{thm:c1a vK K>0}, we obtain the $ C^{1,\alpha}$ regularity of $v$ in $S_{1/2}$.
\end{proof}

In Proposition \ref{prop:infinite ray}, we present several sufficient conditions for the non-existence of the set $\Gamma_{nsc}$.  Next, we establish a sufficient condition for $v$ to be $C^{1,\alpha}$ in $\Omega$ as follows.

\begin{Corollary}
Let $v$ be a solution of \eqref{eq:obs equation} with $\Gamma_{nsc}=\emptyset$. If  $\partial \Omega \in C^{1,\gamma}$ and $\varphi \in C^{1,\gamma}$ on $ \partial \Omega$ for some $\gamma >\frac{n-2}{n}$, then we have  $v \in C_{loc}^{1,\alpha}(\Omega)$, where $\alpha \in (0,1)$  depending only on $n$, $q$, $\lambda$ and $\Lambda$.
\end{Corollary} 
\begin{proof}
First, we assert that $\Sigma_{v}\setminus K=\emptyset$. Assuming the contrary, there exists a support function $\ell\not\equiv 0$ of $v$ such that the relative closed set $E:=\left\{x\in \Omega:\; v(x)=\ell(x)\right\}\subset \Sigma_v$ satisfies $(\overline E)^{ext} \in \partial \Omega$. Since $\partial \Omega \in C^{1,\gamma}$ and $\varphi \in C^{1,\gamma}$ on $ \partial \Omega$ for some $\gamma >\frac{n-2}{n}$, by \cite[Corollary 4]{caffarelli1993note} (or \cite[Corollary 3.6]{jin2016solutions}), we have $(\overline E)^{ext}\subset \overline K$. Hence $\ell\equiv 0$, which is a contradiction. 
	 
Then, by \cite{caffarelli1990ilocalization}, we conclude that $v\in C_{loc}^{1,\alpha}(\Omega\setminus K)$.  Clearly, $v=0$ is $C^{1,\alpha}$ in $\mathring{K}$.  Finally, since $\Gamma_{nsc}=\emptyset$, by using  Theorem \ref{thm:c1a vK K>0}   and Theorem \ref{thm:c1a v K=0}, we conclude that $v\in C_{loc}^{1,\alpha}(\Omega)$.
\end{proof}

\section{$C^{1,1}$ regularity and uniform convexity of the free boundary} \label{sec:c11}

In this section, we study the case $g \equiv 1$, in which $v$ is a solution of 
\begin{equation}\label{eq:obs eq g=1}
\det D^2 v =v^q\chi_{\{v>0\}}, \quad v\ge 0, \quad \text{in } \Omega.
\end{equation}
After a suitable affine transformation, we will always assume that $v\in \E_{1}$ is a solution to \eqref{eq:obs eq g=1} in this section.

Following \cite[Section 4]{savin2005obstacle} or \cite[Section 2]{huang2024regularity}, we can show the $C^{1,1}$ regularity and the uniform convexity of the level set of $v$ around $0$.
\begin{Theorem}\label{thm:c11 bound of K}
Let $n\ge 2$ and $0\le q < n$. Suppose  $v\in \E_{1}$ is a solution to \eqref{eq:obs eq g=1}.
Then, there exist positive constants $c$ small and $C$ large, depending only on $n$ and $q$, such that for each $t \in [0,c] $,  the principal curvatures of the convex hyperspace $\partial\left\{v  \leq t\right\}$ in $B_c(0)$ are between $c$ and $C$.
\end{Theorem}

We will show the upper and lower bound of the principal curvatures in Sections \ref{sec:upper} and \ref{sec:lower}, respectively.

Let us recall the $C^{1,1}$ regularity of the solution $v$ a bit away from the free boundary.
\begin{Lemma}\label{lem:c11 out}
Suppose  $v\in \E_{1}$ is a solution to \eqref{eq:obs eq g=1}. For any $c_0>0$, there exist $c_1$ and $C_1$ depending only on $c_0, n$ and $q$ such that 
\[
c_1 \mathcal{I}_{n} \leq   D^2 v \leq C_1 \mathcal{I}_{n} \quad \text{in } B_c(0) \setminus S_{c_0}.
\]
\end{Lemma}
\begin{proof}
For any fixed  $v \in \E_1$, by recalling the $C^{1,\alpha}$ regularity and the strict convexity of $v$ in $B_c(0)\setminus K$ as indicated in Theorem \ref{thm:c1a vK K>0}, we can define 
\[
h_0(v):= \inf_{x \in B_c(0) \setminus S_{c_0} }\sup\left\{h>0:\; S_h(x) \subset \subset \left(S_1 \setminus S_{c_0/2}\right)\right\}>0.
\]
Noting that $h_0$ is lower semicontinuous with respect to the uniform convergence of functions. Then  we can use Lemma \ref{lem:compactness of E1} and employ compactness arguments to demonstrate that $h_{0}>\delta(n,q,\lambda,\Lambda,c_0)>0$. By virtue of the Lipschitz regularity of  $v$, we have $B_{ch_0}(x) \subset S_{h_0/2}(x) \subset (S_1 \setminus S_{c_0/2} )$ for every $ x\in B_c(0) \setminus S_{c_0}$. Then, the classical regularity theory in \cite{caffarelli1990ilocalization,caffarelli1990interiorw2p,caffarelli1991regularity} yields the Hessian estimate of $v$ in $S_{h_0/2}(x)$.
\end{proof}

Then, combining Theorem \ref{thm:c11 bound of K}, we can obtain certain elliptic regularity structure of $v$ near the free boundary as follows:
Let $K_t'= \left\{ x'  :\;  (x',t) \in K\right\}$ for $t >0$. By Theorem \ref{thm:c11 bound of K}, we have  for  all small ${t_1}, {t_2} >0$ that
\begin{equation}\label{eq:slice relation}
c{t_1}^{-\frac{1}{2}}K_{t_1}'\subset {t_2}^{-\frac{1}{2}}K_{t_2}' \subset C{t_1}^{-\frac{1}{2}}K_{t_1}'.
\end{equation}
Thus, for $t>0$ small, by applying Lemma \ref{lem:E1 normal} with $h=t^{\frac{n+1}{n-q}}$ here, and considering each $x\in\partial S_h\cap B_{c_0}(0)$, we obtain
\begin{equation}\label{eq:growthv}
c\operatorname{dist}(x,K)^{\frac{n+1}{n-q}} \leq v \leq C\operatorname{dist}(x,K)^{\frac{n+1}{n-q}}  \quad  \text{in }B_{c_0}(0) .
\end{equation}
Furthermore, let us redefine the normalization in \eqref{eq:normalized sol} to be 
\begin{equation}\label{eq:normalized sol c11}
v_{h}\left(x', x_n\right):=t^{-\frac{n+1}{n-q}} v\left(t^{\frac{1}{2}} x', t x_n\right), \quad h:=t^{\frac{n+1}{n-q}}.
\end{equation}
Then, after a slight modification of the universal constants, we always have $v_{h} \in \mathcal{E}_{1/h} \subset\mathcal{E}_{1}$.  Thus, by invoking Lemma \ref{lem:c11 out}, we obtain that 
\begin{equation}\label{eq:hessiangrowthv}
c \operatorname{diag}\left\{  t^{\frac{q+1}{n-q}} ,\cdots, t^{\frac{q+1}{n-q}},t^{\frac{2q+1-n}{n-q}} \right\} \leq D^2 v(-te_n) \leq   C \operatorname{diag}\left\{  t^{\frac{q+1}{n-q}} ,\cdots, t^{\frac{q+1}{n-q}},t^{\frac{2q+1-n}{n-q}} \right\}
\end{equation}
for $t>0$ small. A similar Hessian estimate can be derived for other points in  $B_c(0)$ by locating the nearest point and rotating the coordinates. This $C^{1,1}$ regularity \eqref{eq:hessiangrowthv} and the uniform convexity of $\partial K$ also imply the growth behavior \eqref{eq:growthv}.

\subsection{The upper bound estimates}\label{sec:upper}
In this subsection, we provide an upper bound estimate of the principal curvatures, using the Pogorelov-type estimate given by Savin \cite[Proposition 4.3]{savin2005obstacle}; see also \cite[Lemma 2.2]{huang2024regularity}. 

Suppose $w$  is a solution of  $\det D^2 w =g(w)\chi_{\{w>0\}} $ with the artificial assumption that $D_n w<0 $. Then the graph of $w$ can be viewed as the graph of $\mathbf{w}$ in the $e_{n}$ direction,
\[
\left(x_1, . ., x_n, w(x)\right)=X=\left( x_1 , . .,  x_{n-1}, \mathbf{w}\left(x_1, . ., x_{n-1}, x_{n+1}\right),x_{n+1}\right).
\]
Direct calculations yield $D_{x'} w_i=-D_nw D_{x'}  \mathbf{w}  $ and $D_n w D_{n+1}\mathbf{w} =1$. Let $\nu$ denote the unit inner normal vector of the upper graph of $w$ at $(x,w(x))$, then we have
\[ 
\nu = \frac{\left(-Dw, 1\right)}{\sqrt{1+|D w|^2}}= \frac{\left(-D_{x'}\mathbf{w}, 1, -D_{n+1}\mathbf{w}\right)}{\sqrt{1+|D \mathbf{w}|^2}}, \quad \frac{\sqrt{1+|D w|^2}}{\sqrt{1+|D \mathbf{w}|^2}} =|D_n w|,
\]
where by $D \mathbf{w}$ we refer to the derivative of  $\mathbf{w}$ as a graph function with respect to the coordinates $\left(x_1, . ., x_{n-1}, x_{n+1}\right)$. Note that the Gauss curvature of the graph of $w$ at $X$ equals to
\[
\frac{\operatorname{det} D^2 \mathbf{w}}{\left(1+|D\mathbf{w}|^2\right)^{\frac{n+2}{2}}}=K
=\frac{\operatorname{det} D^2 w}{\left(1+|D w|^2\right)^{\frac{n+2}{2}}}
=\frac{g(x_{n+1})}{\left(1+|D w|^2\right)^{\frac{n+2}{2}}}.
\]
By relabeling the coordinates $\left(-x_{n+1} \rightarrow x_{n}\right)$, we obtain the equation
\begin{equation}\label{eq:c11 savin 1}
\operatorname{det} D^2 \mathbf{w}=g\left(x_{n}\right)\left(D_{n}\mathbf{w}\right)^{n+2}   \quad \text{in }\left\{ D_n \mathbf{w} >0\right\}.
\end{equation}
In \cite[Proposition 4.3]{savin2005obstacle}, a Pogorelov-type estimate was presented for the equation \eqref{eq:c11 savin 1}, providing the desired estimate on $D_{ii} \mathbf{w}$ for $1\leq i \leq n-1$.

However, since $D_n v =0$ on $K$, the aforementioned discussion is not applicable up to the coincidence set. Therefore, one can consider the following smooth approximations of $v$. 	For each small $\varepsilon>0$, let $a_{\varepsilon} $ be a smooth non-decreasing function such that  
\[
\quad a_{\varepsilon}(t)= 
\begin{cases}
(t+\varepsilon)^q & \text {if} \quad t\geq 0 \\
\varepsilon^{q+1} & \text {if} \quad  t\leq -\varepsilon^3. 
\end{cases}   
\]
Let $v_{\varepsilon}$ be the convex solution of the following problem
\begin{equation}\label{eq:approx equation}
\M v_{\varepsilon} =a_{\varepsilon}\left(v_{\varepsilon} \right) 
\quad \text {in } \left\{v\leq 1\right\} \cap \left\{x_n \leq 1 \right\}, 
\quad v_{\varepsilon} =v  \quad \text{on }\partial \left( \left\{v\leq 1\right\} \cap \left\{x_n \leq 1 \right\}\right) .
\end{equation}
As $a_{\varepsilon}(v)$ is non-decreasing in $v$, there exists a solution to  \eqref{eq:approx equation}. The proof for the existence is omitted here, as it is analogous to that of Proposition \ref{prop:existence obs}.

\begin{Lemma}\label{lem:uniformly monotonic}
There exists $\varepsilon_0>0$, which depends on $v$, such that for any $\varepsilon<\varepsilon_0$,  we have
\begin{equation}\label{eq:approx domain}
B_{c_1}(0)   \subset F_{\varepsilon}:=\left\{x:\; v_{\varepsilon} (x)<-\delta x_n+\delta^2\right\} \subset B_{c_2}(0)
\end{equation} 
for some positive constants $\delta$, $c_1$ and $c_2 $, all of which depend only on $n$ and $q$, 
\begin{equation}\label{eq:approx section}
v_{\varepsilon} (x)<-\delta (x_n+c_1)+\delta^2 \quad \text{in }B_{ c_1}(0),  
\end{equation}
and
\begin{equation}\label{eq:approx domain des}
D_n v_{\varepsilon} < 0  \quad  \text{in }	F_{\varepsilon}.
\end{equation}
\end{Lemma} 	 
\begin{proof}
By Theorem \ref{thm:c1a vK K>0}, we now have 
\[
c  \max\left\{c|x'|^{\frac{1+\alpha}{\alpha}}-x_n,0\right\}^{\frac{1+\alpha}{\alpha}} \leq v \leq C|x|^{1+\alpha} \quad  \text{in }B_{c}(0) .
\] 
Thus, for all $\delta >0 $ sufficiently small, there exist constants $c_1$ and $c_2$ depending only on $n,q,\delta$, which tend to $0$ as $\delta $ approaches to $0$, such that the following holds for all $ v \in \E_1$:
\[	
B_{2c_1}(0) \subset F:=\left\{x:\; v (x)<-\delta x_n+\delta^2\right\} \subset B_{c_2/2}(0)  .  
\]
and 
\[
v (x)<-\delta (x_n+2c_1)+\delta^2  \quad \text{in }B_{ 2c_1}(0).
\]

The domains in this proof below will be restricted to $\widehat S_1:=\left\{v\leq 1\right\} \cap \left\{x_n \leq 1 \right\}$. 
By the comparison principle, we have $v_{\varepsilon} \leq v$, and $K_{\varepsilon}:= \left\{ v_{\varepsilon} \leq 0 \right\} \supset K=\left\{ v= 0 \right\} $. 
In particular, $\M v_{\varepsilon} \leq (1+\varepsilon)^q \ud x$. By \cite[Lemma 1.6.1]{gutierrez2016monge}, every sequence $\{v_{\varepsilon}\}$ has a subsequence converging to some convex function $v_0$ with $v_0=v$ on $\partial \widehat S_1$, with their Monge-Amp\`ere measures also converging to $\M  v_0$.  
By the equation of $v_{\varepsilon}$, we conclude that $v_0$ satisfies \eqref{eq:obs eq g=1}. This implies $v_0 = v$ in $\widehat S_1$.
Consequently,  $v_{\varepsilon}$ converges to $v$ as $\varepsilon \rightarrow 0$, which implies \eqref{eq:approx domain} and \eqref{eq:approx section}  provided that $\varepsilon_0$ is small.

The convergence of $v_{\varepsilon}$ to $v$ implies that $\lim_{\varepsilon \to 0 } K_{\varepsilon} =K$. By the convexity of $v_{\varepsilon}$, to establish the validity of \eqref{eq:approx domain des}, it suffices to demonstrate that  $D_n v_{\varepsilon} <0$ in $F_{\varepsilon} \cap K_{\varepsilon}$, which is equivalent to showing that $D_n w_{\varepsilon} <0$  in $F_{\varepsilon} \cap K_{\varepsilon}$, where the rescaled function
\[
w_{\varepsilon}(x):=\varepsilon^{-\frac{q+1}{n}} v_{\varepsilon}(x)
\] 
satisfies
\[   
1\leq  \det D^2 w_{\varepsilon} \leq 1+\varepsilon^{-1} \chi_{\{ w_{\varepsilon} \geq   - \varepsilon^{3-\frac{q+1}{n}}  \}} 
\quad \text{in }K_{\varepsilon} ,
\quad w_{\varepsilon} =0 
\quad \text{ on }\partial K_{\varepsilon}. 
\]

Noting that $K\subset K_{\varepsilon}$,  we have $\left\|w_{\varepsilon}\right\|_{L^{\infty}(K_\varepsilon)} \geq c$, which indicates that
\begin{equation}\label{eq:uniformly monotonic}
\left| E_{\varepsilon}\right| \leq C\varepsilon^{3-\frac{q+1}{n}}  \quad \text{for } E_{\varepsilon}= \{ - \varepsilon^{3-\frac{q+1}{n}} \leq w_{\varepsilon} \leq 0  \}=\{   - \varepsilon^{3}\leq   v_{\varepsilon} \leq  0 \}.
\end{equation} 
This yields
\[
\M w_{\varepsilon} (K_{\varepsilon}) \leq |K_{\varepsilon}|+ C\varepsilon^{-1}|E_{\varepsilon}| \leq C.
\]
Thus, the Aleksandrov maximum principle implies the convergence of $w_{\varepsilon}$ to the solution of
\[
\det D^2 w=1   \quad \text{in } K \cap \left\{x_n \leq 1 \right\} , \quad w=0   \quad \text{on }\partial \left(K \cap \left\{x_n \leq 1 \right\}  \right)  .
\]
Let  $c_2$ be small enough. The global continuity of $w$ implies
\[
\inf_{ x\in K_{\varepsilon}\cap B_{c_2}(0) }  w (x)> \frac{1}{2} \sup_{ x\in K_{\varepsilon} \cap B_{c_2}\left(\frac{1}{2}e_n\right)}w(x).
\]
And hence, for $\varepsilon$ small,
\begin{equation}\label{eq:unifromly monotonic}
\inf_{ x\in K_{\varepsilon}\cap B_{c_2}(0) }  w_{\varepsilon} (x)> \frac{3}{4} \sup_{ x\in K_{\varepsilon} \cap B_{c_2}\left(\frac{1}{2}e_n\right)}w_{\varepsilon}(x)>\sup_{ x\in K_{\varepsilon} \cap B_{c_2}\left(\frac{1}{2}e_n\right)}w_{\varepsilon}(x).
\end{equation}
Therefore, the convexity of $w$ implies $D_n w_{\varepsilon} <0$ in $F_{\varepsilon}$. This proves \eqref{eq:approx domain des}.
\end{proof}

\begin{proof}[Proof of the upper bound in Theorem \ref{thm:c11 bound of K}]
Based on Lemma \ref{lem:uniformly monotonic}, we can utilize \cite[Proposition 4.3]{savin2005obstacle} on the corresponding $\mathbf{v}_{\varepsilon} $ to deduce that for any unit vector  $e_i $ orthogonal to $e_n$, we have
\[
\left( \delta^{-1} \left(\delta^2-x_{n}\right)-\mathbf{v}_{\varepsilon}  \right)	D_{ii}\mathbf{v}_{\varepsilon}  \leq C
\quad\text{in }  \left\{x:\;\mathbf{v}_{\varepsilon} (x)\leq  \delta^{-1} \left(\delta^2-x_{n}\right) \right\},
\]
where $c$ and $C$ depend only  $n$, $q$ and $\delta$. From \eqref{eq:approx section}, it follows that $\left( \delta^{-1} \left(\delta^2-x_{n}\right)-\mathbf{v}_{\varepsilon}  \right)>c_1$ whenever $(x',\mathbf{v}_{\varepsilon} ) \in B_{c_1}(0)$. Thus,  we obtain the curvature upper bound of the level sets of $v_{\varepsilon}$ within $B_{c_1}(0)$. By sending $\varepsilon \to 0$,  we find that the principal curvatures of the convex hypersurface  $\partial\left\{v  \leq t\right\}$ are bounded from above within $B_{c_1}(0)$.
\end{proof}

\subsection{The lower bound estimates}\label{sec:lower}
In this subsection, we obtain a lower bound for the principle curvatures of the free boundary by applying a Pogorelov-type estimate to the Legendre transform of $v$, as in \cite[Lemma 4.2]{savin2005obstacle} or \cite[Corollary 2.1]{huang2024regularity} for the case $q=0$.

Let $u:=v^*$ denote the Legendre transform of $v$. Then $u$ is a solution of 
\begin{equation}\label{eq:Legendreu}
\operatorname{det} D^2 u  = \left(x\cdot Du-u\right)^{-q} . 
\end{equation}
Note that if $u$ is a solution of \eqref{eq:Legendreu}, then $u-p\cdot x$ is also a solution for any $p\in\R^n$. 

We first present the following Pogorelov type estimate for \eqref{eq:Legendreu}.
\begin{Theorem}\label{lem:Pogorelov p>0}
Let $w \in C^4(\overline{\Omega})$ be a convex solution to
\[
\operatorname{det} D^2 w  = \left(x\cdot Dw-w\right)^{-q} \quad \text{in } \Omega,  \quad w =0  \quad \text{on } \partial \Omega ,
\]
with $\Omega \subset \left\{ x:\; x\cdot Dw-w>0\right\}$. Then there exists a constant $C>0$, depending only on $n$, $q$ and $\sup _{\Omega}|Dw|$, such that
\[
(-w)\left|D^2 w\right| \leq C, \quad \forall\  x \in \Omega .
\]
\end{Theorem}

\begin{proof}
Assume the maximum of
\begin{equation}\label{eq:auxi function p}
\log w_{ee}+\log |w|+\frac{1}{2}  w_{e}^{2} (y)
\end{equation}
occurs at $x_0$. After a rotation of the coordinates, we can assume that $e=e_1$, and $D^2w(x_0)$ is diagonal. 
	
We denote $G=x\cdot Dw-w$, and write
\[
\log \det D^2w=  -q\log G.
\]
Taking derivatives in the $e_1$ direction, we find at $x_0$ that
\begin{align}
\label{eq:Geq1}	& \sum_{i=1}^{n} \frac{w_{1 i i}}{w_{i i}}=-q\frac{G_1}{ G}, \\
\label{eq:Geq2}	& \sum_{i=1}^{n}\frac{w_{11 i i}}{w_{i i}}-\sum_{i,j=1}^{n}\frac{w_{i j 1}^2}{w_{i i} w_{j j}}=-q\frac{G_{11}}{ G}+q\frac{G_1^2}{ G^2} .
\end{align}

On the other hand, from \eqref{eq:auxi function p} we obtain at $x_0$ that
\begin{align}
\label{eq:Gauxi1}	& \frac{w_{i}}w+\frac{w_{11 i}}{w_{11}}+w_{1} w_{1 i}=0, \\
\label{eq:Gauxi2}	& \frac{w_{11 i i}}{w_{11}}-\frac{w_{11 i}^{2}}{w_{11}^{2}}+
\frac{w_{i i}}w-\frac{w_{i}^{2}}{w^{2}} + w_{1} w_{1 i i} + w_{1 i}^{2}\leq 0.
\end{align}

Multiplying \eqref{eq:Gauxi2} by $w_{i i}^{-1}$ and summing in $i,j$, we obtain by using \eqref{eq:Geq1}, \eqref{eq:Geq2} that 
\begin{equation}\label{eq:G ieq1} 
\frac{1}{w_{11}}\left(\sum_{i,j=1}^{n}\frac{w_{i j 1}^2}{w_{i i} w_{j j}}-q\frac{G_{11}}{ G}+q\frac{G_1^2}{ G^2} \right) -\sum_{i=1}^{n}\frac{w_{11 i}^2}{w_{i i} w_{11}^2}+\frac{n}{w}-\sum_{i=1}^{n}\frac{w_i^2}{w_{i i} w^2}-q\frac{w_1G_1}{ G}+  w_{11} \leq 0 . 
\end{equation}

From \eqref{eq:Gauxi1}, one has
\[
\frac{w_i}{w}=-\frac{w_{11 i}}{w_{11}}, \quad \text { for } i \neq 1,
\]
and
\[ \sum_{i,j=1}^{n}\frac{w_{i j 1}^2}{w_{11}w_{i i} w_{j j}}- \sum_{i=1}^{n}\frac{w_{11 i}^2}{w_{i i} w_{11}^2}-\sum_{i=1}^{n}\frac{w_i^2}{w_{i i} w^2}= \sum_{i, j \neq 1} \frac{w_{i j 1}^2}{w_{i i} w_{j j}}- \frac{w_1^2}{w_{11} w^2} \geq - \frac{w_1^2}{w_{11} w^2} . \]
Plugging this into \eqref{eq:G ieq1}, we obtain 
\begin{equation}  \label{eq:G ieq2}
\frac{q}{w_{11}}\left(\frac{G_1^2}{ G^2}-\frac{G_{11}}{ G}\right) +\frac{n}{w}- \frac{w_1^2}{w_{11} w^2}-q  \frac{w_1G_1}{ G}+  w_{11} \leq 0.  
\end{equation}

By direct calculations, we have
\[
G_1= \sum_{i=1}^{n}x_iw_{1i}=x_1w_{11}, \quad G_{11}=\sum_{i=1}^{n}x_iw_{11i} +w_{11}.
\]
Using \eqref{eq:Gauxi1}, we find that
\[
\frac{G_{11}}{ w_{11}}
=\frac{\sum_{i=1}^{n}x_iw_{11i} +w_{11} }{ w_{11}} 
= - \frac{\sum_{i=1}^{n}x_iw_i}{w}-x_1w_1 w_{11}+ 1
= -\frac{G}{w} -x_1w_1 w_{11}=-\frac{G}{w} -w_1G_1 .
\]
Together with  \eqref{eq:G ieq2}, this  gives us  that 
\[
\frac{q}{w_{11}}\frac{G_1^2}{ G^2}+q  \frac{w_1G_1}{ G} +\frac{n+q}{w}- \frac{w_1^2}{w_{11} w^2}-q  \frac{w_1G_1}{ G}+  w_{11} \leq 0.
\]

In conclusion,
\[
\frac{n+q}{w}- \frac{w_1^2}{w_{11} w^2}+  w_{11} \leq 0 .
\]
Multiplying $w_{11} w^2$ to both sides, we have
\[
\left(w w_{11}\right)^2+(n+q) ww_{11}\leq w_1^2
\]
at $x_0$, from which the result follows.
\end{proof}


Define for each $a>0$ the open cone
\begin{equation}\label{eq:conical domain}
\operatorname{Cone}(0,a)=\left\{ te :\; e \in \mathbb{S}^{n-1},\  \langle e,-e_n\rangle > \frac{1}{1+a},\ 0<t<a  \right\}.
\end{equation}
We will later apply the Pogorelov type estimate in Theorem \ref{lem:Pogorelov p>0} to the Legendre transform $v^*$ of $v$. The following lemma will ensure that $v^*$ is well defined in such a cone.

\begin{Lemma}\label{lem:cone domain of u}
There exist positive constants $\alpha\in(0,1)$, $\delta$, $c_0$, $c_1$, $c$ and $C$, all of which depend only on $n,q,\lambda$ and $\Lambda$, such that for every $v\in \E_1$, we have
\begin{equation*}\label{eq:domain nabla u}
\operatorname{Cone}(0,c_0)  \subset  \nabla v\left(E_\delta \right) \subset  \operatorname{Cone}(0,C),  
\end{equation*}
where $E_\delta:=  \left\{0< v \leq \delta\right\}  \cap \left\{ x_n \leq \delta\right\} $. 
Moreover, the function $u:=v^*$ satisfies
\begin{equation}\label{eq:u behavior}
c\left(|y'|^{\frac{1+\alpha}{\alpha}}|y_n|^{-\frac{1}{\alpha}}+  |y_n|^{\frac{1+\alpha}{\alpha}}\right)\leq  u(y',y_n) \leq C\left(|y'|^{1+\alpha}|y_n|^{-\alpha}+|y_n|^{1+\alpha}\right).
\end{equation}
for $y=(y',y_n) \in \nabla v\left(E_\delta \right) $.
\end{Lemma}
\begin{proof}
Let  $te \in \operatorname{Cone}(0,c_0)$ with $0< t < c_0$ and $e \in \mathbb{S}^{n-1}$. Then it is elementary to check that $v(x) \geq   tx\cdot e $ on $ \partial  (\{v\le\delta\}\cap\{x_n\le\delta\})$ provided that $c_0>0$ is small enough. Since $v(0) =0 =0\cdot e$, by lowering the function $tx\cdot e$ downward, it will eventually touch the graph of $v$ from below at a point $x_0 \in E_\delta$ with $ te \in \partial v(x_0) $. This implies $te \in \nabla v(E_\delta)$, thereby yielding $\operatorname{Cone}(0,c_0)  \subset  \nabla v\left(E_\delta \right)$.  

By the $C^{1,\alpha}$ regularity of $v$ (Lemma \ref{lem:c1a v pointwise}), we have $ \nabla v\left(E_\delta \right)  \in B_C(0)$. Let us assume that $B_{c}\left(\frac{1}{2}e_n\right) \subset K $. Choose $\delta$ sufficiently small so that $E_\delta\subset B_{c/2}(0)$. For every $x \in E_\delta$, since
\[ 
\nabla v (x ) \cdot (z-x) \leq v (z) -v (x)= -v (x) \leq 0 , \quad \forall\  z \in B_{c}\left(\frac{1}{2}e_n\right) ,
\]
we have $\nabla v(x) \in \operatorname{Cone}(0,C)$.  This implies $\nabla v\left(E_\delta \right) \subset  \operatorname{Cone}(0,C)$.

By Theorem \ref{thm:c1a vK K>0}, we now have  
\[
 \Psi_- \leq v \leq \Psi_+ \quad \text{in }B_{c}(0),
\] 
where 
\[
\Psi_- =c\max\left\{c|x'|^{\frac{1+\alpha}{\alpha}}-x_n,0\right\}^{\frac{1+\alpha}{\alpha}}, \quad \Psi_+= C\max\left\{ C|x'|^{1+\alpha}-x_n,0\right\}^{1+\alpha}.
\]
By using the order-reversing property of the Legendre transform,  this implies 
$
 \Psi_+^* \leq v^* \leq  \Psi_-^*
$
in their common domain of definition. The inequality \eqref{eq:u behavior} then follows from the explicit expressions of $\Psi_-^*$ and $\Psi_+^*$ below: with 
\[
y=(y',y_n) :=\left(C^{\frac{1}{1+\alpha}}(1+\alpha)^2\Psi_+^{\frac{ \alpha}{1+\alpha}}|x'|^{\alpha-1}x',-C^{-\frac{ \alpha}{1+\alpha}}(1+\alpha)\Psi_+^{\frac{ \alpha}{1+\alpha}}\right) ,
\]
we obtain that
\[
\Psi_+^*(y):= x\cdot y-\Psi_+(x)= c_2|y'|^{\frac{1+\alpha}{\alpha}}|y_n|^{-\frac{1}{\alpha}}+c_3 y_n^{\frac{1+\alpha}{\alpha}},
\]
where $c_2, c_3> 0$ are positive. This establishes the left-hand side in \eqref{eq:u behavior}. Similarly, the corresponding calculation for $\Psi_-^*$ gives the right-hand side in \eqref{eq:u behavior}.
\end{proof}

\begin{proof}[Proof of the lower bound in Theorem \ref{thm:c11 bound of K}]
Now, we provide the lower bound estimate of the principal curvatures. By examining the Taylor expansion of $v$ along its level sets,	it suffices to show $|\nabla v(x_0)| \I_n \leq CD^2 v(x_0)$ for all $x_0 \in B_c(0) \setminus K$. Recalling the $C^{1,\alpha} $ regularity and strict convexity of $\partial K$, there exists $z \in \partial K \cap B_c(0)$ such that the outer normal of $\partial K$ at $z$ equals $\frac{\nabla v(x_0)}{|\nabla v(x_0)|}$. By invoking Proposition  \ref{prop:Savin-1} and performing an appropriate rotation and translation, we can assume, for simplicity, that $z=0$ and $\nabla v(x_0) = -s_0e_n$ for some small $s_0>0$. 

By Lemma \ref{lem:cone domain of u}, the function $u:=v^* \geq 0$  is well-defined on $\operatorname{Cone}(0,c_0)$. The estimates \eqref{eq:u behavior} implies $u(-se_n) \leq C|s|^{1+\alpha} = o(s)$, and 
\[
u(y) \geq c\left(|y'|^{\frac{1+\alpha}{\alpha}}|y_n|^{-\frac{1}{\alpha}}+  y_n^{\frac{1+\alpha}{\alpha}}\right) \geq -c_1 y_n 
\quad \text{on }  \partial \left(\operatorname{Cone}(0,c)\right).  
\]
Then, by applying Theorem \ref{lem:Pogorelov p>0} to the function  $ u(y)+c_1y_n$ in the section $\left\{ u\leq -c_1 y_n\right\}$, we obtain $ \left(-u(y)-c_1y_n\right) |D^2 u(y)| \leq C_1 $, which yields $|\tilde{y}_n| |D^2 u(\tilde{y})| \leq 2C$ at $\tilde{ y} = -s_0e_n$. Since $D^2 v(x) $ is the inverse of $D^2 u(\tilde{y})$, then $ 2C D^2 v(x_0) \geq |D_n v(x_0)| \I_n $.
\end{proof}

\section{$C^{2,\alpha}$ regularity of the free boundary}\label{sec:c2a regularity}

For any  $v \in \E_1$, based on the $C^{1,\alpha}$ regularity and strict convexity of $v$, and \eqref{eq:v distance alpha 1},  the mapping
\[
T: (y',x_n) \to (y',D_{x_n} v(y',x_n))
\]
transforms the set $\left(x',\psi(x')\right) \in \partial K$ to $(x', 0)$ and provides a local homeomorphism between $B_c(0) \setminus \mathring{K}$ and $\left\{ x_n \leq 0 \right\}$ around $0$. Consider the partial Legendre transform of $v$ with respect to the $x_n$ variable
\[ 
\G v(y',y_n): = \sup_{t}\{ t \cdot y_n -v(y',t)\}= x_n \cdot D_{x_n} v(y',x_n)-v(y',x_n).
\]
It can be readily verified that $\G v $ is a well-defined function on $\overline{B_c^{-}(0)} $ for small $c$ (independent of $v$), where $B_{c}^-(0) :=B_c(0)\cap \left\{ x_n <0 \right\}$. By observing $v(x)+\G v(y)= x_ny_n$, it follows that
\[
\frac{\partial  \G v}{\partial y_i } +\frac{\partial v}{\partial x_i }=0\quad \text { for } i \neq n,\quad y_n=\frac{\partial v}{\partial x_n } , \quad  x_n=\frac{\partial \G v}{\partial y_n }.
\]
Therefore, we formally have
\[
\left( \frac{\partial y_i}{\partial x_j}\right)_{1\leq i,j\leq n}=   \left(\begin{array}{cc} 
\mathcal{I}_{n-1} & 0 \\
v_{jn}  &  v_{nn}\end{array} \right), \quad 
\left( \frac{\partial x_i}{\partial y_j}\right)_{1\leq i,j\leq n}=   \left(\begin{array}{cc} 
\mathcal{I}_{n-1} & 0 \\
-\frac{v_{jn}}{v_{nn}}  &  \frac{1}{v_{nn}}\end{array} \right).
\]
This implies that for $1\leq i \leq j \leq n-1 $,
\[ 
-\frac{\partial}{\partial y_j}\frac{\partial  \G v}{\partial y_i }=  \frac{\partial}{\partial y_j}\frac{\partial v}{\partial x_i }= v_{ij} -\frac{v_{in}v_{jn}}{v_{nn}}, 
\quad -\frac{\partial}{\partial y_i}\frac{\partial  \G v}{\partial y_n }=  -\frac{\partial x_n}{\partial y_i} = \frac{v_{in}}{v_{nn}},
\]
and
\[ 
D_{y_ny_n} \G v+ \frac{(-1)^n}{g v^{q}}\det D_{y'}^2  \G v= 0
\quad  \text{in }B_c^{-}(0) .
\]

Observing that the function $x_n =D_{y_n} \G v(y',0)$ is the graph function of $\partial K$, we now consider the function
\begin{equation}\label{eq:coor s theta x}
\psi( y', y_n ):= \frac{ \G v\left(y',y_n\right)}{y_n} \quad \text{if } y_n <0 ,\quad  \psi( y', 0 ):=D_{y_n} \G v(y',0).
\end{equation}
and we denote this operator as $\G_n: v \to \psi$. Using similar arguments to those of proving the $C^{1,\alpha}$ regularity of $v$ in Theorem \ref{thm:c1a vK K>0}, one can show that $\psi\in C^\alpha(\overline{B^-_{c}})$ for another possibly  smaller $c>0$.

By noting $\G v=y_n\psi$, we can use the inverse formula of the partial Legendre transform to deduce that
\[
v=  y_nD_{y_n} \G v-\G v= y_n D_{y_n}(y_n\psi ) -y_n\psi= y_n^2 D_{y_n} \psi.
\]
Thus, after the change of variables
\[
(y',s)=(y',|y_n|^{\frac{k}{2}}),  \quad k=\frac{n-q}{q+1} , \quad y_n D_{y_n}[\cdot]=   \frac{k}{2} sD_{s} [\cdot],
\] 
we obtain the following equation:
\begin{equation}\label{eq:eq for psi}
\psi_{ss} + \frac{k+2}{k} \frac{\psi_{s}}{s}+\frac{1}{g }\left(\frac{2}{k}\right)^{q+2} \left(\frac{-\psi_{s}}{s}\right)^{-q} \det D_{y'}^2\psi =0 \quad  \text{in }B_c^{+}(0).
\end{equation} 


The main objective of this section is to establish the $C^{2,\alpha}$ regularity of $\psi$ under the coordinates $(y',s)$. Here, we are not concerned with the solutions of equation \eqref{eq:eq for psi}, but rather with the $\psi=\G_n v$ induced by $v \in \E_1$.  
Subsequently, in different coordinate systems, the radius $c$ in $B_c $ or $B_c^+$ may vary in different places. As long as $c$ is sufficiently small, it will not affect our proof.

Let us first consider the case where $g \equiv 1$. In this case, $v$ is smooth and strictly convex in $B_c(0)\setminus K$, and thus, $\psi$ is smooth in $B_c^{+}(0)$ and satisfies \eqref{eq:eq for psi} in the classical way. Theorem \ref{thm:c11 bound of K} and Lemma \ref{lem:c11 out} imply the following elliptic structure of $\psi$. 
\begin{Lemma}\label{lem:psi elliptic 1}	
Suppose $g \equiv 1$. Then the function $\psi = \G_n v$ satisfies 
\begin{equation}\label{eq:psi elliptic 1}
c\leq 	-\frac{\psi_s}{s}  \leq C  \quad  \text{and }\quad 	c\I_{n-1} \leq  D_{y'}^2 \psi \leq C\I_{n-1}  \quad  \text{in }B_c^+(0).
\end{equation}
\end{Lemma}

\begin{proof}
Given $p=(p',p_n) \in \partial K$, we then let $p_n = \psi(p',0)$ and $b' = D_{x'}\psi(p',0)$. After considering an affine transformation $(x', x_n) \to (x' - p', x_n -  b' \cdot (x'-p') - p_n)$ that translates the point $p $ to $0$ and the tangent plane of $\partial K$ at $p$ to the plane $\left\{x_n = 0 \right\}$, we may replace $v$ with $v(x' + p', b' \cdot x'+x_n + p_n)$. Consequently, the corresponding $\psi$ transforms into $\psi(y' + p', s) - b' \cdot y'-p_n$, and vice versa. 

Therefore, without loss of generality, we only need to show \eqref{eq:psi elliptic 1} for $y'=0$. Note that since
\[
-\frac{\psi_s}{s} =\frac{2v}{k|y_n|^{k+1}}= \frac{2v}{k|v_n|^{k+1}}, \quad 
D_{y_iy_j}^2 \psi =  \frac{D^2_{y_iy_j} \G v }{y_n} = \frac{1}{|v_n|} \cdot \left[ v_{ij} -\frac{v_{in}v_{jn}}{v_{nn}}\right] 
\]
hold for $1\leq i\leq j \leq n-1$, it suffices to show that for any $v \in \E_1$ with $g \equiv 1$,  we have 
\[	
c\leq \frac{v}{|v_n|^{k+1}}  \leq C \ \   \text{and }\ \  c\I_{n-1} \leq  \frac{1}{|v_n|} \cdot \left[ v_{ij} -\frac{v_{in}v_{jn}}{v_{nn}}\right] \leq C\I_{n-1} \ \  \text{on } \left\{ te_n:\; t \in (-c,0) \right\}.
\]
By considering the normalized function $v_h$ defined in \eqref{eq:normalized sol c11} with $h=c^{-1}v(te_n)$ and noting the invariance of our conclusions under this normalization, we can confine our focus to the case where $ te_n \in \left(S_{c} \setminus S_{c/2}\right)$.
According to Theorem \ref{thm:c1a vK K>0} and Lemma \ref{lem:c11 out}, we  have $c\leq v \leq C$, $c\leq | v_n| \leq C$ and $c\I_n \leq D^2 v \leq C\I_n$ at this point, which imply the estimates we want. 
\end{proof} 
 
We now present a modified version of \cite[Theorem 7.3]{de2021certain} (where $b$ is assumed to be a constant in \eqref{eq:singular eq} below). The proof will be similar.
\begin{Lemma}\label{lem:holder K ndiv}
Let  $w$ be a solution of 
\begin{equation}\label{eq:singular eq}
\mathcal{L} w:=w_{nn} +b\frac{w_{n}}{x_n}+\sum_{1 \leq i, j \leq n-1 }a^{ij}(x)w_{ij}=0  \quad \text{in } B_1^{+}(0) ,
\end{equation}
where the coefficients $\left\{a_{i j}(x) \right\}_{1\leq i, j \leq n-1}$ are uniformly elliptic with ellipticity constants $0<\lambda \leq \Lambda$, and $b$ is a constant satisfying $b \in [b_0, b_1]$ for some constants $b_1\ge b_0>1$. Then, there exists $\alpha \left(n, \lambda,\Lambda, b_0,b_1\right) \in (0,1)$ such that
\[
\left\|w\right\|_{C^{\alpha}\left(B_{1/2}^{+}(0)\right)} \leq C\left(n, \lambda,\Lambda, b_0,b_1\right)\left\| w\right\|_{L^{\infty}\left(B_{1}^{+}(0)\right)}.
\]
\end{Lemma}
\begin{proof}
According to the classical elliptic regularity theory, it suffices to establish the following Harnack inequality: there exist two positive constants $\sigma$ and $\delta$, both of which depend only on $n, \lambda,\Lambda, b_0$ and $b_1$, such that for every solution $\tilde{w}$ of \eqref{eq:singular eq} with $\left\| \tilde{w}\right\|_{L^{\infty}\left(B_{1}^{+}(0)\right)} \leq 1$, if $\tilde{w}\left(\frac{1}{2} e_n\right) > 0$, then $\tilde{w} \geq -1+\sigma$ on $B_{\delta}^+$.

From the classical Harnack inequality, we get that for $l>0$ small,
\[
\tilde{w} +1 \geq c\left(n, \lambda,\Lambda, b_0,b_1,l\right) \quad \text{on }\left\{\left|x^{\prime}\right| \leq 3 / 4\right\} \times\left\{x_n=l\right\} .
\]
Let
\[
w_{\epsilon}= c\left( -|x'|^2+Mx_n^2 -\epsilon x_n^{1-b_0}\right)
\]
with $M =2(n-1)\Lambda/\lambda$ and $\epsilon>0$, and choose $d(\epsilon)>0$ so that
\[
w_\epsilon \leq-\frac{c}{2} \quad  \text{if } x_n \leq d(\epsilon), \quad d(\epsilon) \rightarrow 0 \quad \text{as }\epsilon \rightarrow 0. 
\] 
We can check that when $l$ is small,
\[
\begin{split}
w_\epsilon \leq \frac{c}{2} & \quad \text{on }\left\{\left|x^{\prime}\right| \leq 3 / 4\right\} \times\left\{x_n=l\right\}, \\
w_\epsilon \leq-\frac{c}{2} & \quad \text{on }\left\{|x|^{\prime}=\frac{3}{4}\right\} \times\left\{0 \leq x_n \leq l\right\} .
\end{split}
\]
Since $\mathcal{L} w_\epsilon>\mathcal{L} \left(-\epsilon x_n^{1-b_0}\right) = (1-b_0)\epsilon (b-b_0) \geq 0$, the comparison principle implies that
\[
w_\epsilon+\frac{c}{2} \leq \tilde{w}+1 \quad \text{in }\left\{\left|x^{\prime}\right| \leq 3 / 4\right\} \times\left\{d(\epsilon) \leq x_n \leq l\right\} .
\]
By letting $\epsilon \rightarrow 0$, we obtain the desired estimate.
\end{proof}

Subsequently, we can demonstrate the $C^{2,\alpha}$ regularity of $\psi$ for the case of $g \equiv 1$.	
\begin{Theorem}
Suppose $g \equiv 1$. Then for any $\alpha \in (0,1)$, the function $\psi = \G_n v$ satisfies 
\begin{equation}\label{eq:c2a g=1}
\left\|\psi  \right\|_{C^{2,\alpha} \left(\overline{B_c^+(0)}\right)} +  \left\| \frac{D_{s}\psi}{s} \right\|_{C^{\alpha} \left(\overline{B_c^+(0)}\right)} \leq C(n,q,\alpha)
\end{equation}
under the coordinates $(y',s)$.
\end{Theorem}
\begin{proof}
According to Lemma \ref{lem:psi elliptic 1}, we have
\[ 
0\leq B(y',s):=  q\left(\frac{2}{k}\right)^{q+2}  \left(\frac{-\psi_{s}}{s}\right)^{-q-1} \det D_{y'}^2\psi \leq C.
\]
By differentiating \eqref{eq:eq for psi} in the $s$ direction and multiplying by $s^{-1}$, we find that $w= \frac{\psi_s}{s} \in L^{\infty}(B_c^{+}(0))$ is a solution of 
\[
w_{ss} + \left[\frac{3k+2}{k}+B\right]  \frac{w_{s}}{s}+  \left(\frac{2}{k}\right)^{q+2} \left(\frac{-\psi_{s}}{s}\right)^{-q} \sum_{1 \leq i, j \leq n-1 } \Psi^{'ij}  w_{ij}  =0  ,
\]
where $\{ \Psi^{'ij}\}$ is the cofactor matrix of $D_{y'}^2 \psi$. 
Applying Lemma \ref{lem:psi elliptic 1} and Lemma \ref{lem:holder K ndiv}, we obtain that $  \left\| \frac{D_{s}\psi}{s} \right\|_{C^{\alpha_0} \left(\overline{B_{c/2}^+(0)}\right)} \leq C$ for some $\alpha_0>0$.

Subsequently, we make the even extension of $\psi$, i.e., $\psi(y',-s)=\psi(y',s)$ such that $\left\| \frac{D_{s}\psi}{s} \right\|_{C^{\alpha_0} \left(\overline{B_{c/2}(0)}\right)} \leq C$. By noting the concavity of the operator
\[ 
F(D^2u ,x)= u_{ss} + \frac{k+2}{k} \frac{\psi_{s}}{s}+\left(\frac{2}{k}\right)^{q+2} \left(\frac{-\psi_{s}}{s}\right)^{-q} \det D_{y'}^2u ,
\]
the classical Evans-Krylov theorem  and the Schauder estimates yield the interior $C^{2,\alpha_0}$ regularity of $\psi$ in $B_c(0)$, possibly for a different $\alpha_0>0$.

Next,  for each $1\leq l\leq n-1$, by differentiating \eqref{eq:eq for psi} in the $x_l$ variable and multiplying by $s^{-1}$, we see that $\psi_{l}:=D_{e_{l}} \psi \in L^{\infty}(B_c^{+}(0))$ satisfies
\[
(\psi_{l})_{ss} + \left[\frac{k+2}{k}+B\right]  \frac{(\psi_{l})_{s}}{s}+ \left(\frac{2}{k}\right)^{q+2} \left(\frac{\psi_{s}}{s}\right)^{-q} \sum_{1 \leq i, j \leq n-1 } \Psi^{'ij}  (\psi_{l})_{ij}  =0  .
\]
By the Schauder estimates \cite[Theorem 5.4]{huang2024regularity}, we have
\[
\left\|\psi_{l} \right\|_{C^{2,\alpha_0} \left(\overline{B_c^+(0)}\right)} +  \left\| \frac{D_{s}\psi_{l}}{s} \right\|_{C^{\alpha_0} \left(\overline{B_c^+(0)}\right)} \leq C.
\]

It remains to show the $C^{\alpha}$ regularity of $-\frac{\psi_s}{s}$. The estimates on $\psi_{l}$ implies 
\[ 
f(y',s):= (q+1)\left(\frac{2}{k}\right)^{q+2}   \det D_{y'}^2\psi \in C^{\alpha} \left(\overline{B_c^+(0)}\right)
\]
for all $\alpha \in (0,1)$.   From \eqref{eq:eq for psi}, we find that $\zeta =\left(-\frac{\psi_s}{s}\right)^{q+1}$ is a solution of
$\left( s^{\beta} \zeta\right)_s = s^{\beta}f$, where $\beta=\left(\frac{2(k+1)(q+1)}{k}\right).$
This gives $\zeta (y',s) =\zeta(y',0) +  s^{-\beta} \int_{0}^{s} r^{\beta} f(y',r)dr $ 
and 
\[ 
\zeta_s(y',s) = -\frac{k+2}{k} s^{-\beta-1} \int_{0}^{s} r^{\beta} f(y',r)dr +f \in  C^{\alpha} \left(\overline{B_c^+(0)}\right).
\]
Noting $c\leq 	-\frac{\psi_s}{s}  \leq C $, we find that $-\frac{\psi_s}{s} \in C^{\alpha}$, and the proof is completed.
\end{proof}

Theorem \ref{thm:classify}  is derived from the $C^{2,\alpha}$ estimate \eqref{eq:c2a g=1} after blowing-down at infinity. In the case of $q=0$, this is an essentially equivalent form of  \cite[Theorem 3.1]{huang2024regularity} by considering the partial Legendre transform of $\psi$ with respect to the $x'$ variable.

\begin{proof}[Proof of Theorem \ref{thm:classify}]
Since  $K$ is a non-compact convex closed set, we can assume for simplicity that $0 \in \partial K$, $K \subset \left\{x_n \geq 0 \right\}$, and the ray $\left\{ te_n  :\; t \geq 0 \right\} \subset  K$. Thus, $\partial K$ is a graph over $\R^{n-1}$. Recalling  Lemma \ref{prop:infinite ray}, for each $t >0$, $K_t'= \left\{ x':\; (x',t) \in K\right\}  $ is compact, as there are no rays on $\partial K$. 
By considering the normalization (See Definition \ref{def:normal type1}) of $v$ at $0$ and recalling \eqref{eq:slice relation}, we have 
\[
c{t_1}^{-\frac{1}{2}}K_{t_1}'\subset {t_2}^{-\frac{1}{2}}K_{t_2}' \subset C{t_1}^{-\frac{1}{2}}K_{t_1}', \quad \forall\  t_1,t_2 >0,
\]
which remains invariant under the normalization \eqref{eq:normalized sol c11}. In particular, $\partial K$ is an entire graph. Without loss of generality, let us assume that $K_1'$ has already been normalized in the sense that 
\[
B_c'(0) \subset K_1' \subset B_C'(0).
\]
Then, we have  
\[
ct^{\frac{1}{2}}{B}_{1}'(0) \subset K_t'\subset Ct^{\frac{1}{2}}{B}_{1}'(0) ,\quad \forall\  t>0.
\]
After appropriately modifying the universal constants, we can consider the following normalized function as in \eqref{eq:normalized sol c11},
\[
{v}_h(x',x_n)= t^{-\frac{n+1}{n-q}}v\left(t^{\frac{1}{2}}x',tx_n\right), \quad \psi_h (y',s):=\G_n {v}_h= t^{-1}\psi(t^{\frac{1}{2}}y',t^{\frac{1}{2}}s) ,
\]
where $\psi=\G_n v$ and $h=t^{\frac{n+1}{n-q}}$. Since $\G_n v_h $ is well defined on $\overline{B_c^{-}(0)}$ (where $c$ is independent of $h$), we know that $\psi$ is defined in $\R^n_+$. By \eqref{eq:c2a g=1},  
\[	
\left\|\psi_h  \right\|_{C^{2,\alpha} \left(\overline{B_c^+(0)}\right)} +  \left\| \frac{D_{s}\psi_h }{s} \right\|_{C^{\alpha} \left(\overline{B_c^+(0)}\right)} \leq C. 
\]
Let $h \to \infty$. According to the $C^{2,\alpha}$ estimates \eqref{eq:c2a g=1}, a standard scaling argument implies that $\psi$ is a quadratic polynomial. Observing  \eqref{eq:psi elliptic 1}, $\psi(0)=0$, and $x_n=\psi(x',0)$ is the graph function of $\partial K$, we conclude that $\psi(y',s) $ is affine equivalent to
\begin{equation}\label{eq:psi global sol}
\omega(y',s)
=\frac{1}{2}|y'|^2-\frac{(q+1)^{\frac{q+2}{q+1}}}{(n-q)(n+1)^{\frac{1}{q+1}}} s^2.
\end{equation}
Consequently, $v$ is  affine equivalent to \eqref{eq:glo model 2}.
\end{proof}

\begin{proof}[Proof of Corollary \ref{thm:classify0}]
If $|K|=0$, then $\det D^2 v = 1$ in $\mathbb{R}^n$ in the Aleksandrov sense. It follows from the theorems of J\"orgens \cite{jorgens1954losungen}, Calabi \cite{calabi1958improper}, and Pogorelov \cite{pogorelov1978minkowski} that the solution must be quadratic, and thus, $K=\emptyset$ or is a single point. 

If $0<|K|<\infty$, then the coincidence set $K$ is bounded, and the Legendre transform $u$ of $v$ satisfies $\det D^2 u = 1$ in $\mathbb{R}^n\setminus\{0\}$. It follows from J\"orgens \cite{jorgens1955harmonische} for $n = 2$ and Jin-Xiong \cite{jin2016solutions} for $n \geq 3$ that $u$ must be affine equivalent to $\int_{0}^{|x|} \left( r^n + 1 \right)^{\frac{1}{n}} dr$.  Hence, $v$ is  affine equivalent to $\int_{0}^{|x|} \max\left\{r^n-1,0\right\}^{\frac{1}{n}}dr$, and $K$ is an ellipsoid.

If $|K|=\infty$, then the conclusion follows from Theorem \ref{thm:classify}.   
\end{proof}
 
Lastly, let us assume that $g >0$ is $C^\alpha$ for some fixed $\alpha \in (0,1)$ and $\partial K$ is strictly convex at $0$. We will prove Theorem \ref{thm:c2a origorous} at $0$, using the perturbation arguments in \cite{jian2007continuity}. 

By considering a suitable normalization, which is still denoted as $v$, we can assume for simplicity that $v \in \E_1$ is normalized, $D_n v \leq 0$ in $B_C(0)$,  $g(0)=1$, $\left\| {g}\right\|_{C^{\alpha}\left( B_C(0)\right)} \leq 2$, and 
\[
\|g-1\|_{L^{\infty}\left(B_C(0)\right)}=\sigma <\sigma_0
\]
is sufficiently small. The approximation functions of $v$ and $\psi$ are constructed as follows:
\begin{Lemma}\label{lem:psi approx sigma}
Let  $\tilde{v}$ be the non-negative convex solution of
\[  
\det D^2\tilde{v}= \tilde{v}^q\chi_{\{\tilde{v}>0\}}, \quad \text{in }B_c(0),\quad \tilde{v}=v \quad \text{on }\partial B_c(0),
\] 
and define $\tilde{\psi} = \G_n \tilde{v}$ as in \eqref{eq:coor s theta x}.  Then, $\tilde{\psi} $ satisfies \eqref{eq:psi elliptic 1} and \eqref{eq:c2a g=1}, and we have
\[
\left\|\tilde{\psi} -\psi \right\|_{L^{\infty} \left(B_\tau^+(0)\right)}  \leq C\sigma
\]
for some small $\tau>0$.
\end{Lemma} 

\begin{proof}
Since $D_n v \leq 0$, the functions
\[
\overline{w}(x',x_n)= v\left(x',(1-2\sigma)x_n-C_0\sigma\right)
\]
and
\[
\underline{w}(x',x_n)= v\left(x',(1+2\sigma)x_n+C_0\sigma\right)
\]
satisfy
\[
\det D^2 \overline{w}\leq  (1-\sigma) \overline{w}^q\chi_{\{\overline{w}>0\}}  \quad \text{in }B_c(0),\quad \overline{w} \geq v =\tilde{v}\quad  \text{on }\partial B_c(0),   
\]
and
\[
\det D^2 \underline{w}\geq  (1+\sigma)\underline{w}^q\chi_{\{\underline{w}>0\}} \quad \text{in }B_c(0),\quad \underline{w} \leq v=\tilde{v} \quad  \text{on }\partial  B_c(0) ,
\]
respectively, where $C_0>0$ is a constant.
The comparison principle implies
\[
\underline{w}\leq  \tilde{v} \leq \overline{w} \quad \text{in } B_c(0).
\]
Since the partial Legendre transform reverses the order relation, this gives
\[
\G \overline{w}  \leq \G \tilde{v} \leq \G \underline{w} \quad \text{in } B_\tau^+(0).
\]
Note that since $\G v$ is Lipschitz continuous in the last variable, and
\[
\G \overline{w}(y',y_n) =  \G {v}\left(y', \frac{y_n}{1-2\sigma} \right)+\frac{C_0\sigma }{1-2\sigma}y_n ,\quad
\G \underline{w}(y',y_n)= \G {v}\left(y', \frac{y_n}{1+2\sigma} \right)-\frac{C_0\sigma }{1+2\sigma}y_n,
\]
we obtain that
\[
\left\|\tilde{\psi} -\psi \right\|_{L^{\infty}\left(B_\tau^+(0)\right)} = \left\|\frac{\G \tilde{v}-\G v}{y_n}\right\|_{L^{\infty}}  \leq C\sigma.
\] 

Finally, since $\sigma$ is small, $\widetilde{K}:=\left\{  \tilde{v}(x)=0\right\}=\left\{x_n\ge \tilde\psi(x',0)\right\}$ is an upper graph, and $|\widetilde{K}|>0$. By a possible translation, there is an appropriate scaling of $\tilde{v}$ defined in \eqref{eq:normalized sol} belonging to $\E_1$. Hence, by applying Theorem \ref{thm:c1a vK K>0} to $\tilde v$, we obtain the strict convexity of $\partial \widetilde{K}$, and that $\tilde{\psi} $  satisfies \eqref{eq:psi elliptic 1} and \eqref{eq:c2a g=1} around $0$.
\end{proof}

\begin{Lemma}\label{lem:global C2a bounded 1}
Suppose $\alpha \in (0,1)$, $g\equiv 1$, $\psi_1 $ and $\psi_2  $ are two solutions of 
\eqref{eq:eq for psi} that satisfy \eqref{eq:psi elliptic 1} and \eqref{eq:c2a g=1}, then we have
\[ 	\left\|\psi_1 -\psi_2  \right\|_{C^{2,\alpha} \left(\overline{B_{c/2}^+(0)}\right)} \leq C  	\left\|\psi_1 -\psi_2  \right\|_{L^{\infty} \left(\overline{B_c^+(0)}\right)}.
\]
\end{Lemma}
\begin{proof}
By considering the linearized equation for $\psi_1-\psi_2$, the proof follows directly from Lemma \ref{lem:holder K ndiv} and Lemma \cite[Theorem 5.4]{huang2024regularity}.
\end{proof}

\begin{proof}[Proof of Theorem \ref{thm:c2a origorous}]
After considering an appropriate normalization around $0$, we can assume for simplicity that $D_n v \leq 0$ in $B_C(0)$,  $g(0)=1$, $\left\| {g}\right\|_{C^{\alpha}\left( B_C(0)\right)} \leq 2$, and 
\[
\|g-1\|_{L^{\infty}\left(B_C(0)\right)}=\sigma <\sigma_0
\]
is small to be fixed later.

We will first establish the regularity of $ \psi= \G_n v $ when $g(x,v,Dv)=g(x)$ is independent of $(v,Dv)$.

For $m=1,2,\cdots$, let $r_m=  r_1^{m}$, where $r_1>0 $ is small to be fixed later.
As demonstrated in Definition \ref{def:normal type1}, for each $m$, there exist a linear transformation $  {\D_m'}$ and a constant $b_m >0$ with	$b_m^{n-q}=\det {\D_m'} $ such that
\[
{v_m}(x)=b_m^{-2}r_m^{-\frac{n+1}{n-q}}v\left( {\D_m'} r_m^{\frac{1}{2}} x', r_mx_n\right) \in \E_{1}.
\] 
Then, $v_m$ is a solution of
\[
\det D^2  {v}_m(x) =  {g}_m(x) v_m^q\chi_{\left\{ {v}_m > 0\right\}}, \quad {g}_m(x)=g\left( {\D_m'} r_m^{\frac{1}{2}} x', r_mx_n\right)
\]
such that
\begin{equation}\label{eq:gk estimate}
\sigma_m:= \left\| {g}_m-1\right\|_{L^{\infty}\left(B_C(0)\right)}	\leq \min\left\{C\left(
r_m^{\frac{1}{2}}\| {\D_m'} \|+r_m
\right)^{\alpha},\sigma_0 \right\}. 
\end{equation}

Let $\psi_m=\G_n v_m$. Then, we can utilize Lemma \ref{lem:psi approx sigma} to obtain $\tilde{v}_m$ and $\tilde{\psi}_m =\G_n \tilde{v}_m$ that satisfies  
\[
\left\|\tilde{\psi}_m -\psi_m  \right\|_{L^{\infty} \left(B_c^+(0)\right)}  \leq C\sigma_m.
\]
Note that since
\[
\psi_m  (y',s): =\G_n v_m= \frac{1}{r_m} \psi \left(\D_m'  r_m^{\frac{1}{2}} y',b_m^{\frac{n-q}{q+1}}r_m^{\frac{1}{2}} s \right)  ,
\]
we have
\[
\psi_{m+1}  (y',s)= \frac{1}{r_1} \psi_{m}  \left(\D_{m+1}'  \D_{m}'^{-1} r_1^{\frac{1}{2}} y', b_{m+1}^{\frac{n-q}{q+1}}b_m^{-\frac{n-q}{q+1}}r_1^{\frac{1}{2}} s\right).
\]
Thus,
\[
\left\|\tilde{\psi}_{m+1} -\frac{1}{r_1} \tilde{\psi}_{m}  \left(\D_{m+1}'  \D_{m}'^{-1} r_1^{\frac{1}{2}} y', b_{m+1}^{\frac{n-q}{q+1}}b_m^{-\frac{n-q}{q+1}}r_1^{\frac{1}{2}} s \right)\right\|_{L^{\infty} \left(B_c^+(0)\right)}  \leq \frac{C\sigma_m}{r_1} \leq C\sigma_{m-1}.
\]
Taking into account the $C^{2,\alpha}$ regularity of $\tilde{\psi}_m $, we arrive at the following quadratic approximation,
\begin{align}\label{eq:app free}
&\left|\psi_m (y',s) -\left(\tilde{\psi}_m (0)+   D\tilde{\psi}_m (0)\cdot (y',s)+  (y',s)\cdot D^2\tilde{\psi}_m (0)\cdot (y',s)^{T}\right) \right|  \nonumber \\
&\leq C(\sigma_m + |(y',s)|^{2+\alpha}).
\end{align}
Applying Lemma \ref{lem:global C2a bounded 1}, we obtain
\begin{align*}
\left|  r_m \tilde{\psi}_m (0)-  r_{m+1}  \tilde{\psi}_{m+1} (0)\right| 
&\leq C\sigma_{m-1} r_m,\\
\left\| r_m^{\frac{1}{2}}\T_{m+1}\T_m^{-1}D\tilde{\psi}_m (0)- r_{m+1}^{\frac{1}{2}}    D\tilde{\psi}_{m+1} (0) \right\| &\leq C\sigma_{m-1}   r_m^{\frac{1}{2}},\\
\left\|  \T_{m+1} \T_m^{-1}D^2\tilde{\psi}_m (0) \T_m^{-1,T}\T_{m+1}^{T}- D^2\tilde{\psi}_{m+1} (0) \right\| &\leq C\sigma_{m-1} ,
\end{align*}
where $\T_m=\operatorname{diag}\left\{ \D_m',b_m^{\frac{n-q}{q+1}}\right\}=\operatorname{diag}\left\{ \D_m',(\det {\D_m'})^{\frac{1}{q+1}}\right\}$. Hence, by using \eqref{eq:psi elliptic 1} and \eqref{eq:gk estimate},  we find that $\| \D_1'\|+\| \D_1'^{-1}\|\leq C$ and $\| \D_{m+1}'^{-1} \D_m'\|+\| \D_{m+1}' \D_m'^{-1}\| \leq C$ provided that $\sigma_0>0$ is small. Thus, for any given $\epsilon>0$, by selecting a sufficiently small $r_1>0$, we can infer that 
\[\| \D_m'\|+\| \D_m'^{-1}\| \leq C^{m}\leq Cr_1^{-m\epsilon }= Cr_m^{-\epsilon} \]
for all $m\geq 1$.
Subsequently, by choosing $\epsilon$ sufficiently small, we can deduce from \eqref{eq:gk estimate} that $\sigma_m \| \D_m'\|^2 \leq  Cr_m^{ \frac{\alpha}{4}} $  and 
\[
\left\|\D_m'^{-1} D_{y'}^2\tilde{\psi}_m (0) \D_m'^{-1,T}-\D_1'^{-1}D_{y'}^2\tilde{\psi}_{1} (0)\D_1'^{-1,T} \right\|\leq \sum_{j=1}^m\sigma_j \| \D_j'\|^2 \leq C\sum_{j=1}^m r_j^{ \frac{\alpha}{4}}  \leq Cr_1^{ \frac{\alpha}{4}}
\]
for all $m\geq 1$.
Thus, $\D_m'$, and consequently $\T_m$, are strictly positive definite and bounded independent of $m$, and $\partial K  \in C^{1,1}$ is uniformly convex at $0$.
Hence, for any $m \geq  1$, we can further deduce that
\[
 \sigma_m  \leq Cr_m^{\frac{\alpha}{2}}
\]
and
\begin{align*}
\left|  r_m \tilde{\psi}_m (0)-  r_{m+1}  \tilde{\psi}_{m+1} (0)\right|  &\leq C r_m^{1+\frac{\alpha}{2}},\\
\left\| r_m^{\frac{1}{2}} \T_m^{-1}D\tilde{\psi}_m (0)- r_{m+1}^{\frac{1}{2}}   \T_{m+1}^{-1}D\tilde{\psi}_{m+1} (0) \right\| & \leq C r_m^{\frac{1+\alpha}{2}}, \\
\left\|   \T_m^{-1}D^2\tilde{\psi}_m (0) \T_m^{-1,T}-     \T_{m+1}^{-1} D^2\tilde{\psi}_{m+1} (0) \T_{m+1}^{-1,T} \right\| &\leq Cr_m^{ \frac{\alpha}{2}}. 
\end{align*}
This implies $\|D^2\tilde{\psi}_m (0)-A_{\infty}\|+\|D\tilde{\psi}_m (0)-B_{\infty}\|+|\tilde{\psi}_m (0)-C_{\infty}| \leq Cr_m^{  \frac{\alpha}{2}}$ holds for some  $A_{\infty},B_{\infty} $ and $C_{\infty}$.
Recalling \eqref{eq:app free}, we then conclude that
\[
\left| \psi(y',s) -\left(C_{\infty}+ B_{\infty}\cdot (y',s) +  (y',s)\cdot A_{\infty}\cdot (y',s)^{T} \right) \right| \leq C|(y',s)|^{2+\alpha}
\quad  \text{in }B_c(0).
\]

In this way, we have established the pointwise $C^{2,\alpha}$ regularity of $\psi$ on $B_c'(0)$, which implies that $\partial K$ is $C^{2,\alpha}$ and uniformly convex, as well as the growth behavior \eqref{eq:v distance optimal} of $v$.
Furthermore, we can also show that $\psi \in C^{2,\alpha}(B_c(0))$. Here, we omit the details and merely outline the corresponding proof.
Let us assume for simplicity that $C_{\infty}+ B_{\infty}\cdot (y',s) +  (y',s)\cdot A_{\infty}\cdot (y',s)^{T} =\omega(y',s)$, where $\omega$ is the global solution defined in \eqref{eq:psi global sol}. 
For convenience, we  now set $v_m(x)= r_m^{-\frac{n+1}{n-q}}v\left( r_m^{\frac{1}{2}} x', r_mx_n\right) $.
By leveraging the order-reversing property of the partial Legendre transform, we then derive the following estimates for the approximation sequence
\[
\left\| v_m-\varpi\right\|_{L^{\infty}\left(B_c(0)\right)}  \leq C\sigma_m \leq Cr_m^{ \frac{\alpha}{2}},
\]
where $\varpi$ is given in \eqref{eq:glo model 2}. 
Subsequently, by applying the interior regularity theory for Monge-Amp\`ere equations, we can show that $v_m$ is uniformly convex in $B_c(0)\cap \left\{ v_m >\delta\right\} $ for sufficiently small $\delta>0$, provided that $\sigma_m$ is sufficiently small and $\left\| v_m-	\varpi\right\|_{C^{2,\alpha}\left(B_c(0)\cap \left\{ v_m >\delta\right\}\right)} \leq C_\delta r_m^{ \frac{\alpha}{2}}$ holds for some large $C_\delta$. Consequently, $\left\| \psi_m-	\omega\right\|_{C^{2,\alpha}\left(B_c(0)\cap \left\{ s>c_\delta\right\}\right)} \leq C_\delta r_m^{ \frac{\alpha}{2}}$. By employing a standard scaling argument, we establish the regularity of $v$ near the free boundary using \eqref{eq:normalized sol c11}, as well as the  $C^{2,\alpha}$ regularity of $\psi$ in $B_c(0)$.

Finally, when $g(x,v,Dv)$ depends on $(v,Dv)$, applying Theorem  \ref{thm:c1a vK K>0}, we have $\tilde{g}(x):= g(x,v,Dv) \in C^{\gamma}$ for some $\gamma >0$.
In this case, we can repeat the previous proof to conclude the $C^{1,1}$ regularity and uniform convexity of $\partial K$. This implies that $v$ scales as in \eqref{eq:normalized sol c11}.
When $g(x,v,Dv)=g(x,v)$ is independent of $Dv$, we find that $\tilde{g}(x)  \in C^{\alpha}$, and consequently $\psi \in C^{2,\alpha}$.
In general, when $g$ depends on $(x,v,Dv)$, we only have $\tilde{g}(x) \in C^{\beta}$, where $\beta = \min\left\{\frac{q+1}{n-q},1\right\}\alpha$. Nevertheless, it is noteworthy that $\nabla v = 0$ on $K$ and by Lemma \ref{lem:cone domain of u}, $c|\nabla v|\leq  |v_n| =|y_n|=s^{\frac{2(q+1)}{n-q}}$ under a coordinate transformation. This allows us to refine the estimate \eqref{eq:gk estimate} to $\sigma_m \leq Cr_m^{\min\left\{\frac{q+1}{n-q},\frac{1}{2}\right\}\alpha} \sigma_0$, and the previous discussion still ensures $\psi \in C^{2,\min\left\{\frac{2(q+1)}{n-q},1\right\}\alpha}$. This concludes the proof of Theorem \ref{thm:c2a origorous}.
\end{proof}

\section{An application to the $L_p$ Minkowski problem}\label{sec:Minkowski}

Let us  study the following dual version of the obstacle problem \eqref{eq:obs equation}, specifically the following isolated singularity problem. 
Suppose $q\in [0,n)$. Let $u \geq 0$ be a strictly convex solution to
\begin{equation}\label{eq:sing from dual}
\det D^2 u=\frac{f(x, Du)}{(u^*)^{q} }   + a\delta_0 \quad \text{around } 0,\quad  u^*=x\cdot Du-u > 0 \quad \text{outside } 0,
\end{equation} 
with $a> 0$, and $f$ being bounded and positive. Then, $u$ is merely Lipschitz at $0$, and we can express
\begin{equation}\label{eq:tangentconef}
u(\theta,r )= r\phi(\theta)+o(r)  
\end{equation}
under the polar coordinates $(\theta,r)$.  
Let $v$ denote the Legendre transform of   $u$. Then $v$ is a non-negative solution of the obstacle problem
\begin{equation}\label{eq:obs from dual}
\det D^2 v=\frac{1}{f(Dv,x)}\cdot v^p \chi_{\{v>0\}} ,
\end{equation}
and the tangent cone function $\phi$ is the support function of $K:=\left\{ v=0\right\}=\partial u(0)$. Therefore, the eigenvalues of the inverse matrix of $\nabla_{\mathbb{S}^{n-1}}^2\phi+\phi\I_{n-1}$ are the principal curvatures of $\partial K$. By virtue of the strict convexity of $u$, we have $\left\{ v \leq \kappa \right\}$ is compact for some $\kappa>0$. From Proposition \ref{prop:infinite ray} and Theorem \ref{thm:c1a vK K>0}, it follows that $\partial K$ is $C^{1,\alpha}$ and strictly convex.
Under the assumption that $\log f\in C^{\alpha}$ for some $\alpha \in (0,1)$, we can derive from Theorem \ref{thm:c2a origorous} that $\phi \in C^{2,\beta}$ and satisfies  $c\I_{n-1} \leq \nabla_{\mathbb{S}^{n-1}}^2\phi+\phi\I_{n-1} \leq C\I_{n-1} $ for some positive $c$ and $C$ that depend on $u$, where $\beta=\min\left\{\frac{2(q+1)}{n-q},1\right\}\alpha$. 
Subsequently, we shall investigate the asymptotic expansion of the Hessian of $u$ at the point $0$, where the $q = 0$ case has  been discussed in \cite{huang2024regularity}. By setting
\[
(\theta,s)=(\theta,r^{\frac{k}{2}}) \in \mathbb{S}^{n-1}\times [0,\infty),
\quad k=\frac{n-q}{q+1} ,
\]
and introducing the function 
\begin{equation}\label{eq:def of zeta}
\zeta( \theta,s )= \frac{u(\theta,r )}{r} \quad \text{for } r>0 ,\quad  \zeta( \theta,0 )=\phi(\theta),
\end{equation}
the calculation in \cite[Section 4]{huang2024regularity} shows that
\[
D^2 u=
\left(\begin{array}{cccc}
\frac{1}{r} \zeta_{\theta_1 \theta_1}+\zeta_r+\frac{1}{r} \zeta & \cdots & \frac{1}{r} \zeta_{\theta_1 \theta_{n-1}}  & 	\zeta_{r \theta_1}\\
\cdots & \cdots & \cdots & \cdots \\
\frac{1}{r} \zeta_{\theta_1 \theta_{n-1}} & \cdots & \frac{1}{r} \zeta_{\theta_{n-1}
\theta_{n-1}}+\zeta_r+\frac{1}{r} \zeta  &	\zeta_{r \theta_{n-1}}  \\
\zeta_{r \theta_1} & \cdots & \zeta_{r \theta_{n-1}} & r \zeta_{r r}+2 \zeta_r 
\end{array}\right)
\]
under any orthonormal frame  $\theta$ of the unit sphere. We further write
\begin{equation}\label{eq:matrix operators a}
D^2 u= \J \left\{F_{ij}(\zeta) \right\}\J^T,
\end{equation}
where
\[
\J= \operatorname{diag}\left\{r^{-\frac{1}{2}},\cdots,r^{-\frac{1}{2}},r^{\frac{k-1}{2}}\right\}=\operatorname{diag}\left\{s^{-\frac{1}{k}},\cdots,s^{-\frac{1}{k}},s^{\frac{k-1}{k}}\right\}
\] 
and
\[
\{F_{ij}(\zeta)\}=
\left(\begin{array}{cccc}
\zeta_{\theta_1 \theta_1}+\zeta+\frac{k}{2} s \zeta_s & \cdots & \zeta_{\theta_1 \theta_{n-1}} &	\frac{k}{2} \zeta_{s \theta_1} \\	
\cdots & \cdots & \cdots & \cdots \\	
\zeta_{\theta_1 \theta_{n-1}} & \cdots & \zeta_{\theta_{n-1} \theta_{n-1}}+\zeta+\frac{k}{2} s \zeta_s 	& \frac{k}{2} \zeta_{s \theta_{n-1}} \\
\frac{k}{2} \zeta_{s \theta_1} & \cdots & \frac{k}{2} \zeta_{s \theta_{n-1}} 	 &	\left(\frac{k}{2}\right)^2 \zeta_{s s}+\frac{k(k+2)}{4} \frac{\zeta_s}{s}
\end{array}\right). 
\]

\begin{Theorem}\label{thm:c2a regularity u}
Suppose $q\in [0,n)$, $\alpha \in (0,1)$, $\log f\in C^{\alpha}$, and $u \geq 0$ is a strictly convex solution to \eqref{eq:sing from dual} with $a>0$. Let $\beta=\min\left\{\frac{2(q+1)}{n-q},1\right\}\alpha$ and $k= \frac{n-q}{q+1}$. Then, we have
\begin{equation}\label{eq:decay63}
cr^{k+1} \leq u(x)-\phi(\theta) r\leq Cr^{k+1} \quad \text{in }B_c(0) , 
\end{equation}
where $c$ and $C$ are positive constants that depend on $u$.
And under the coordinates $(\theta,s)$, the function $\zeta$ defined by \eqref{eq:def of zeta} is $C^{2,\beta}$ around $\mathbb{S}^{n-1} \times \left\{0\right\}$, and the matrix $\{F_{ij}(\zeta)\}_{1\leq i,j \leq n}$ defined by \eqref{eq:matrix operators a} is elliptic and $C^{\beta}$  around $\mathbb{S}^{n-1} \times \left\{0\right\}$.
Furthermore, if $f(x,Du)=f(Du)$ is independent of $x$, then we can take $\beta=\alpha$ instead. 

More precisely, the $C^{\beta}$ (or $C^{\alpha}$) norm and the elliptic constants of $\{F_{ij}(\zeta)\}_{1\leq i,j \leq n}$ in  $\mathbb{S}^{n-1} \times [0,c_0]$ are bounded by $C_{0}$, where the positive constants $c_0$ and $C_0$ depend only on $n$, $q$, $\alpha$, $\operatorname{diam} (K)$, $a^{-1}$, $\left\|\log f\right\|_{C^{\alpha}}$, and a positive quantity $\kappa$ for which $\left\{ v \leq \kappa \right\}$ is compact, with $v$ denoting the Legendre transform of $u$.
\end{Theorem}

\begin{proof}
This proof is directly derived by applying Theorems \ref{thm:c1a vK K>0} and \ref{thm:c2a origorous} to the Legendre transform $v$ of $u$, with the details outlined as follows.

Let us first assume that  $v \in \E_1$, and $\left\|\log f\right\|_{C^{\alpha}} \leq 2$ in $B_C(0)$. 
We claim that for any  $(-e_n,s_0)$ and $(\theta_1,s_1)$ with $s_1 <s_0$ and $s_0+|e_n+\theta_1|\leq c$ being small, the matrix $F_{ij}(\zeta)$ is elliptic at 
$(-e_n,s_0)$ and satisfies 
\begin{equation}\label{eq:holder argument}
|F_{ij}(\zeta)(-e_n,s_0)-F_{ij}(\zeta)(\theta_1,s_1)| \leq Cd^{\beta},
\quad \text{ for }d:=|(-e_n,s_0)-(\theta_1,s_1)|.
\end{equation}
Let us assume that the quadratic expansion of $\partial K$ at $0$ is given by $x_n =\frac{1}{2}|x'|^2$.
Let $r_m=  r_1^{m}$ with $r_1>0 $ being small, $v_m(x)= r_m^{-\frac{n+1}{n-q}}v\left( r_m^{\frac{1}{2}} x', r_mx_n\right) $, and  let $\varpi$ be given in \eqref{eq:glo model 2}.
Following the proof of Theorem \ref{thm:c2a origorous}, we now have for $\sigma_m:=Cr_m^{ \frac{\beta }{2}}$ that
\[
\left\| v_m-\varpi\right\|_{L^{\infty}\left(B_c(0)\right)}  \leq \sigma_m, \quad
\left\| v_m-	\varpi\right\|_{C^{2,\beta}\left(B_c(0)\cap \left\{ v_m >\delta\right\}\right)} \leq C_\delta   \sigma_m
\]
for all $m\geq 1$, for some sufficiently large $C$ and $C_\delta$. Let $u_m$ and $\varpi^*$ denote the the Legendre transform of $v_m$ and $\varpi$, respectively.   Both $u_m$ and $\varpi^*$ are well-defined inside a conical domain $\overline{\operatorname{Cone}(0,c) } $ along $-e_n$, where $\operatorname{Cone}(0,c)  $ is defined in \eqref{eq:conical domain} for the $x$ variable, and we shall also use this notation for the corresponding domain in the $(\theta,s)$ coordinates. Note that
\[
u_m (x',x_n) =r_m^{-\frac{n+1}{n-q}}u\left( r_m^{\frac{n+q+2}{2(n-q)}} x', r_m^{\frac{q+1}{n-q}}x_n\right)   
\]
and 
\[
\varpi^*(x',x_n) = \left(\frac{1}{2}\frac{|x'|^2}{|x_n|^2}+\frac{1}{k(k+1)^{\frac{1}{q+1}}}|x_n|^{k}\right)|x_n|, \quad k=\frac{n-q}{q+1}.
\] 
By leveraging the order-reversing property of the Legendre transform, we then have
\begin{equation}\label{eq:umappr}
\left\| u_m-	\varpi^*\right\|_{L^{\infty}\left(\operatorname{Cone}(0,c) \right)} \leq C\sigma_m, \quad \left\| u_m-	\varpi^*\right\|_{C^{2,\beta}\left(\operatorname{Cone}(0,c) \cap \left\{ x_n >c_\delta\right\}\right)} \leq C_\delta \sigma_m.
\end{equation}
In particular, under the coordinates $(\theta,s)$, we have 
\[
\left\|F_{ij}\left(\frac{u_m}{r}\right)   -F_{ij}\left(\frac{\varpi^*}{r}\right)	\right\|_{C^{\beta}\left(\operatorname{Cone}(0,c) \cap \left\{ s >c_\delta\right\}\right)} \leq C_\delta \sigma_m ,
\]
with $c_\delta$ being as small as desired, provided that $C_\delta$ is sufficiently large. 
Note that the convex function  $\varpi^*$ is locally uniformly elliptic and locally $C^3$ in $\left\{ x_n >0\right\}$. Given the scaling invariance $\varpi^*(x',x_n)= b^{-k-1}\varpi^* (b^{\frac{k+2}{2}} x',b x_n)$ for any $b>0$,  the matrix $F_{ij}\left(\frac{\varpi^*}{r}\right):=\J^{-1} D^2\varpi^*  \J^{T,-1}$ is uniformly elliptic and $C^1$ in $\overline{\operatorname{Cone}(0,c)} $ under the coordinates $(\theta,s)$.
Thus, \eqref{eq:umappr}, together with standard scaling arguments,  implies the expansion of $u$ in the direction of $-e_n$ and the ellipticity of $F_{ij}(\zeta)$ at $(-e_n,s_0)$, and also provides the estimate \eqref{eq:holder argument} if $d \leq \varepsilon |s_0| $ and $\varepsilon$ is small. Then the estimate \eqref{eq:decay63} follows from \eqref{eq:matrix operators a} and the ellipticity of $F_{ij}(\zeta)$.

Furthermore, if one denotes $$F_{ij}(\zeta)(-e_n,0)= F_{ij}\left(\frac{\varpi^*}{r}\right)(-e_n,0),$$ then we have in $V(-e_n):=\left\{(\theta,s):\;  c_{\delta}|e_n+\theta| <s<c \right\} $ that
\[
|F_{ij}\left(\zeta\right)(\theta,s)  -F_{ij}\left(\zeta\right)(-e_n,0)|\leq C_{\delta}\left(s+|e_n+\theta|\right)^{ \beta}  .
\]
By repeating the above arguments, for any point $e $ on $\mathbb{S}^{n-1}$ that is near $-e_n$, we have in $V(e):=\left\{(\theta,s):\;  c_{\delta}|\theta-e| <s<c \right\} $ that
\[
|F_{ij}\left(\zeta\right)(\theta,s)  -F_{ij}\left(\zeta\right)(e,0)|\leq C_{\delta}\left(s+|\theta-e|\right)^{ \beta}  .
\]
Then for any point $e $ on $\mathbb{S}^{n-1}$ that is near $-e_n$, we have
\[
\begin{split}
&|F_{ij}\left(\zeta\right)(e,0)  -F_{ij}\left(\zeta\right)(-e_n,0)|\\
\leq & |F_{ij}\left(\zeta\right)(e,0)  -F_{ij}\left(\zeta\right)(\theta,s)| 
 +|F_{ij}\left(\zeta\right)(\theta,s)  -F_{ij}\left(\zeta\right)(-e_n,0)|\\
\leq  &  C_{\delta}|e_n+e|^{ \beta},
\end{split}
\] 
where $(\theta,s)=\left(\frac{e-e_n}{|e-e_n|}, |e+e_n|\right)$.
Thus, we conclude the proof of   \eqref{eq:holder argument} by observing that if $d \geq \varepsilon |s_0| $, then
\[
\begin{split}
|F_{ij}\left(\zeta\right)(\theta_1,s_1)  -F_{ij}\left(\zeta\right)(-e_n,s_0)|
&\leq |F_{ij}\left(\zeta\right)(\theta_1,s_1)  -F_{ij}\left(\zeta\right)(\theta_1,0)| \\
& \quad+|F_{ij}\left(\zeta\right)(\theta_1,0)  -F_{ij}\left(\zeta\right)(-e_n,0)|\\
&\quad+|F_{ij}\left(\zeta\right)(-e_n,s_0)  -F_{ij}\left(\zeta\right)(-e_n,0)| \\
&\leq  C(1+\varepsilon^{-1})^{\beta}d^{\beta} .
\end{split}
\]

Recall that we assumed $v\in\E_1$ and $\left\|\log f\right\|_{C^{\alpha}} \leq 2$ in $B_C(0)$ in the above arguments. In general, for each $\theta_0 \in \mathbb{S}^{n-1}$, we can substitute $v$ with $b v(\T x-X)$ so that these assumptions are satisfied for $b v(\T x-X)$ and its equation, where $b$ is positive, $X$ is a point and $\T$ is the composition of a dilation and a rotation, resulting in the corresponding $u$ being replaced by $b u(b^{-1} \T^{-1} x)+X\cdot x$. Thus,  we obtain the ellipticity of $F_{ij}(\zeta)$ and the estimate \eqref{eq:holder argument} in an appropriate conical domain. Moreover, if $f(x,Du)=f(Du)$ is independent of $x$, then we can take $\beta=\alpha$ instead. 

Therefore, to complete the proof, we need only to show the existence of these transformations $b,X,\T$, and the boundedness of their magnitude.
Recall that $K:=\left\{v=0\right\}=\partial u(0)$. By replacing $u$ with $\left(\det \D\right)^{-\frac{2}{n-q}}  u(\D x)$ for a suitable affine transformation $\D$, we may assume for simplicity that 
\[
B_{1}(\tilde{x}) \subset K  \subset B_{C(n)}(\tilde{x})
\]
holds for some $\tilde{x}$, where the magnitude of the transformation $\D$ depends only on $n$, $\operatorname{diam} (K)$ and $a^{-1}$. After a rotation, we may assume that $\theta_0=-e_n$. Let us denote $x_0$ as the point on $\partial K $ where the outer normal is $-e_n$. By replacing $v$ with $v(x-x_0)$, which is equivalent to replacing $u$ with $u+x_0\cdot x$, we can further assume that $x_0=0$.
Then, we consider the normalization $v_{h}\left(x', x_n\right):=h^{-1} v\left(\D_t' x', t x_n\right)$ with $h:=h(t)=(t \operatorname{det} \D_t')^{\frac{2}{n-q}}$ of $v$ as defined in Definition  \ref{def:normal type1}. 
Since $u$ is strictly convex, there exists $\kappa >0$ such that $V=\left\{ v \leq \kappa \right\}$ is compact. Thus, the normalization $v_h \in \E(\kappa/h)$ is well-defined for small $h$. By noting \eqref{eq:v distance alpha} in Theorems \ref{thm:c1a vK K>0}, for sufficiently small $h$, $v_h$ satisfies our requirements.
\end{proof}

Finally, let us discuss an application of our results to the $L_p$ Minkowski problem, where $p =q+1\in (1,n+1)$. The Minkowski problem addresses the existence, uniqueness, and regularity of closed convex hypersurfaces with prescribed Gaussian curvature (see, e.g., Pogorelov \cite{pogorelov1978minkowski}). The $L_p$ Minkowski problem was introduced by Lutwak \cite{lutwak1993brunnminkowskifirey}, investigating the conditions under which a given measure $\mu$ on $\mathbb{S}^{n}$ can be realized as the $L_p$ surface area measure of a convex body $P \subset \R^{n+1}$ containing $0$. 

Let $p\in(1,n+1)$, which is the range we focus on in this paper. Suppose that the prescribed measure $\mu$ has a bounded positive density $f$ with respect to the spherical Lebesgue measure.
Let $h \in C(\mathbb{S}^{n})$ denote the support function of $P$. Then the Monge-Amp\`ere equation 
corresponding to this $L_p$ Minkowski problem is
\begin{equation}\label{eq:Lp minkowski}
\det\left( \nabla_{\mathbb{S}^{n}}^2 h +h \I_n\right) 
 =h^{p-1}  f
\end{equation}
where a solution $h$ is understood in the Aleksandrov sense. 
If $0 \in \mathring{P}$ (equivalently, $h$ is positive), then the regularity theory developed by Caffarelli \cite{caffarelli1990ilocalization,caffarelli1990interiorw2p,caffarelli1991regularity} are applicable. 
However, if $p\in (1,n+1)$, there exists a positive $C^{\alpha}$ function $f$ such that $0 \in \partial P$; see Hug-Lutwak-Yang-Zhang \cite{hug2005discrete} and Chou-Wang  \cite{chou2006lpminkowski}.

Assume that the convex body $P$ is a solution to the $L_p$ Minkowski problem \eqref{eq:Lp minkowski} with $0 \in \partial P$ and  $f$ being bounded and positive. Then, by \cite{chou2006lpminkowski,caffarelli1990ilocalization}, the convex body $P$ is strictly convex and $C^{1,\alpha}$ outside of $0$. Let us further assume that the boundary of the convex body $P$ includes the graph of a non-negative function $u:U(\subset \R^n) \to \R$ around $0$. Then, $u$ is a generalized solution of 
\begin{equation}\label{eq:sing from Lp}
\det D^2 u=\frac{\tilde{f}}{(u^*)^{p-1} }   + a\delta_0  \quad \text{in }U,
\quad 
u^*=x\cdot Du-u > 0 \quad  \text{outside } 0,
\end{equation}
for some $a \geq 0$, where $\tilde{f}=(1+|Du|^2)^{\frac{n}{2}+p} \Big/ f\left(\frac{(-Du,1)}{\sqrt{1+|Du|^2}}\right)  $ and $\delta_0$ is the Dirac measure at the origin.
It is well-known that the support function $h $ of $P$, when restricted to the hyperplane $x_{n+1}=-1$, yields $v(x)=h(x,-1)=u^*(x)$ and the corresponding obstacle problem
\begin{equation}\label{eq:obs from Lp}
\operatorname{det} D^2 v=\tilde{g} v^{p-1}  \quad \text{in }\R^n,  \quad v \geq 0 \text{ is convex},
\end{equation}
where $\tilde{g}(x)=(1+|x|^2)^{-\frac{n}{2}-p} f\left(\frac{(-x,1)}{\sqrt{1+|x|^2}}\right)$. This is one of our motivations to study \eqref{eq:obs equation}. The support function of the convex set $K:=\left\{v=0\right\}$  is the tangent cone function of $u$ at $0$ and $a=|K| $.
By \cite{chou2006lpminkowski}, both $u$ and $v$ are strictly convex, and thus, smooth outside the zero level set. 
 Consequently,  according to Lemma \ref{lem:ext are exp} and  Lemma \ref{prop:infinite ray}, $\Gamma_{nsc}=\emptyset$. Thus, there are only two possibilities for $K$: either $K=\left\{0\right\}$ or $|K|>0$. 

In the case where $K=\left\{0\right\}$, $u$ is $C^{1,\alpha}$ near $0$.   By using the order-reversing property of the Legendre transform, Theorem \ref{thm:c1a v K=0} implies that
\[
c|x|^{\frac{1+\alpha}{\alpha}} \leq u(x) \leq C|x|^{1+\alpha} 
\quad \text{in }B_{c}(0)
\]
for a universal $\alpha\in (0,1)$, where $c$ and $C$ are positive constants that depend on $u$.  

For the remaining case where  $|K|>0$, $u$ is merely Lipschitz at $0$. Assuming $\alpha \in (0,1)$, $\log f\in C^{\alpha}$,  we then can apply Theorem \ref{thm:c2a regularity u} to deduce that
\[ 
u(\theta,r )= r\phi(\theta)+O(r^{k+1})  
\]
under the polar coordinates $(\theta,r)$, with $\phi \in C^{2,\alpha}$ such that $ \nabla_{\mathbb{S}^{n-1}}^2\phi+\phi\I_{n-1}  $ is bounded and positive definite. Furthermore, under the coordinates $(\theta,s)$, the function $\zeta$ defined by \eqref{eq:def of zeta} is $C^{2,\alpha}$ around $\mathbb{S}^{n-1} \times \left\{0\right\}$, and the matrix $\{F_{ij}(\zeta)\}_{1\leq i,j \leq n}$ defined by \eqref{eq:matrix operators a} is elliptic and $C^{\alpha}$  around $\mathbb{S}^{n-1} \times \left\{0\right\}$.

\small

\bibliography{references}

\smallskip

\noindent T. Jin and X. Tu

\noindent Department of Mathematics, The Hong Kong University of Science and Technology\\
Clear Water Bay, Kowloon, Hong Kong\\[1mm]
Email: \textsf{tianlingjin@ust.hk} (T. Jin) and \textsf{maxstu@ust.hk} (X. Tu).




\medskip

\noindent J. Xiong

\noindent School of Mathematical Sciences, Laboratory of Mathematics and Complex Systems, MOE\\ Beijing Normal University, 
Beijing 100875, China\\[1mm]
Email: \textsf{jx@bnu.edu.cn}

\end{document}